\documentclass[stsy]{arxiv_version}              

\usepackage{url}
\renewcommand\P[1]{\mathbb{P}\left(#1\right)}
\newcommand\E[1]{\mathbb{E}\left[#1\right]}
\newcommand{\bbZ}{\mathbb{Z}}
\newcommand{\bbR}{\mathbb{R}}
\newcommand{\calC}{\mathcal{C}}
\newcommand{\calD}{\mathcal{D}}

\newcommand{\calS}{\mathcal{S}}
\newcommand{\calE}{\mathcal{E}}
\newcommand{\bars}{\bar{s}}
\newcommand{\bbars}{\mathbf{\bar{s}}}

\newcommand\numberthis{\addtocounter{equation}{1}\tag{\theequation}}

\allowdisplaybreaks 
\usepackage[dvipsnames]{xcolor}
\usepackage{tikz}
 \usepackage{multirow}
\usetikzlibrary{shapes}
\usepackage{multirow}
\usetikzlibrary{arrows, backgrounds}
\usetikzlibrary{automata,arrows,positioning,calc, decorations.pathreplacing}
\usepackage{enumitem}
\usepackage{pgfplots}
\usetikzlibrary{arrows.meta}
\usetikzlibrary{patterns}
\usetikzlibrary{shapes.geometric}
\usetikzlibrary{shapes.misc}
\tikzset{cross/.style={cross out, draw=black, minimum size=2*(#1-\pgflinewidth), inner sep=0pt, outer sep=0pt},
cross/.default={1pt}}

\pgfplotsset{compat=newest,
    width=7.5cm,
    height=5.5cm,
    scale only axis=true,
    max space between ticks=25pt,
    try min ticks=5,
    every axis/.style={
        axis y line=left,
        axis x line=bottom,
        axis line style={thick,->,>=latex, shorten >=-.4cm}
    },
    every axis plot/.append style={thick},
    tick style={black, thick}
}
\tikzset{
    semithick/.style={line width=0.8pt},
}
\usepgfplotslibrary{groupplots}
\usepgfplotslibrary{dateplot}

\definecolor{orange}{RGB}{230, 159, 0}
\definecolor{blue1}{RGB}{86, 180, 233}
\definecolor{terquise}{RGB}{0, 158, 115}
\definecolor{red}{RGB}{213, 98, 0}
\definecolor{pink}{RGB}{204, 121, 167}

\usepackage{pifont}

\usepackage{bbm}
\usepackage[font=normalsize]{subcaption}
\usepackage{colortbl}
\usepackage[colorlinks=true,breaklinks=true,bookmarks=true,urlcolor=blue,citecolor=blue,linkcolor=blue,bookmarksopen=false,draft=false]{hyperref}
\usepackage{natbib}
 \NatBibNumeric
 \def\bibfont{\small}%
 \bibpunct[, ]{[}{]}{,}{n}{}{,}%
 \usepackage{cleveref}

\definecolor{orange}{rgb}{0.90196078431, 0.62352941176, 0}
\definecolor{blue1}{rgb}{0.33725490196, 0.70588235294, 0.91372549019}
\definecolor{terquise}{rgb}{0, 0.61960784313, 0.45098039215}
\definecolor{red1}{rgb}{0.83529411764, 0.38431372549, 0}
\definecolor{pink}{rgb}{0.8, 0.47450980392, 0.65490196078}
\definecolor{yellow}{rgb}{0.8, 0.89411764705, 0.25882352941}
\definecolor{brown1}{rgb}{0, 0.21960784313, 0.31372549019}
\TheoremsNumberedThrough     

\EquationsNumberedThrough    

\begin{document}


\RUNAUTHOR{Varma and Maguluri}

\RUNTITLE{Power-of-$d$ Choices Load Balancing}

\TITLE{Power-of-$d$ Choices Load Balancing in the Sub-Halfin-Whitt Regime}

\ARTICLEAUTHORS{%
\AUTHOR{Sushil Mahavir Varma}
\AFF{Industrial and Systems Engineering, Georgia Institute of Technology, \EMAIL{sushil@gatech.edu}, \URL{https://sites.google.com/view/sushil-varma/home}}
\AUTHOR{Francisco Castro}
\AFF{Anderson School of Management, University of California, Los Angeles, \EMAIL{francisco.castro@anderson.ucla.edu}, \URL{https://fcocastro.github.io}}
\AUTHOR{Siva Theja Maguluri}
\AFF{Industrial and Systems Engineering, Georgia Institute of Technology, \EMAIL{siva.theja@gatech.edu}, \URL{https://sites.google.com/site/sivatheja/}}
} 

\ABSTRACT{%
We consider the load balancing system under Poisson arrivals, exponential services, and homogeneous servers. Upon arrival, a job is to be routed to one of the servers, where it is queued until service. We consider the Power-of-$d$ choices routing algorithm, which chooses the queue with minimum length among $d$ randomly sampled queues. We study this system in the many-server heavy-traffic regime where the number of servers goes to infinity simultaneously when the load approaches the capacity. In particular, we consider a sequence of systems with $n$ servers, where the arrival rate of the $n^{\text{th}}$ system is $\lambda=n-n^{1-\gamma}$ for some $\gamma \in (0, 0.5)$, known as the sub-Halfin-Whitt regime. It was shown by [Liu Ying (2020)] that under Power-of-$d$ choices routing with $d \geq n^\gamma \log n$, the queue length behaves similarly to that of JSQ and that there are asymptotically zero queueing delays. 

The focus of this paper is to characterize the behavior when $d$ is below this threshold. We obtain high probability bounds on the queue lengths for various values of $d$ and large enough $n$. In particular, we show that when $d$ grows polynomially in $n$ but slower than in [Liu Ying (2020)], i.e., if $d$ is $\Theta\left((n^\gamma\log n)^{1/m})\right)$ for some integer $m>1$, then the asymptotic queue length is $m$ with high probability. This finite queue length behavior is similar to JSQ in the so-called nondegenerate slowdown regime (where $\gamma=1$). Moreover, if $d$ grows polylog in $n$, i.e., slower than any polynomial, but is at least $\Omega(\log (n)^3)$, the queue length blows up to infinity asymptotically. Such behavior is similar to that under JSQ in the so-called super slowdown regime ($\gamma>1$). We obtain these results by using an iterative state space collapse approach. We first establish a weak state-space collapse (SSC) on the queue lengths. Then, we bootstrap on weak SSC to iteratively narrow down the region of the collapse. After enough steps, this inductive refinement provides the bounds we seek. We establish these sequences of collapse using Lyapunov drift arguments. 
}%



\maketitle
\section{Introduction}


We study a load-balancing queuing system in which a single stream of jobs arrives governed by a Poisson process and is routed to one of the $n$ homogeneous servers, operating with a service rate equal to one. Each server is endowed with a queue of maximum buffer size $b$. 

 The job dispatcher uses a load balancing or routing algorithm to route arriving jobs to the queues. The literature considers many possible routing algorithms ranging from random routing to Joining the Shortest Queue (JSQ). In random routing, a new job joins a queue selected uniformly at random. On the other hand, a new job joins the shortest queue under JSQ. While random routing has no informational requirements---the dispatcher does not need to know any information about the system primitives and state---it does not provide optimal delay performance. In contrast, JSQ has more informational requirements---the dispatcher needs to know the system state to determine the shortest queue---but it has a proven near-optimal delay performance, e.g., see \cite{foschini_ht_jsq_diffusion}. In this paper, we consider an in-between policy known as Power-of-$d$ choices, in which once a job arrives, $d$ queues are sampled uniformly at random from the $n$ queues. Then, the job joins the shortest among the $d$ sampled ones. Note that $d=1$ is the same as random routing, and $d=n$ is the same as JSQ.

For a tractable analysis of the performance of the routing algorithm, the literature considers different asymptotic regimes, where the number of servers goes to infinity, the load on the system approaches its capacity, or both happen simultaneously. As we explain below in \Cref{asymptotic-regimes}, the performance of JSQ has been studied extensively under these regimes and combinations thereof. On the other hand, the performance analysis of Power-of-$d$ choices is comparatively limited. We contribute towards this deficiency by analyzing Power-of-$d$ under the sub-Halfin-Whitt asymptotic regime (see \Cref{asymptotic-regimes}). In this regime, the arrival rate of jobs increases with the number of servers at a rate of $\lambda = n -n^{1-\gamma}$ with $\gamma\in (0,0.5)$. Under this scaling, our goal is to characterize the system's asymptotic delay and steady-state behavior for growing choices, i.e. $d \rightarrow \infty$ as $n \rightarrow \infty$. 

It is known by \cite{sub_halfin_whitt_lei} that if $d$ is sufficiently large $(d \geq n^\gamma \log n)$, Power-of-$d$ behaves like JSQ, and the jobs experience zero asymptotic delays in steady-state. In particular, \cite{sub_halfin_whitt_lei} shows that the asymptotic queue lengths at each server are either zero or one. However, for smaller values of $d$, one expects the delay to be higher. In particular, we will later show that the queue lengths can be finite but greater than one or even asymptotically infinite depending on how $d$ scales with $n$. Thus the asymptotic queue lengths are qualitatively different from JSQ, i.e., they are not just zero-one but exhibit a rich steady-state distribution. Characterizing such behavior under all the scenarios warrants a new approach compared to \cite{sub_halfin_whitt_lei}. In this paper, we aim to provide a unified framework for almost all scales of $d$. Note that these results were first established in \cite{brightwell2012supermarket, brightwell2018supermarket} and our paper provides an independent and alternative proof revealing further insights. We provide a detailed comparison with \cite{brightwell2012supermarket, brightwell2018supermarket} in Section~\ref{sec: related_work}. Before presenting our main contributions, we briefly outline prior work on various asymptotic regimes.
\subsection{Many-Server-Heavy-Traffic Regimes}\label{asymptotic-regimes}
In general, it is challenging to determine the exact delay under a routing policy. So, it has been studied in various asymptotic regimes to gain insights into the optimality of routing policies. We now provide a comprehensive overview of these asymptotic regimes and the insights obtained in the literature for different routing policies. In turn, we explain how our results fit into the literature.

\textbf{Mean field.}  
In this regime, the number of servers increases to infinity while maintaining a constant load on each server. It has been shown in the literature \cite{mitzenmacher_power_of_d, mitzenmacher_thesis, vvedenskaya_power_of_2, lei_power_of_2_sub_halfin_whitt} that under the Power-of-$d$ choices algorithm, even for $d=2$, the steady-state queue lengths exhibits a double exponential tail as opposed to an exponential tail for random routing. On the other hand, under JSQ, it was shown \cite{mukherjee2018universality} that almost all the queues have length zero or one. Thus, all the jobs experience asymptotically zero delay. In addition, it was shown that the same behavior holds true for Power-of-$d$ with growing choices, i.e. $d \rightarrow \infty$, as $n \rightarrow \infty$.

\textbf{Classical heavy-traffic }. Another popular regime considered in the literature is the classical heavy-traffic regime. In this regime, the load converges to the capacity while the number of servers is constant. In particular, let $\epsilon \overset{\Delta}{=} 1-\lambda/n$ be such that $1-\epsilon$ quantifies the load on the system. Then, $\epsilon \downarrow 0$ is the heavy-traffic regime. This regime allows one to analyze the bottlenecks in the system. Under any routing algorithm, the queue lengths in this regime increase to infinity asymptotically. Under JSQ \cite{foschini_ht_jsq_diffusion, atilla_jsq_ht_drift, hurtado_ht_jsq_transform}, an appropriately scaled queue length converges to an exponential distribution with a mean depending on the variance of the arrivals and services. In addition, the limiting behavior of Power-of-$d$ for all $d \geq 2$ is identical to that of JSQ \cite{maguluri2014heavy}, while that of random routing is worse by a factor of $n$.

\textbf{Many-server-heavy-traffic}. One can also consider a hybrid of the mean field and the classical heavy-traffic regime, i.e., many-server-heavy-traffic regime, wherein the load increases to capacity simultaneously while the number of servers increases to infinity. Depending on the relative rate at which the load and the number of servers converge to their asymptotes, one can obtain different viewpoints on the performance of the routing algorithms. In particular, the mean-field and classical heavy traffic are two extreme ways to scale the system and provide different perspectives. For instance, there is a distinction between the performance of JSQ and Power-of-$d$ for small $d$ in the mean-field as opposed to an identical limiting behavior in classical heavy traffic. Considering many-server-heavy-traffic regimes provides us with a more comprehensive understanding of the performance of various load balancing algorithms, allows us to differentiate between their performance, and enables us to pick the right $d$ in Power-of-$d$ type algorithms. Studying such regimes was first initiated by Halfin-Whitt \cite{halfin1981heavy} in an M/M/$n$ queue.

More precisely, the parameterization of the arrival rate as $\lambda = n - \beta n^{1-\gamma}$ for some $\gamma \in (0, \infty)$ is defined as the many-server-heavy-traffic regimes. The parameter $\gamma$ determines the relative rate at which $\lambda$ and $n$ converge to their asymptotes. As $\gamma$ increases from 0 to $\infty$, the load on each server is more prominent, resulting in higher delays. Note that the mean-field regime is a special case with $\gamma=0$ and $\beta<1$, and the classical heavy traffic is interpreted as $\gamma \rightarrow \infty$. 

Now, we discuss the performance of JSQ for $\gamma \in (0, \infty)$ as summarized in Fig. \ref{fig: regimes}. For the sub-Halfin-Whitt regime, i.e., $\gamma \in (0, 0.5)$, similar to the mean-field regime, the delay experienced by the jobs is asymptotically zero. A phase transition occurs in the Halfin-Whitt regime, i.e., $\gamma =0.5$. In this regime \cite{eschenfeldt_halfin_whitt, braverman_halfin_whitt, banerjee_halfin_whitt, banerjee_halfin_whitt_insensitivity}, a vanishing fraction of jobs experiences a constant delay bounded away from zero. Similar results were proved by \cite{super_halfin_whitt_lei, zhisheng_debankur_super_halfin_whitt} for the super-Halfin-Whitt regime, i.e., $\gamma \in (0.5, 1)$. Another phase transition occurs at the nondegenerate slowdown (NDS) regime, i.e., $\gamma=1$. In this case \cite{nds_varun_walton}, incoming jobs experience a non-zero, finite delay. When $\gamma$ increases beyond one, it is called the super slowdown regime, and the limiting queue length at each server increases to infinity. The authors of \cite{hurtado_ht_jsq_transform} analyze the limiting stationary distribution of appropriately scaled queue lengths for $\gamma \geq 2$ and show that its behavior is similar to the classical heavy-traffic regime. The case $\gamma \in (1,2)$ was recently resolved in \cite{raj2024exponential}. \Cref{fig: regimes} (left) summarizes this discussion by illustrating the delay performance of the JSQ policy under different asymptotic regimes.

As the load in the system increases, one expects the delay under any routing algorithm to be higher. Consistent with the intuition, increasing delay with $\gamma$ is observed under JSQ, as previously discussed. On the other hand, one can fix a $\gamma$ and consider the delay performance as $d$ is varied in the Power-of-$d$ choices routing. Similar to how JSQ exhibits higher delay for more loaded regimes, one would expect the delay to increase as $d$ reduces. In this paper, we quantify such behavior exhibited by Power-of-$d$ for all $\gamma \in (0, 0.5)$. We now present our main contributions.

\begin{figure}[bth!]
    \centering
    \FIGURE{
    \begin{tikzpicture}[scale=0.85]
            \fill[blue!10] (0, 0) -- (8, 0) -- (8, -4) -- (0, 0);
            \fill[orange!15] (0, 0) -- (8, -4) -- (8, -8) -- (0, 0);
            \fill[green!25] (0, 0) -- (4, -8) -- (0, -8) -- (0, 0);
            \draw[very thick, blue] (0, 0) -- (8, -4);
            \draw[very thick, orange] (0, 0) -- (8, -8);
            \draw[thick] (0, 0) -- (4, -8);
             \draw[thick, ->] (-0.4, 0) -- (8,0);
            \draw[thick, <-] (0, 1) -- (0, -8);
            \node[rotate=-15] at (4.6, -1.25) (a) {Sub-Halfin-Whitt ($0$ delay) $\gamma \in (0, 0.5)$ \cite{sub_halfin_whitt_lei}};
            \node[rotate=-35] at (3.6, -2.6) (b) {Super-Halfin-Whitt $\gamma \in (0.5, 1)$ \cite{super_halfin_whitt_lei}};
            \node[rotate=-26, blue] at (5.2, -2.4) {Halfin-Whitt $(\gamma=0.5)$ \cite{eschenfeldt_halfin_whitt, braverman_halfin_whitt, banerjee_halfin_whitt, banerjee_halfin_whitt_insensitivity}};
            \node[rotate=-45, RedOrange] at (6, -5.6) {NDS: $\gamma = 1$ (Finite delay) \cite{nds_varun_walton}};
            \node[rotate=-53] at (3, -4.1) (c) {Super Slowdown ($\infty$ delay) $\gamma \in (1,2)$ \cite{raj2024exponential}};
            \node[rotate=-76] at (1.1, -4.6) (d) {Super Slowdown ($\infty$ delay) $\gamma \in [2,\infty)$ \cite{heavy_traffic_daniela}};
            \node at (2.8, 0.3) (e) {Mean Field $(\gamma = 0)$};
            \node[rotate=-90] at (-0.4, -3.8) (f) {Classical Heavy Traffic $(\gamma = \infty)$}; 
            \node[blue] at (8.2, -4.3) (j) {$\epsilon=\frac{1}{\sqrt{n}}$};
            \node[RedOrange] at (8.2, -8.2) (k) {$\epsilon=\frac{1}{n}$};
            \node[OliveGreen] at (4, -8.3) (k) {$\epsilon=1/n^2$};
             \node at (8.5, 0.3) (j) {$\log n$};
             \node at (0.6, 0.9) (k) {$\log\epsilon$};
             \node[rectangle, draw, black, minimum height=15] at (4, -9) {$d=n$ fixed (JSQ)};
        \fill[blue!10] (9.5, -8) -- (17.5, 0) -- (17.5, -3) -- (9.5, -8);
        \fill[orange!15] (9.5, -8) -- (17.5, -3) -- (17.5, -5) -- (9.5, -8);
        \fill[green!25] (9.5, -8) -- (17.5, -5) -- (17.5, -6.5) -- (9.5, -8);
        \draw[thick, ->] (9.3, -8) -- (17.5, -8);
        \draw[thick, ->] (9.5, -8.2) -- (9.5, 0);
        \draw[thick] (9.5, -8) -- (17.5, 0);
        \draw[thick] (9.5, -8) -- (17.5, -3);
        \draw[thick] (9.5, -8) -- (17.5, -5);
        \draw[thick] (9.5, -8) -- (17.5, -6.5);
        \node at (18.1, 0) (a) {JSQ};
        \node at (18.3, -3) (a) {$d=n^\gamma$};
        \draw (17.7, -5) -- (19, -5);
        \draw[thick, ->] (17.8, -5) -- (17.8, -4.2) node at (18.4, -4.4) {Pol};
        \draw[thick, ->] (17.8, -5) -- (17.8, -5.8) node at (18.4, -5.6) {Log};
        \node at (18.6, -6.5) (c) {$d=\log (n)^3$};
        \node[rotate=38] at (13.8, -4.5) (d) {Zero-Delay \cite{sub_halfin_whitt_lei, brightwell2018supermarket, brightwell2012supermarket}, [This Paper]};
        \node[rotate=26] at (14.5, -5.5) (e) {Finite-Delay \cite{brightwell2018supermarket, brightwell2012supermarket}, [This Paper]};
        \node[rotate=17] at (14.7, -6.5) (f) {Infinite-Delay \cite{brightwell2012supermarket} [This Paper]};
        \node[rotate=5] at (14.4, -7.6) (g) {Infinite-Delay [Open]};
        \node at (17.7, -8.3) (h) {$n$ increasing};
        \node at (10.8, 0) (i) {$d$ increasing};
         \node[rectangle, draw, black] at (13.5, -9) {$\gamma \in (0, 0.5)$ fixed};
    \end{tikzpicture}}
    {Performance of JSQ $(d=n)$ under many-server-heavy-traffic regimes $(\gamma \in [0, \infty])$, where $\epsilon = n^{-\gamma}$ (left) and performance of Power-of-$d$ for different choices of $d$ under the sub-Halfin-Whitt regime, i.e. $\gamma \in (0, 0.5)$ (right).
    \label{fig: regimes}}{}
\end{figure}
\subsection{Main Contributions}
Our focus is on understanding the performance of the Power-of-$d$ for different choices of $d$. Note that \Cref{fig: regimes} (left) provides the performance of JSQ, and augmenting it with Power-of-$d$ would correspond to adding a third dimension for $d$ as a function of $n$. The special case of JSQ as depicted in \Cref{fig: regimes} (left) corresponds to one slice of the three dimensional figure with $d=n$. 

In this paper, we restrict ourselves to the sub-Halfin-Whitt regime, i.e. $\gamma \in (0, 0.5)$, and consider a broad range of values of $d$. It was shown in \cite{sub_halfin_whitt_lei} that Power-of-$d$ with $d \geq n^\gamma \log n$ has an identical limiting behavior as JSQ. We go beyond this range and provide a quantitative distinction between JSQ and Power-of-$d$ by characterizing the performance of the Power-of-$d$ for $d<n^\gamma$. Our results are almost the same as in \cite{brightwell2012supermarket, brightwell2018supermarket}, but we prove them using a different approach. A detailed comparison with these works can be found in Section~\ref{sec: related_work}. A summary of the results in this paper is given in \Cref{fig: regimes} (right).

\textbf{Finite Delay:} First, we consider the case when $d = (n^{\gamma} \log n)^{1/m}$ for some positive integer $m$.
We show that the queue lengths exhibits the following behavior with high probability: most of the queues are of length $m$ and a vanishing fraction are either longer or shorter. In particular, we show that the fraction of queues with length less than $i$ is equal to $n^{-\gamma} d^{i-1}(1+o(1))$ for $i \leq m$ and the fraction of queues with length more than $m$ is at most $o(n^{-\gamma} d^{m-1})$ which is $o(1)$.
It is worth noting that these results are applicable for the pre-limit system as well, i.e. for all finite, large enough $n$ (and we provide explicit expressions for all the $o(\cdot)$ terms). These results imply that when $m \geq 2$, the queue lengths are non-zero but finite, behaving qualitatively similar to that of JSQ in NDS regime. 
However, a fundamental difference in behavior is that while our results show that the queue lengths are essentially concentrated around $m$ for Power-of-$d$ in sub-Halfin-Whitt regime, the limiting queue lengths 
of JSQ in NDS are spread over multiple values and the distribution has a nontrivial support.
Also note that, when we pick $m=1$, our result implies that the jobs experience zero asymptotic delay and the queue lengths are either zero or one. The result in this special case was first 
established in \cite{sub_halfin_whitt_lei}.

\textbf{Infinite Asymptotic Delay:} Now, we consider the case when $d$ is Poly-Log$(n)$ but is at least  $\Omega\left(\log (n)^3\right)$. Note that, Poly-Log$(n)$ is smaller than $d = (n^{\gamma} \log n)^{1/m}$ for any $m \in \bbZ_+$. We show that all the queue lengths are $\Theta(\log n/\log d)$ with high probability. This implies that the asymptotic queue lengths are infinite. Similar to the finite delay case, we characterize the fraction of queue lengths smaller or larger than $m$ for the pre-limit system. Note that, such a behavior is qualitatively similar to that of JSQ in the super slowdown regime. However, there is again a fundamental difference in behavior because while we show that the queue lengths concentrate around $\Theta(\log n/\log d)$ for Power-of-$d$, JSQ in the super slowdown regime has a large support.
Extending the result to the case when $d < \log (n)^3$ is an open future research direction.

\begin{table}[bth!]
    \TABLE{Power-of-$d$ choices for $\gamma \in (0, 0.5)$ \label{tab: main_result}}
    {\begin{tabular}{|c|c|c|c|}
        \hline
       Value of $d$  & Regime &  Queue Length & References \\
       \hline
       $d \geq n^\gamma \log n$  & Zero-Delay & $\approx 1$ &  This paper and \cite{sub_halfin_whitt_lei, brightwell2018supermarket, brightwell2012supermarket}\\
       Polynomial $(d=(n^{\gamma}\log n)^{1/m})$ & Finite-Delay & $\approx m = \Theta\left(\frac{\log n}{\log d}\right)$  & This paper and \cite{brightwell2018supermarket, brightwell2012supermarket} \\
       Poly-log and $d \geq \log (n)^3$  & Infinite-Delay & $\Theta\left(\frac{\log n}{\log d}\right)$ &  This paper and \cite{brightwell2012supermarket}\\
       $d \leq \log (n)^3$ & Infinite-Delay & $\Theta\left(\frac{\log n}{\log d}\right)$ & Open \\
       \hline
    \end{tabular}}{}
\end{table}

\textbf{Methodological Contribution:} In contrast to the prior work on load balancing that is based on fluid and diffusion limits (e.g., see: \cite{eschenfeldt_halfin_whitt, banerjee_halfin_whitt}), Stein's method (e.g. see: \cite{sub_halfin_whitt_lei, lei_power_of_2_sub_halfin_whitt}), transform method (e.g. see: \cite{hurtado_ht_jsq_transform}), and a combination of iterative SSC and Stein's method \cite{lei_coxian_2_sub_halfin_whitt, lei_coxian_k_sub_halfin_whitt}, our approach uses iterative SSC alone without the use of Stein's method. We first obtain a crude bound on the possible values of the queue lengths, i.e. a weak state space collapse. We then iteratively bootstrap from this weak SSC to obtain more and more refined SSC. This iterative refinement is inductively repeated ($m$ times) until a tight characterization of the steady-state queue lengths as described above is obtained.
Lyapunov drift-based arguments achieve each step of the refinement. 

Iterative SSC was used as an intermediate step in characterizing the limiting distribution of queue lengths in \cite{lei_coxian_2_sub_halfin_whitt, lei_coxian_k_sub_halfin_whitt} to study the case of $m=1$ with Coxian service times. Using their SSC methodology directly in our setting does not suffice as it results in only a crude bound (see Section \ref{sec: special_case_upper_bound} for detailed discussion), so further refinement is required to obtain tight queue length bounds. The novelty of our approach lies in independently constructing a sequence of Lyapunov functions which allows us to obtain tight queue length bounds simply by applying iterative SSC enough times. The main takeaway of our methodology is that iterative SSC is a powerful tool to analyze queueing systems in mean field types of regimes, i.e., the queueing system concentrates around the fixed point of the corresponding deterministic, dynamical system.


\subsection{Related Work} \label{sec: related_work}
The prior work that is closest to ours are \cite{amarjit_power_of_d_sub_halfin_whitt, sub_halfin_whitt_lei, brightwell2018supermarket}. In \cite{amarjit_power_of_d_sub_halfin_whitt}, the analysis for Power-of-$d$ was carried out for the `finite delay' regime, i.e. $d = (n^{\gamma} \log n)^{1/m}$ for $m \in \bbZ_+$, and a process level law of large numbers is established to show convergence of the queue length process to its mean. By observing the mean, it was noted that most of the queue lengths are $m$ in the limit. However, the lower order terms, i.e. the fraction of queues with length larger or smaller than $m$ is not characterized in this result. Although \cite{amarjit_power_of_d_sub_halfin_whitt} also provides a diffusion process that characterizes further fluctuations around the mean, the steady-state distribution of the diffusion process is not characterized. In addition, to conclude that these results holds for the steady-state of the pre-limit process, interchange of limits is required which is not established. If these two steps were completed, then the approach in \cite{amarjit_power_of_d_sub_halfin_whitt} would obtain the lower order terms. In contrast, by directly working with the steady-state quantities (as opposed to process level convergence), we  characterize the dominant lower order term and show that it is exactly $n^{-\gamma} d^{i-1}$ thereby obtaining a sharper characterization of the steady-state queue lengths distribution. Moreover, we also obtain bounds on the fluctuations around these lower-order terms. 

In addition, \cite{sub_halfin_whitt_lei} is also closely related to our result which shows that the asymptotic delay experienced by the jobs is zero under Power-of-$d$ for $d \geq n^\gamma \log n$. A key difference is that \cite{sub_halfin_whitt_lei}  characterizes all the moments of the total number of jobs in the system, whereas we present high probability tail bounds on the queue length distribution. It is worth noting that both of these bounds imply zero waiting probability in the steady state when $m=1$ (see Appendix~\ref{app: zero_waiting}). Another difference is in the methodology.
Note that the iterative SSC is a natural and powerful technical framework to establish that a stochastic system concentrates around a fixed point. As the deterministic, dynamical system is a good approximation of the stochastic behavior of the power-of-$d$ load balancing in the sub-Halfin-Whitt regime, we are able to construct a sequence of Lyapunov functions to establish tight upper and lower bounds on the queue lengths. Such an approach does not require the use of Stein's method combined with a weak SSC as in \cite{sub_halfin_whitt_lei}. We refer the reader to Section~\ref{sec: special_case_upper_bound} for a more detailed comparison between the proof of \cite{sub_halfin_whitt_lei} and ours for the special case of $m=1$.


The main result of \cite{brightwell2018supermarket} and its previous (arXiv) version \cite{brightwell2012supermarket} are closely related to our result. In addition to a few minor technical differences, the papers differ on the proof methodology. In particular, \cite{brightwell2018supermarket} considers the case of $m \notin \bbZ_+$, whereas we consider $m \in \bbZ_+$ along with a logarithmic function for our choice of $d$. Also, while \cite{brightwell2018supermarket} considers a wider range of values of $\gamma$, we consider a broader range of values of $d$, that is, we allow $d$ to be smaller than any polynomial (poly-log). Note that \cite{brightwell2012supermarket} considers the same range of $\gamma$ and $d$ as ours. 
Nonetheless, our proof methodology is different and reveals further insights. In particular, \cite{brightwell2018supermarket, brightwell2012supermarket} analyzes the drift of a sequence of Lyapunov functions in a finite time. As the analysis is carried out in finite time, \cite{brightwell2018supermarket, brightwell2012supermarket} have to ensure that the Lyapunov function stays small while the other Lyapunov functions in the sequence decrease. Such a complication is circumvented in our methodology as we take a steady-state approach. Also, our sequence of Lyapunov functions is completely different from that of \cite{brightwell2018supermarket, brightwell2012supermarket} providing an alternative proof and revealing geometric insights. In particular, while \cite{brightwell2018supermarket, brightwell2012supermarket} provides an elegant algebraic construction of the sequence of Lyapunov functions, our sequence is based on geometric intuition via the trajectory of the ODE approximation.


Now, we present a non-exhaustive overview of the literature on load balancing under the sub-Halfin-Whitt regime. JSQ and Power-of-$d$ for $d \geq n^\gamma \log n$ was analyzed in \cite{sub_halfin_whitt_lei}. This result was extended for more general settings in the literature: coxian-2 service distribution in \cite{lei_coxian_2_sub_halfin_whitt}, coxian-$k$ service distribution in \cite{lei_coxian_k_sub_halfin_whitt}, and parallel jobs arriving in the system in \cite{weina_sub_hafin_whitt_parallel_jobs}. In addition, Power-of-$d$ choices has been analyzed in \cite{lei_power_of_2_sub_halfin_whitt, Eschenfeldt_Gamarnik_sub_halfin_whitt_power_of_2} for $d=2$, and \cite{amarjit_power_of_d_sub_halfin_whitt} provides a transient analysis for growing sequences of $d$. The reader can refer to the survey paper \cite{debankur_survey} for a holistic review of the literature.

Lastly, we would like to point out that our results have a similar qualitative flavor as in \cite{jonckheere2018asymptotics}. In particular, for a given $m \in \bbZ_+$, \cite{jonckheere2018asymptotics} characterizes the critical load below which the blocking probability is very small for a load balancing system with finite buffer equal to $m$; under a state-dependent random routing policy.



\subsection{Notation}
The set of all positive integers (excluding zero) is denoted by $\bbZ_+$. For some $k \in \bbZ_+$, the set of numbers $\{1,2, \hdots, k\}$ is denoted by $[k]$. We use the shorthand $w.h.p.$ to denote ``with high probability''.

\section{Model}
Consider a load balancing system with $n$ homogeneous servers. A single stream of jobs arrive, governed by a Poisson process with rate $\lambda < n$. Upon arrival, the job is routed to one of the servers, where it waits in a queue before getting served. Each queue has a maximum buffer size $b \in \bbZ_+$. Preemption is not allowed and the job cannot move within queues. The service times for all servers are i.i.d. exponential random variables with rate $\mu=1$. An illustration of the model is given in Fig. \ref{fig: model}.
\begin{figure}[bth!]
    \FIGURE{
    \begin{tikzpicture}[scale=0.4]
        \draw[black, thick] (0, 0) -- (5, 0) -- (5, -2) -- (0, -2);
        \draw[black, thick] (6.1, -1) circle (1) node at (6.1, -1) {1};
        \draw[black, thick] (0, -3) -- (5, -3) -- (5, -5) -- (0, -5);
        \draw[black, thick] (6.1, -4) circle (1) node at (6.1, -4) {2};
        \draw[fill=black] (6.1, -6) circle (0.1);
        \draw[fill=black] (6.1, -6.5) circle (0.1);
        \draw[fill=black] (6.1, -7) circle (0.1);
        \draw[black, thick] (0, -8) -- (5, -8) -- (5, -10) -- (0, -10);
        \draw[black, thick] (6.1, -9) circle (1) node at (6.1, -9) {$n$};
        \node (rect) at (-5, -5) [draw,thick,minimum width=2cm,minimum height=1cm] {Load Balancer};
        \draw[black, thick, ->] (-10, -5) -- (rect.west) node at (-9.2, -4.4) {$\lambda$};
        \draw[black, thick, ->] (rect.east) -- (0, -1);
        \draw[black, thick, ->] (rect.east) -- (0, -4);
        \draw[black, thick, ->] (rect.east) -- (0, -9);
        \draw[black, thick, ->] (7.1, -1) -- (8.6, -1) node[anchor=south] {$\mu=1$};
        \draw[black, thick, ->] (7.1, -4) -- (8.6, -4) node[anchor=south] {$\mu=1$};
        \draw[black, thick, ->] (7.1, -9) -- (8.6, -9) node[anchor=south] {$\mu=1$};
    \end{tikzpicture}}{
    A homogeneous load balancing model with $n$-servers.
    \label{fig: model}}{}
\end{figure}
A natural state descriptor for the system is the number of jobs in each queue. However, 
it is mathematically more convenient to consider $\BFs \in (\bbZ_+ \cup \{0\})^b$ as the state descriptor. Here, $s_i$ is the number of queues with length at least $i$, and $b$ is the maximum buffer size. The state space is given by 
\begin{align*}
    \calS := \left\{\BFs \in (\bbZ_+ \cup \{0\})^b : s_{i_1} \leq s_{i_2} \leq n \quad \forall i_1 \geq i_2 \in [b] \right\}.
\end{align*}
In addition, we also denote the number of queues with at least 0 jobs by $s_0 = n$. Once a job arrives, $d$ queues are sampled uniformly at random, with replacement from $n$ queues. Then, the job is routed to the smallest among the $d$ sampled queues. This algorithm is known as Power-of-$d$ choices in the literature. Under this routing scheme, the process $\{\BFs(t): t \geq 0\}$ is a finite state-space, irreducible, continuous time Markov chain. Thus, the CTMC $\{\BFs(t): t \geq 0\}$ is positive recurrent and exhibits a unique stationary distribution. Denote by $\bbars$ a random variable with the same distribution as the stationary distribution of the CTMC.

As the exact analysis is challenging, we consider a many-server-heavy-traffic asymptotic regime, wherein the number of servers are scaled to infinity $(n \rightarrow \infty)$ and the arrival rate increases to the capacity $(\lambda/n \rightarrow 1)$. In particular, consider a sequence of load balancing systems parameterized by $n$. The arrival rate for the $n^{\textit{th}}$ system is given by $\lambda = n - n^{1-\gamma}$. In this paper, we are interested in the case of $\gamma \in (0, 0.5)$, known as the sub-Halfin-Whitt regime. In addition, our focus is on growing choices in Power-of-$d$, i.e. $d \rightarrow \infty$ as $n \rightarrow \infty$. The goal is to characterize the limiting steady-state distribution $\bbars^{(n)}$ as $n \rightarrow \infty$. In the rest of the paper, we suppress the dependence of $\bbars$ on $n$ for notational convenience.

\section{Results and Insights}
In this section, we develop intuition by considering an ODE approximation of the load-balancing system. This will guide the limiting behavior of the stochastic model. In particular, we expect the limiting stationary distribution of the stochastic model to concentrate around the fixed point of the ODE approximation. Note that, we do not directly work with the ODE to prove the result for the stochastic model. We only leverage intuition from the ODE approximation. To prove such a result, we find the region where the steady-state stochastic system resides with high-probability by iteratively narrowing down the possible regions of the state-space \cite{lei_coxian_2_sub_halfin_whitt, lei_coxian_k_sub_halfin_whitt}.
\subsection{Intuition: ODE approximation} \label{sec: mean_field_model}
To simplify the arguments in this section, consider $b=\infty$, i.e., the queues have infinite buffer capacity. Now, the evolution of $n^{\textit{th}}$ system can be approximated by abstracting out the stochasticity to obtain the following ODE:
\begin{align*}
    \frac{ds_i}{dt} = \lambda\left(\left(\frac{s_{i-1}}{n}\right)^d - \left(\frac{s_i}{n}\right)^d\right) - \left(s_i - s_{i+1}\right) \quad \forall i \in \bbZ_+. \numberthis \label{eq: fluid_model_diff_eqn}
\end{align*}
The rate of change of $s_i$ is the difference of the rate at which it increases and decreases. The first term on the RHS is the product of arrival rate of the customers $(\lambda)$ and the probability that the incoming customer will join a queue with length equal to $i-1$. This is equal to the rate at which $s_i$ is increasing. The rest of the terms is the product of the service rate $(\mu=1)$ and the number of queues $(s_i - s_{i+1})$ with queue length equal to $i$. This is equal to the rate at which $s_i$ is decreasing.

To obtain the fixed point of the dynamical system, substitute $ds_i/dt=0$ for all $i \in \bbZ_+$ resulting in a set of non-linear equations. After solving these equations, one obtains the following solution:
\begin{align*}
    \frac{s_i}{n} = \left(\frac{\lambda}{n}\right)^{\frac{d^i-1}{d-1}} \quad \forall i \in \bbZ_+. \numberthis \label{eq: og_fixed_point}
\end{align*}
The above suggests a candidate stationary distribution of the stochastic system. In \cite{mitzenmacher_power_of_d, vvedenskaya_power_of_2}, this intuition is made formal in  the mean field regime, i.e.,  for $\lambda = (1-\beta)n$ with $\beta \in (0,1)$, and $d=2$. Now, we consider the regime where $\lambda=n-n^{1-\gamma}$ for $\gamma \in (0, 0.5)$ and the routing is governed by Power-of-$d$ with $d \rightarrow \infty$ as $n \rightarrow \infty$. Then, \eqref{eq: og_fixed_point} can be approximated as follows:
\begin{align*}
    s_i &= n\left(1 - n^{-\gamma}\right)^{\frac{d^i-1}{d-1}} \overset{(a)}{\approx} n\left(1 - n^{-\gamma}\right)^{d^{i-1}} \overset{(b)}{\approx} \begin{cases}
    n - n^{1-\gamma}d^{i-1} \quad &\forall i \in [b] : n^{1-\gamma}d^{i-1} = o(n) \\
    o(n) \quad &\textit{otherwise}.
    \end{cases} \numberthis \label{eq: fixed_point_before_d}
\end{align*}
where $(a)$ follows by approximating $d^i-1 \approx d^i$ and $d-1 \approx d$ as $d$ scales to infinity. Next, $(b)$ follows by Taylor's series expansion up to the first order term. Now, define $m \in \bbZ_+$ to be the smallest integer such that the number of queues with length at least $m+1$ is $o(n)$, i.e. $s_{m+1}=o(n)$ and $s_m = \Theta(n)$. Then, we must have $n^{1-\gamma}d^m \approx n$ implying that $d \approx n^{\gamma / m}$. More precisely, our choice of $d$ is such that it satisfies $d = (2m n^\gamma)^{1/ m}\log (d)^{1/m}$ which is approximately equivalent to $d \approx n^{\gamma / m}$ when the lower order terms are ignored. For this value of $d$, from \eqref{eq: fixed_point_before_d}, we get
\begin{align*}
    s_i \approx \begin{cases}
    n - \frac{2 m n\log d}{d^{m-i+1}} \quad &\forall i \in [m] \\
    o(n) &\textit{otherwise}.
    \end{cases} \numberthis \label{eq: fixed_point_fluid_model}
\end{align*}
Fig. \ref{fig: fixed_point} illustrates the fixed point in terms of queue occupancy. We expect the stationary distribution of the stochastic model to concentrate around the above fixed point. The main contribution of our paper is to prove that the heuristic argument is indeed correct. We state the formal result in the next sub-section.
\begin{figure}[bth!]
    \FIGURE{
    \begin{tikzpicture}[scale=0.5]
        \draw[black, thick] (0, 0) circle (0.75);
        \draw[black, thick] (2, 0) circle (0.75);
        \draw[black, thick] (4, 0) circle (0.75);
        \draw[black, thick] (16, 0) circle (0.75);
        \draw[black, thick] (8, 0) circle (0.75);
        \draw[black, thick] (10, 0) circle (0.75);
        \draw[black, thick] (12, 0) circle (0.75);
        \draw[black, thick] (14, 0) circle (0.75);
        \draw[black, thick] (18, 0) circle (0.75);
        \draw[fill=black] (5.5, 0) circle (0.1);
        \draw[fill=black] (6, 0) circle (0.1);
        \draw[fill=black] (6.5, 0) circle (0.1);
        \draw[black, thick] (-0.75, 12) -- (-0.75, 1.25) -- (0.75, 1.25) -- (0.75, 12);
        \draw[black, thick] (1.25, 12) -- (1.25, 1.25) -- (2.75, 1.25) -- (2.75, 12);
        \draw[black, thick] (3.25, 12) -- (3.25, 1.25) -- (4.75, 1.25) -- (4.75, 12);
        \draw[black, thick] (15.25, 12) -- (15.25, 1.25) -- (16.75, 1.25) -- (16.75, 12);
        \draw[black, thick] (7.25, 12) -- (7.25, 1.25) -- (8.75, 1.25) -- (8.75, 12);
        \draw[black, thick] (9.25, 12) -- (9.25, 1.25) -- (10.75, 1.25) -- (10.75, 12);
        \draw[black, thick] (11.25, 12) -- (11.25, 1.25) -- (12.75, 1.25) -- (12.75, 12);
        \draw[black, thick] (13.25, 12) -- (13.25, 1.25) -- (14.75, 1.25) -- (14.75, 12);
        \draw[black, thick] (17.25, 12) -- (17.25, 1.25) -- (18.75, 1.25) -- (18.75, 12);
        \draw[fill=brown] (-0.5, 1.75) rectangle (0.5, 2.75);
        \draw[fill=brown] (-0.5, 3.25) rectangle (0.5, 4.25);
        \draw[fill=brown] (-0.5, 4.75) rectangle (0.5, 5.75);
        \draw[brown, very thick, dashed] (-0.48, 6.25) rectangle (0.48, 7.25);
        \draw[fill=brown] (-0.5, 7.75) rectangle (0.5, 8.75);
        \draw[fill=brown] (-0.5, 9.25) rectangle (0.5, 10.25);
        \draw[fill=brown] (-0.5, 10.75) rectangle (0.5, 11.75);
        \draw[fill=brown] (1.5, 1.75) rectangle (2.5, 2.75);
        \draw[fill=brown] (1.5, 3.25) rectangle (2.5, 4.25);
        \draw[fill=brown] (1.5, 4.75) rectangle (2.5, 5.75);
        \draw[brown, very thick, dashed] (1.48, 6.25) rectangle (2.48, 7.25);
        \draw[fill=brown] (1.5, 7.75) rectangle (2.5, 8.75);
        \draw[fill=brown] (1.5, 9.25) rectangle (2.5, 10.25);
        \draw[fill=brown] (1.5, 10.75) rectangle (2.5, 11.75);
        \draw[fill=brown] (3.5, 1.75) rectangle (4.5, 2.75);
        \draw[fill=brown] (3.5, 3.25) rectangle (4.5, 4.25);
        \draw[fill=brown] (3.5, 4.75) rectangle (4.5, 5.75);
        \draw[brown, very thick, dashed] (3.48, 6.25) rectangle (4.48, 7.25);
        \draw[fill=brown] (3.5, 7.75) rectangle (4.5, 8.75);
        \draw[fill=brown] (3.5, 9.25) rectangle (4.5, 10.25);
        \draw[black, thick] (5.25, 12) -- (5.25, 1.25) -- (6.75, 1.25) -- (6.75, 12);
        \draw[brown, very thick, dashed] (5.5, 1.75) rectangle (6.5, 2.75);
        \draw[brown, very thick, dashed] (5.5, 3.25) rectangle (6.5, 4.25);
        \draw[brown, very thick, dashed] (5.5, 4.75) rectangle (6.5, 5.75);
        \draw[brown, very thick, dashed] (5.48, 6.25) rectangle (6.48, 7.25);
        \draw[brown, very thick, dashed] (5.5, 7.75) rectangle (6.5, 8.75);
        \draw[brown, very thick, dashed] (5.5, 9.25) rectangle (6.5, 10.25);
        \draw[fill=brown] (15.5, 1.75) rectangle (16.5, 2.75);
        \draw[fill=brown] (15.5, 3.25) rectangle (16.5, 4.25);
        \draw[fill=brown] (7.5, 1.75) rectangle (8.5, 2.75);
        \draw[fill=brown] (7.5, 3.25) rectangle (8.5, 4.25);
        \draw[fill=brown] (7.5, 4.75) rectangle (8.5, 5.75);
        \draw[brown, very thick, dashed] (7.48, 6.25) rectangle (8.48, 7.25);
        \draw[fill=brown] (7.5, 7.75) rectangle (8.5, 8.75);
        \draw[fill=brown] (7.5, 9.25) rectangle (8.5, 10.25);
        \draw[fill=brown] (9.5, 1.75) rectangle (10.5, 2.75);
        \draw[fill=brown] (9.5, 3.25) rectangle (10.5, 4.25);
        \draw[fill=brown] (9.5, 4.75) rectangle (10.5, 5.75);
        \draw[brown, very thick, dashed] (9.48, 6.25) rectangle (10.48, 7.25);
        \draw[fill=brown] (11.5, 1.75) rectangle (12.5, 2.75);
        \draw[fill=brown] (11.5, 3.25) rectangle (12.5, 4.25);
        \draw[fill=brown] (11.5, 4.75) rectangle (12.5, 5.75);
        \draw[fill=brown] (13.5, 1.75) rectangle (14.5, 2.75);
        \draw[fill=brown] (13.5, 3.25) rectangle (14.5, 4.25);
        \draw[fill=brown] (17.5, 1.75) rectangle (18.5, 2.75);
        \node[align=left] at (22.45, 2.25) {$s_1 \approx n-\frac{2mn\log d}{d^m}=\lambda$};
        \draw[red, thin] (-1.25, 1.5) rectangle (26, 3);
        \node[align=left] at (21.7, 3.75) {$s_2\approx n-\frac{2mn\log d}{d^{m-1}}$};
        \draw[red, thin] (-1.25, 3) rectangle (26, 4.5);
        \node[align=left] at (21.7, 5.25) {$s_3\approx n-\frac{2mn\log d}{d^{m-2}}$};
        \draw[red, thin] (-1.25, 4.5) rectangle (26, 6);
        \node[align=left] at (22.15, 8.25) {$s_{m-1}\approx n-\frac{2mn\log d}{d^{2}}$};
        \draw[red, thin] (-1.25, 7.5) rectangle (26, 9);
        \node[align=left] at (21.8, 9.75) {$s_{m}\approx n-\frac{2mn\log d}{d}$};
        \draw[red, thin] (-1.25, 9) rectangle (26, 10.5);
        \node[align=left] at (21, 11.25) {$s_{m+1}=o(n)$};
        \draw[red, thin] (-1.25, 10.5) rectangle (26, 12);
    \end{tikzpicture}}
   {Illustration of the fixed point of the ODE approximation in terms of queue occupancy.
    \label{fig: fixed_point}}{}
\end{figure}
\subsection{Main Result} \label{sec: main_result}
Now, we present the main results of the paper below.
\begin{theorem} \label{theo: informal}
Let $\left\{m_n \in \bbZ_+ : n \in \bbZ_+\right\}$ be a sequence such that either $m_n \equiv m \in \bbZ_+$ or $m_n \rightarrow \infty$. Consider a load balancing model operating under Power-of-$\lfloor d \rfloor$ routing algorithm with $d = (2m_n n^\gamma)^{1 /m_n} \log (d)^{1/m_n}$. If further $d = \Omega(\log (n)^3)$ and $b = O(\log (n)^3)$, then with probability at least $1-\left(\frac{1}{n}\right)^{(m_n\log n)/9}$, for large enough $n$, we have
\begin{align*}
    \bars_i =\begin{cases} n - \frac{2m_n n\log d}{d^{m_n-i+1}} +o\left(\frac{2m_n n\log d}{d^{m_n-i+1}}\right) \quad &\forall i \in [m_n] \\
    o(n) &\textit{for } i = m_n+1\\
    o(1) &\textit{otherwise}.
    \end{cases}
\end{align*}
\end{theorem} 
The result in Theorem \ref{theo: informal} is obtained by proving a high probability lower bound and a high probability upper bound separately on $\BFs$ in Theorem \ref{theo: lower_bound} and Theorem \ref{theo: upper_bound}, respectively. These are presented at the end of this section, wherein, we also explicitly characterize the $o(\cdot)$ terms.

Note that the above theorem considers $d$ as a solution of an implicit equation. However, 
one can obtain upper and lower bounds on $d$, matching up to a logarithmic term. In particular, we have $(2mn^\gamma)^{1 /m} \leq d \leq (2mn^\gamma)^{1 /m} \log (n)^{1/m}$. 
The above result shows that the stationary distribution of the stochastic model concentrates around the fixed point of the ODE approximation given by \eqref{eq: fixed_point_fluid_model}. To further understand the result, consider the limit as $n \rightarrow \infty$ to get
\begin{align*}
    \lim_{n \rightarrow \infty} \frac{\bars_i}{n} = \begin{cases}
    1 \quad & \forall i \in [m] \\
    0 &\textit{otherwise}.
    \end{cases}
\end{align*}
Thus, most of the queues have lengths equal to $m$ which implies that an incoming customer joins a queue with length $m-1$ with high probability. 
In particular, we have
\begin{align*}
    m = \gamma \frac{\log n}{\log d} + \frac{\log (2m\log d)}{\log d} = \gamma \frac{\log n}{\log d}\left(1 + o(1)\right) \leq \frac{\log n}{\log d} ,\numberthis \label{eq: m}
\end{align*}
where the last inequality holds for $n$ large enough. So, if $d$ is a polynomial in $n$, then $m$ is finite. On the other hand, if $d$ is smaller than any polynomial, then $m$ will increase to infinity as $n \rightarrow 
\infty$. We discuss such a phase transition in the queue length behavior below.
\subsubsection{Phase Transitions}
The limiting steady-state performance in this regime exhibits phase transitions as the values of different parameters like $\gamma$ and $d$ are varied. We outline two such phase transitions below.

\textbf{Phase Transition as $d$ varies:} It was shown in \cite{sub_halfin_whitt_lei} that if $d \geq n^{\gamma} \log n$, the delay experienced by the customers is asymptotically zero. This is consistent with our result as $d=n^{\gamma} \log n$ corresponds to the case of $m=1$ implying asymptotically zero-delay. Now, if the value of $d$ decreases beyond $n^\gamma$, one would expect the steady-state queue length distribution to be higher. According to Theorem \ref{theo: informal}, if $d$ is of the form $d = n^{\gamma /m}$ for some $m \geq 2$, the steady-state queue lengths are equal to $m$. In particular, if $d$ is a polynomial less than $n^{\gamma}$, then the waiting times are non-zero but finite. As discussed before, such a qualitative behavior is similar to that of JSQ in the NDS regime. Now, consider the case when $d$ is smaller than any polynomial, for instance, Poly-Log$(n)$. This corresponds to $d \approx n^{\gamma/m}$ with $m \approx \gamma \log n/\log d$ which increases to infinity as $n \rightarrow \infty$. In this case, the queue lengths are asymptotically infinite, which corresponds to infinite delay. As discussed before, such behavior is qualitatively similar to that of JSQ in the super slowdown regime.

To summarize, we characterize the limiting steady-state behavior of the stochastic model in Theorem \ref{theo: informal}. Based on this, we show that different orders of delays can emerge depending on the choice of $d$. In particular, zero-delay for $d \geq n^\gamma \log n$, finite-delay for polynomial $d(n)$, and infinite delay for poly-log $d(n)$. Such a phase transition is reminiscent of the phase transition of JSQ as the load increases from the mean field to the classical heavy traffic regime. To summarize, we characterize the limiting steady-state behavior of the stochastic model in Theorem \ref{theo: informal}. Based on this, we show that different orders of delays can emerge depending on the choice of $d$. In particular, zero-delay for $d \geq n^\gamma \log n$, finite-delay for polynomial $d(n)$, and infinite delay for poly-log $d(n)$. Such a phase transition is reminiscent of the phase transition of JSQ as the load increases from the mean field to the classical heavy traffic regime.

\textbf{Phase Transition as $\gamma$ varies:} It is well known \cite{vvedenskaya_power_of_2, mitzenmacher_power_of_d} that the queue length under the power of $d$ choices for $d \geq 2$ is approximately $\Theta(\log \log n/\log d)$ for the mean field regime, i.e. $\gamma = 0$. On the other hand, we show that the queue lengths are $m = \Theta(\gamma \log n/\log d)$ when $\gamma > 0$. In particular, we observe a phase transition in the queue length behavior as the value of $\gamma$ moves from $0$ to $>0$. In other words, as the value of $\gamma$ increases, the arrival rate increases resulting in a higher load on the system, which results in larger queue lengths.

Another phase transition is observed at $\gamma =0.5$. In particular, we show that $\frac{n-s_1}{n^{1-\gamma}} \overset{P}{\rightarrow} 0$ as $n \rightarrow \infty$ for $\gamma \in (0, 0.5)$. On the contrary, when $\gamma = 0.5$, \cite{eschenfeldt_halfin_whitt} shows that $(\frac{n-s_1}{\sqrt{n}}, \frac{s_2}{\sqrt{n}})$ converges weakly to a two-dimensional OU process under the JSQ policy, which is similar to $m=1$ for our case. Thus, our proof technique fails precisely at $\gamma=0.5$. In particular, iterative SSC is only capable of providing high probability upper and lower bounds on a stochastic process. These upper and lower bounds would only match if the appropriately scaled and centered stochastic process converges to a Dirac-delta as is the case with $\frac{n-s_1}{n^{1-\gamma}}$ when $\gamma \in (0, 0.5)$. However, for $\gamma = 0.5$ and $m = 1$, $\frac{n-s_1}{\sqrt{n}}$ exhibits a non-degenerate distribution in the limit, which cannot be accounted for using simply iterative SSC. As the proof is quite involved with several variables $(d, m)$ dependent on $n$, we simply impose the restriction of $\gamma < 0.5$. For $m>1$, we refer the readers to \cite{brightwell2018supermarket} that considers a broader range of $\gamma \in \left(0, 1/(1+1/m)\right)$ while restricting $d$ to be a polynomial in $n$.




\subsubsection{Matching Upper and Lower Bounds} We present two theorems characterizing matching high probability lower and upper bounds on $\BFs$. Taken together, these two theorems give us Theorem \ref{theo: informal}.
\begin{theorem} \label{theo: lower_bound}
Consider the same setup as Theorem \ref{theo: informal}. Then, there exists $n_{LB} \in \bbZ_+$ such that for all $n \geq n_{LB}$, we have
\begin{align*}
    \bars_i \geq n - 2m \frac{n \log d}{d^{m-i+1}} - 4m d^{i-1}\sqrt{m n}\log n -16m^3 \frac{n \log (d)^2}{d^{m-i+2}} \textit{ w.p. at least } 1 - \left(\frac{1}{n}\right)^{(m \log n)/5} \quad \forall i \in [m].
\end{align*}
\end{theorem}

\begin{theorem} \label{theo: upper_bound}
Consider the same setup as Theorem \ref{theo: informal}. Then, there exists $n_{UB} \in \bbZ_+$, such that for all $n \geq n_{UB}$, we have
\begin{align*}
   \bars_i &\leq n - 2m \frac{n\log d}{d^{m-i+1}} + 19m d^{i-1}\sqrt{m n}\log n + 49m^3 \frac{n \log (d)^2}{d^{m-i+2}}+\frac{n^{1-\gamma}}{d^{m-i}}\mathbbm{1}\left\{m > 1\right\}  \ \forall i \in [m] \\
    \bars_{m+1} &\leq 18m d^{m-1}\sqrt{m n}\log n + 48m^3 \frac{n \log (d)^2}{d^{2}}+ n^{1-\gamma}\mathbbm{1}\left\{m >1\right\} \\
    \sum_{l=m +2}^b \bars_{l} &\leq 1.
\end{align*}
each one with probability at least $1-\left(\frac{1}{n}\right)^{(m\log n)/9}$.
\end{theorem}
We prove Theorem \ref{theo: lower_bound} in Section \ref{sec: lower_bound_general_case} and Theorem \ref{theo: upper_bound} in Section \ref{sec: general_case_upper_bound}. Our result only holds true for $b \leq \log (n)^3$ for technical reasons discussed in Section \ref{sec: general_case_upper_bound}. However, note that, since the queue lengths are of size $m$ w.h.p. which is  $o(\log n)$,  the finite buffer requirement $b \leq \log (n)^3$ is inconsequential. So, we expect that the result holds true for all $b \in \bbZ_+$. In particular, by \eqref{eq: m}, we have
\begin{align*}
    m = \Theta\left(\frac{\log n}{\log d}\right) = o(\log n),
\end{align*}
where the last equality follows as $d = \Omega(\log (n)^3)$. 

The terms $4m d^{i-1}\sqrt{m n}\log n$, $16m^3 \frac{n \log (d)^2}{d^{m-i+2}}$, and $\frac{n^{1-\gamma}}{d^{m-i}}\mathbbm{1}\left\{m > 1\right\}$ in Theorem \ref{theo: lower_bound} and similar terms in Theorem \ref{theo: upper_bound} are lower order terms compared to the leading term $2m \frac{n \log d}{d^{m-i+1}}$. We verify it by considering their ratio as follows:
\begin{align*}
    &\frac{4m d^{i-1}\sqrt{m n}\log n + 16m^3 \frac{n \log (d)^2}{d^{m-i+2}}+\frac{n^{1-\gamma}}{d^{m-i}}\mathbbm{1}\left\{m > 1\right\}}{2m \frac{n \log d}{d^{m-i+1}}} = \frac{2\sqrt{m}d^{m}\log n}{\sqrt{n}\log d} + \frac{8m^2\log d}{d} + \frac{d n^{-\gamma}\left\{m > 1\right\}}{2m \log d} \\ 
    \overset{(a)}{\leq}{}& 4n^{\gamma - 0.5}\log (n)^{2.5} + \frac{8\log (n)^2 \log d}{d} + (2m \log d)^{1/m-1}n^{-\gamma(1-1/m)}\left\{m > 1\right\}  \overset{(b)}{\rightarrow} 0, \numberthis \label{eq: lower_order_terms}
\end{align*}
where $(a)$ follows as $d^m = 2m n^\gamma \log d$ and $m=o(\log n)$. Next, $(b)$ follows by substituting $d = \Omega(\log (n)^3)$ and noting that $\gamma < 0.5$. This shows that the rate of increase of $d$ should be large enough to assure that the bound obtained in Theorem \ref{theo: lower_bound} concentrates around the fixed point \eqref{eq: fixed_point_fluid_model}. One can observe that even $d \approx \log (n)^2$ suffices to ensure $(b)$ holds. However, we require a slightly stronger condition due to certain intermediate bounds in the proof. To extend the result to $d < \log (n)^2$, one needs to improve the lower order terms obtained in Theorem \ref{theo: lower_bound}. We leave this as a possible future work and provide heuristic calculations to understand the correct lower-order scaling in Appendix~\ref{app: lower_order_terms}.

\subsection{Iterative State Space Collapse} \label{sec: iterative_ssc}
The central idea to prove Theorem \ref{theo: informal} is the iterative state space collapse framework based on drift analysis which was first developed in \cite{lei_coxian_2_sub_halfin_whitt, lei_coxian_k_sub_halfin_whitt}. As the stochastic system is expected to concentrate around the fixed `point' of the ODE approximation, we keep slicing off the state space until the stationary distribution is implied to live in a ball around the fixed point. In particular, the stationary distribution is iteratively shown to collapse to smaller regions of the state space. Each step of the iteration is achieved by analyzing drift of a carefully engineered Lyapunov function. We present an intuitive overview of the framework and defer the details to Appendix \ref{app: iterative_ssc}

\textbf{Weak SSC:} We start by proving a weak SSC. Let $V_1(\cdot)$ be a Lyapunov functions with $\Delta V_1(\BFs) \leq -\gamma$ for $\BFs \in \calS$ such that $V_1(\BFs) \geq B$ for some $\gamma, B \in \bbR_+$. Then, by standard drift arguments \cite[Theorem 1]{bertsimas_tail_bounds}, we obtain that $\P{V_1(\bbars) \geq B+j}$ is very small for a large enough $j$. This implies that $\bbars$ collapses to the set $\calE_1 = \{V_1(\BFs) \leq B+j\} \subseteq \calS$. 

\textbf{Refining the SSC:} Now, consider the subset of the state space $\calE_1 \subseteq \calS$ where $\bbars$ resides with high probability as shown in the previous step. We analyze the drift of another Lyapunov function $V_2(\cdot)$ restricted to the set $\BFs \in \calE_1$ such that $V_2(\BFs) \geq B$. As $\calE_1 \subseteq \calS$, one would obtain a stronger upper bound on the drift of $V_2(\cdot)$ as opposed to the bound over all whole state-space. Iterative SSC \cite[Lemma 10]{lei_coxian_2_sub_halfin_whitt} framework says that, as long as $\P{\bbars \notin \calE_1} \approx 0$, negative drift of $V_2(\cdot)$ over $\calE_1$ implies a high probability bound on the steady-state distribution. This implies that $\bbars$ now collapses to the set $\calE_2 = \calE_1 \cap \{V_2(\BFs) \leq B+j\}$. A key takeaway is that, the SSC framework poses a trade-off between relaxing the negative drift condition to $\BFs \in \calE_1$ and the steady-state probability $\P{\bbars \notin \calE_1}$. One can carefully negotiate this trade-off to obtain meaningful results.

\textbf{Further Refinements:} The refinement can be repeated multiple times. Consider a family of Lyapunov functions $\{V_k(\BFs): k \in \bbZ_+\}$ and define
\begin{align*}
    \calE_k = \calE_{k-1} \cap \{V_k(\BFs) \leq B + j\}.
\end{align*}
Now, we inductively analyze the drift of $V_{k+1}(\BFs)$ for $\BFs \in \calE_k$ such that $V_{k+1}(\BFs) \geq B$. By iterative SSC framework, negative drift of $V_{k+1}(\BFs)$ implies the high probability bound $\P{V_{k+1}(\BFs) \leq B+j} \approx 1$. Thus, after $k+1$ iterations, the region of SSC shrinks to $\calE_{k+1} = \calE_k \cap \{V_{k+1}(\BFs) \leq B+j\}$. In further sections, we use this inductive framework to prove Theorem \ref{theo: lower_bound} and Theorem \ref{theo: upper_bound}.

Lemma \ref{lemma: iterative_ssc} in Appendix \ref{app: iterative_ssc} formalizes the above intuition and is the workhorse of our proof, used to show each step of the SSC refinement. Our proof is tailored to analyze a queueing system in the steady state allowing us to greatly simplify the methodology compared to \cite{brightwell2018supermarket}. In particular, each step of our refinement directly implies high probability bounds on the steady-state distribution of the queue length $(\BFs)$ process. On the other hand, \cite{brightwell2018supermarket} (see Lemma 3 in their paper) essentially works with a transient version of Lemma \ref{lemma: iterative_ssc}, so they need to worry about any given Lyapunov function to stay small while the other Lyapunov functions decrease, further complicating the proof. Another difference is that our sequence of Lyapunov functions is completely different from that of \cite{brightwell2018supermarket} providing an alternative proof and revealing geometric insights. In particular, while \cite{brightwell2018supermarket} provides an elegant algebraic construction of the sequence of Lyapunov functions, our sequence is based on geometric intuition via the trajectory of the ODE approximation.

\section{Lower Bound}
In this section, we prove Theorem \ref{theo: lower_bound}, by iteratively showing high probability lower bounds on $\{s_i\}_{i \in [m]}$ based on iterative SSC described in the previous section. At a high level, starting with an empty system, we prove a high probability lower bound on $s_1$ as most of the incoming jobs would join an empty queue. Further, using the lower bound on $s_1$, we obtain a lower bound on $s_2$. This procedure is continued for $\{s_i\}_{i \in [m]}$. As most of the proof is algebraic, we first focus on the special case of $m=2$ to provide intuition behind the methodology.
\subsection{Special case \texorpdfstring{$(m=2)$}{}}
\subsubsection{Approximate ODE Trajectory}
We start by analyzing the trajectory of the ODE approximation for $m=2$. The main idea is that the stochastic model would follow a noisy sample path around the ODE trajectory. We present a cartoon of the approximate ODE trajectory based on the differential equations given by \eqref{eq: fluid_model_diff_eqn} in Fig. \ref{fig: fluid_model_trajectory}. The model is initialized by an all-empty system. 

\textbf{Trajectory} $\mathbf{(o) - (a)}$: Each incoming job joins an empty queue until $s_1 = n - n/d$, as there is a surplus of empty queues. In particular, if $s_1 \leq n - n\log d/d$, then the probability of sampling at least one empty queue $\left(1 - (s_1/n)^d\right)$ is almost one. On the other hand, if $s_1 \geq n - n/(d\log d)$, then the probability is almost zero. Mathematically, we have
\begin{align*}
    1 - \left(1 - \frac{\log d}{d}\right)^d &\approx 1 - \frac{1}{d} \rightarrow 1 \\
    1 - \left(1 - \frac{1}{d \log d}\right)^d &\approx \frac{1}{\log d} \rightarrow 0.
\end{align*}
Consistent with the intuition, one can confirm that $ds_1/dt > 0$ when $s_1 \leq n - n\log d/d$. We have
\begin{align*}
    \frac{ds_1}{dt} &= \lambda\left(1 - \left(\frac{s_1}{n}\right)^d\right) - s_1 + s_2\geq \lambda\left(1 - \left(1-\frac{\log d}{d}\right)^d\right) - n + \frac{n \log d}{d} \\
    &\approx n\left(1 - \frac{1}{d}\right)- n + \frac{n \log d}{d}= \frac{n (\log d-1)}{d} > 0.
\end{align*}
Note that, one can show that $s_2$ increases as well but at a smaller rate compared to $s_1$ by observing that $0 < ds_2/dt \ll ds_1/dt$. We omit the details here for brevity. We approximately represent the trajectory as horizontal from $(o)$ to $(a)$ in Fig. \ref{fig: fluid_model_trajectory}.

\textbf{Trajectory} $\mathbf{(a) - (b)}$:
Once $s_1 \approx n - n/d$, most of the incoming jobs start joining a queue with length one, as the probability of sampling an empty queue is asymptotically zero.  Mathematically, for $s_1 \geq n - n/(d \log d)$ and $s_2=o(n)$, we have
\begin{align*}
    \frac{ds_2}{dt} &= \lambda\left(\left(\frac{s_1}{n}\right)^d - \left(\frac{s_2}{n}\right)^d\right) - s_2 + s_3\geq \lambda\left(\left(1 - \frac{1}{d\log d}\right)^d - o(1)^d\right) - o(n) \approx n - \frac{n}{\log d} - o(n) > 0. 
\end{align*}
Thus, $s_2$ increases until $s_2 = \Theta(n)$. Similar to $(o) - (a)$, $s_1$ increases as well but at a smaller rate compared to $s_2$ as $ds_1/dt \ll ds_2/dt$. We approximately represent the trajectory as vertical from $(a)$ to $(b)$ in Fig. \ref{fig: fluid_model_trajectory}.

\textbf{Trajectory} $\mathbf{(b) - (c)}$: As $s_2$ increases, $ds_1/dt$ increases and $ds_2/dt$ decreases. After a critical point, $ds_1/dt$ and $ds_2/dt$ are comparable resulting in both $s_1$ and $s_2$ increasing at a similar rate. This is approximately represented as a tilted trajectory from $(b)$ to $(c)$. Once $s_1 = \lambda \approx n - n/d^2$, the arrival rate is equal to the effective service rate (number of busy servers), and the ODE trajectory converges to that point. One can verify that $(c)$ is indeed the fixed point (ignoring the logarithmic terms) of the ODE approximation.
\begin{figure}[hbt!]
    \FIGURE{
    \begin{tikzpicture}[scale=0.7]
    \draw[black, thick, ->] (0,0) -- (0,6);
    \draw[black, thick, ->] (0,0) -- (7,0);
    \draw[red, ultra thick] (0,0) -- (4,0);
    \draw[red, ultra thick, ->] (0, 0) -- (2,0);
    \draw[red, ultra thick] (4,0) -- (4,3);
    \draw[red, ultra thick, ->] (4,0) -- (4,1.5);
    \draw[red, ultra thick] (4,3) -- (5,4);
    \draw[red, ultra thick, ->] (4,3) -- (4.5, 3.5);
    \draw[black, thick] (5.5,0) -- (5.5,5.5);
    \draw[black, thick] (0,0) -- (5.5,5.5);
    \draw[black, thin, dashed] (0, 5.5) -- (5.5,5.5);
    \node at (5.5, -0.3) (1) {$n$};
    \node at (-0.3, 5.5) (1) {$n$};
    \node at (7, 0.2) (1) {$s_1$};
    \node at (0.25, 6) (1) {$s_2$};
    \node at (0, -0.3) {$(o)$};
    \node at (4, -0.3) (1) {$(a)$};
    \node at (3.7, 3) (1) {$(b)$};
    \node at (5.25, 4) (1) {$(c)$};
    \matrix[ampersand replacement=\&] at (10,4){
    \node (species1)  {
        \begin{tabular}{|c |c| c|}
        \hline 
         & $s_1$ & $s_2$ \\ \hline
        $(o)$ & 0 & 0 \\
        $(a)$ & $n-n/d$ & 0 \\
        $(b)$ & $n-n/d$ & $\Theta(n)$ \\
        $(c)$ & $s_1=n-n/d^2$ & $s_2=n-n/d$ \\ \hline
        \end{tabular}
    }; \\
};
    \end{tikzpicture}}{
   Approximate ODE trajectory with the initial condition equal to $s_1=0$.
    \label{fig: fluid_model_trajectory}}{}
\end{figure}

Note that Fig. \ref{fig: fluid_model_trajectory} is the transient behavior of the ODE approximation starting from an all-empty system. On the other hand, our goal is to characterize the steady-state behavior of the stochastic system. Steady-state corresponds to a fixed distribution invariant with time and it doesn't follow a transient trajectory as shown in Fig. \ref{fig: fluid_model_trajectory}. However, the sequence of SSC we establish in the next sub-section is inspired by the transience of the ODE approximation.  Such an interpretation of the ODE approximation is also consistent with the fundamentals of the Lyapunov drift arguments. In particular, drift of a Lyapunov function is a transient quantity as it depends on the state of the system. In turn, it implies a high probability bound on the steady-state distribution. Lastly, note that we do not have a cartoon for the ODE approximation for the general case as it would be $m$ dimensional. However, we carry forward the algebraic intuitions developed using $m=2$ for the general case.

\subsubsection{Stochastic Analysis}
The form of iterative SSC is inspired by the ODE trajectory depicted in Fig. \ref{fig: ssc_lower_bound_1}. In particular, there is a drift that pulls the system closer to the ODE trajectory. We use this idea to show that the states that are far from the trajectory are experienced with low steady-state probability. 
\begin{figure}[hbt!]
\FIGURE{
    \begin{minipage}[b]{0.24\textwidth}
    \centering
     \begin{tikzpicture}[scale=0.5]
     \fill[blue!20] (3.8, 0) -- (5.5, 0) -- (5.5, 5.5) -- (3.8, 3.8);
     \fill[pattern=crosshatch, pattern color=red!45] (0, 0) -- (3.8 , 0) -- (3.8, 3.8);
    \draw[black, thick, ->] (0,0) -- (0,6);
    \draw[black, thick, ->] (0,0) -- (7,0);
    \draw[red, ultra thick] (0,0) -- (4,0);
    \draw[red, ultra thick] (4,0) -- (4,3);
    \draw[red, ultra thick] (4,3) -- (5,4);
    \draw[black, thick] (5.5,0) -- (5.5,5.5);
    \draw[black, thick] (0,0) -- (5.5,5.5);
    \draw[black, thick] (3.8, 0) -- (3.8, 3.8);
    \draw[blue, very thick, ->] (1.8, 1.3) -- (3.6, 0.7);
    \draw[blue, very thick, ->] (1, 0.5) -- (3.6, 0.2);
    \draw[blue, very thick, ->] (2.4, 2) -- (3.6, 1.25);
    \draw[black, dashed, thick] (3.8, 3.8) -- (3.8, 6);
    \draw[black, thick, ->] (3.8, 5.8) -- (5, 5.8) node at (5.2, 6.4) {Large $s_1$};
    \end{tikzpicture}
    \subcaption{$s_1 \uparrow$ in the red region until it enters the blue region.}
    \label{fig: lower_bound_step1}
    \end{minipage}
    \begin{minipage}[b]{0.24\textwidth}
    \centering
    \begin{tikzpicture}[scale=0.5]
     \fill[blue!20] (3.8, 0) -- (4.8, 0) -- (4.8, 3.8) -- (5.5, 3.8) --  (5.5, 5.5) -- (3.8, 3.8);
     \fill[pattern=crosshatch, pattern color=red!45] (4.8, 0) -- (5.5, 0) -- (5.5, 3.8) -- (4.8, 3.8);
    \draw[black, thick, ->] (0,0) -- (0,6);
    \draw[black, thick, ->] (0,0) -- (7,0);
    \draw[red, ultra thick] (0,0) -- (4,0);
    \draw[red, ultra thick] (4,0) -- (4,3);
    \draw[red, ultra thick] (4,3) -- (5,4);
    \draw[black, thick] (5.5,0) -- (5.5,5.5);
    \draw[black, thick] (0,0) -- (5.5,5.5);
    \draw[black, thick] (3.8, 0) -- (3.8, 3.8);
    \draw[black, thick] (4.8,0) -- (4.8, 3.8);
    \draw[black, thick] (5.5,3.8) -- (4.8, 3.8);
    \draw[blue, very thick, ->] (5.4, 0.3) -- (4.9, 0.5);
     \draw[blue, very thick, ->] (5.4, 1.3) -- (4.9, 1.5);
      \draw[blue, very thick, ->] (5.4, 2.5) -- (4.9, 2.7);
     \draw[black, thick, dashed] (2.5, 3.8) -- (6.5, 3.8);
     \draw[black, thick, ->] (3, 3.8) -- (3, 4.5) node at (1.5, 4.4) {Large $s_2$};
     \draw[black, thick, dashed] (4.8, 3.8) -- (4.8, 6.2);
     \draw[black, thick, ->] (4.8, 5.8) -- (3.8, 5.8) node at (3.4, 6.4) {Small $s_1$};
    \end{tikzpicture}
    \subcaption{$s_1 \downarrow$ \& $s_2 \uparrow$ in the red region until it enters the blue region.}
    \label{fig: lower_bound_step2a_1}
    \end{minipage}
        \begin{minipage}[b]{0.24\textwidth}
    \centering
      \begin{tikzpicture}[scale=0.5]
     \fill[blue!20] (4.8, 3.8) -- (5.5, 3.8) -- (5.5, 5.5) -- (4.2, 4.2) -- (4.8, 3.6);
     \fill[pattern=crosshatch, pattern color=red!45] (3.8, 0) -- (4.8, 0) -- (4.8, 3.6) -- (4.2, 4.2) -- (3.8, 3.8);
    \draw[black, thick, ->] (0,0) -- (0,6);
    \draw[black, thick, ->] (0,0) -- (7,0);
    \draw[red, ultra thick] (0,0) -- (4,0);
    \draw[red, ultra thick] (4,0) -- (4,3);
    \draw[red, ultra thick] (4,3) -- (5,4);
    \draw[black, thick] (5.5,0) -- (5.5,5.5);
    \draw[black, thick] (0,0) -- (5.5,5.5);
    \draw[black, thick] (3.8, 0) -- (3.8, 3.8);
    \draw[black, thick] (4.8,0) -- (4.8, 3.8);
    \draw[black, thick] (5.5,3.8) -- (4.8, 3.8);
    \draw[black, thick] (4.8, 3.6) -- (4.2, 4.2);
    \draw[blue, very thick, ->] (4.5, 0.3) -- (4.3, 2);
    \draw[blue, very thick, ->] (3.9, 3.5) -- (4.4, 3.7);
    \draw[black, thick, dashed] (3.2, 5.2) -- (6, 2.4);
    \draw[black, thick, ->] (3.5, 4.9) -- (4.5, 5.9) node[midway, above, sloped, align=center] at (4.4, 6.4) {Large $s_1+s_2$};
    \end{tikzpicture}
    \subcaption{$s_1+s_2 \uparrow$ in the red region until it enters the blue region.}
    \label{fig: lower_bound_step2a_2}
    \end{minipage}
    \begin{minipage}[b]{0.24\textwidth}
    \centering
      \begin{tikzpicture}[scale=0.5]
     \fill[blue!20] (4.8 ,3.8) -- (5.5, 3.8) -- (5.5, 5.5) -- (4.6, 4.6) -- (4.6, 3.8) -- (4.8, 3.6);
     \fill[pattern=crosshatch, pattern color=red!45] (4.6, 4.6) -- (4.6, 3.8) -- (4.2, 4.2);
    \draw[black, thick, ->] (0,0) -- (0,6);
    \draw[black, thick, ->] (0,0) -- (7,0);
    \draw[red, ultra thick] (0,0) -- (4,0);
    \draw[red, ultra thick] (4,0) -- (4,3);
    \draw[red, ultra thick] (4,3) -- (5,4);
    \draw[black, thick] (5.5,0) -- (5.5,5.5);
    \draw[black, thick] (0,0) -- (5.5,5.5);
    \draw[black, thick] (3.8, 0) -- (3.8, 3.8);
    \draw[black, thick] (4.8,0) -- (4.8, 3.8);
    \draw[black, thick] (5.5,3.8) -- (4.8, 3.8);
    \draw[black, thick] (4.8, 3.6) -- (4.2, 4.2);
    \draw[black, thick] (4.6, 3.8) -- (4.6, 4.6);
    \draw[blue, very thick, ->] (4.2, 4.2) -- (4.9, 4.2);
    \draw[black, thick, dashed] (4.6, 2.3) -- (4.6, 6.5);
    \draw[black, thick, ->] (4.6, 5.8) -- (5.6, 5.8) node at (6, 6.4) {Large $s_1$};
    \end{tikzpicture}
    \subcaption{$s_1$ increases in the red region until it enters the blue region.}
    \label{fig: lower_bound_step2b}
    \end{minipage}}{
 Graphical representation of iterative SSC for lower bound for $m=2$: The red hatched region is shown to have low steady-state probability leading to the collapse into the solid blue region. In addition, the arrows represent the drift -  $(ds_1/dt, ds_2/dt)$.
    \label{fig: ssc_lower_bound_1}}{}
\end{figure}
Now, we elucidate the steps to prove the theorem that are outlined in Fig. \ref{fig: proof_outline_lb}.
\begin{itemize}
    \item \textbf{Step 1:} We first show that $\bars_1 \geq n - 2m n\log d/d$ w.h.p. corresponding to trajectory $(o)-(a)$ in the ODE approximation. Intuitively, all incoming jobs join empty queues due to their availability in surplus.
    \item \textbf{Step 2a:} Given the lower bound on $\bars_1$, we obtain $\bars_2 \geq n - 6m^2 n \log d/d$ w.h.p. corresponding to trajectory $(a)-(b)$ in the ODE approximation. This is carried out in two-steps as depicted in Fig. \ref{fig: lower_bound_step2a_1} and \ref{fig: lower_bound_step2a_2}.
    \item \textbf{Step 2b:} We improve the lower bound on $\bars_1$ using the lower bound on $\bars_2$ to get $\bars_1 \geq n - 9m^2n\log d/d^2$ that loosely corresponds to trajectory $(b)-(c)$ in the ODE approximation. This step is illustrated in Fig. \ref{fig: lower_bound_step2b}.
    \item \textbf{Step 2c:} This step improves the lower bound on $\bars_2$ previously obtained in Step 2a. In particular, we leverage the newly obtained lower bound on $\bars_1$ and repeat the same steps as in Step 2a, to obtain a better lower bound on $\bars_2$. We call this a bootstrapping step as a weaker lower bound on $\bars_2$ results in a stronger lower bound on itself.
    \item \textbf{Step 2d:} Similar to the previous step, we improve the lower bound on $\bars_1$ previously obtained in Step 2b. In particular, using a better lower bound on $\bars_2$ obtained in Step 2c, we obtain a better lower bound on $\bars_1$ by following the outline of Step 2b.
\end{itemize}
\textbf{Why Iterative SSC?:} Ideally, we would like to construct a single Lyapunov function whose drift analysis would reveal that the stochastic system concentrates around the fixed point of the ODE. However, it is not clear if such a Lyapunov function exists, or at least, we were not successful in constructing such a function. In our approach, we construct a sequence of Lyapunov functions that mimics the trajectory of the ODE. The ODE trajectory in Fig.~\ref{fig: fluid_model_trajectory} reveals that the drift of $s_1$ and $s_2$ are of different orders in different parts of the state space. For example, the upward drift of $s_1$ is very large compared the drift of $s_2$ when the system is close to an empty system. Such a behavior is reminiscent of the state space collapse result in classical heavy-traffic regime \cite{atilla_jsq_ht_drift} and two timescale algorithms in reinforcement learning (e.g. see \cite[Chapter 6]{borkar2009stochastic} and \cite{haque2023tight}). Such a two timescale behavior warrants the need of constructing multiple Lyapunov functions to analyze the collapse of $s_1$ and $s_2$. 

In the next sub-section, we build upon the intuition to extend the proof to the general case as shown in Fig. \ref{fig: proof_outline_lb}. We do not separately present the proof for $m=2$ as it is subsumed in the general case.
\begin{figure}[htb!]
    \FIGURE{
    \begin{tikzpicture}
         \node[rectangle, draw=black!60, fill=green!5, very thick, minimum size=7mm, align=center] at (-8, -0.5) (1) {\textbf{Step 1:} $\bars_1 \geq n - \frac{2 m n \log d}{d}\left(1+o(1)\right)$ (Fig. \ref{fig: lower_bound_step1})};
          \node[rectangle, draw=black!60, fill=green!5, very thick, minimum size=7mm, text width=65mm, align=center] [below= 7mm of 1] (2) {\textbf{Step 2:} $\bars_1 \geq n - \frac{2 m n \log d}{d^{2}}\left(1+o(1)\right)$ and $\bars_2 \geq n - \frac{2 m n \log d}{d}\left(1+o(1)\right)$};
          \draw[fill=black] ([yshift=-5mm] 2.south) circle (0.05);
          \draw[fill=black] ([yshift=-7mm] 2.south) circle (0.05);
          \draw[fill=black] ([yshift=-9mm] 2.south) circle (0.05);
           \node[rectangle, draw=black!60, fill=green!5, very thick, minimum size=7mm, text width=70mm, align=center] [below= 15mm of 2] (k) {\textbf{Step $k$:} $\bars_i \geq n - \frac{2 m n \log d}{d^{k-i+1}}\left(1+o(1)\right)$ for all $i \in [k]$};
            \draw[fill=black] ([yshift=-5mm] k.south) circle (0.05);
          \draw[fill=black] ([yshift=-7mm] k.south) circle (0.05);
          \draw[fill=black] ([yshift=-9mm] k.south) circle (0.05);
           \node[rectangle, draw=black!60, fill=green!5, very thick, minimum size=7mm, text width=70mm, align=center] [below= 15mm of k] (m) {\textbf{Step $m$:} $\bars_i \geq n - \frac{2 m n \log d}{d^{m-i+1}}\left(1+o(1)\right)$ for all $i \in [m]$};
         \draw[fill=green!5, draw=black!60, very thick] (-3.4, -1.05) -- (4.6, -1.05) -- (4.6, -3.6) -- (-3.4, -3.6) -- (-3.4, -1.05);
        \node[rectangle, draw=black!60, fill=red!5, thick, minimum size=7mm] [below right= 2.5mm and 8mm of 1] (2a) {\textbf{Step 2a:} $\bars_2 \geq n - \left(\frac{6m^2 n\log d}{d}\right)$ (Fig. \ref{fig: lower_bound_step2a_1} \& \ref{fig: lower_bound_step2a_2})};
        \node[rectangle, draw=black!60, fill=red!5, thick, minimum size=7mm] [below= 4mm of 2a] (2b) {\textbf{Step 2b:} $\bars_1 \geq n - \left(\frac{9m^2 n\log d}{d^2}\right)$ (Fig. \ref{fig: lower_bound_step2b})};
         \draw[fill=green!5, draw=black!60, very thick] (-3.4, -4.3) -- (4.6, -4.3) -- (4.6, -6.6) -- (-3.4, -6.6) -- (-3.4, -4.3);
        \node[rectangle, draw=black!60, fill=red!5, thick, minimum size=7mm] [below= 9mm of 2b] (3a) {\textbf{Step 2c:} $\bars_2 \geq n - \frac{2 m n \log d}{d}\left(1+o(1)\right)$};
        \node[rectangle, draw=black!60, fill=red!5, thick, minimum size=7mm] [below= 6mm of 3a] (3b) {\textbf{Step 2d:} $\bars_1 \geq n - \frac{2 m n \log d}{d^{2}}\left(1+o(1)\right)$};
        \node [above = 1mm of 1] (base_case) {Base Case};
        \node [above = 0.9mm of 2a] (bootstrap_1) {First iteration of Bootstrapping};
        \node [above = 0.9mm of 3a] (bootstrap_2) {Second iteration of Bootstrapping};
        \draw[black, thick] ([yshift=6mm] 2.east) -- (-3.4, -1.05);
        \draw[black, thick] ([yshift=-6mm] 2.east) -- (-3.4, -6.6);
        \draw[fill=green!3] ([yshift=6mm] 2.east) -- (-3.4, -1.05) -- (-3.4, -6.6) -- ([yshift=-6mm] 2.east);
    \end{tikzpicture}}{Proof Outline: High Probability Lower Bounds.
    \label{fig: proof_outline_lb}}{}
\end{figure}
\subsection{General Case} \label{sec: lower_bound_general_case}
The proof is mainly divided in five steps and we present five lemmas corresponding to these steps. The proof outline is presented in Fig. \ref{fig: proof_outline_lb} and we naturally prove it using induction. In step $k$, we provide a high probability lower bound on $\bars_k$ and improve the previous bounds obtained for $\{\bars_i\}_{i \in [k-1]}$. The proof of all the lemmas mentioned in this section is deferred to Appendix \ref{app: lb_lemmas}.

\textbf{Base Case:} We start by obtaining a high probability lower bound on $\bars_1$ (Step 1) in the following lemma.
\begin{lemma}[Base Case] \label{lemma: base_case_lower_bound} Consider the same setup as Theorem \ref{theo: informal}. Then, there exists $n_0 \in \bbZ_+$ such that for all $n \geq n_0$,  \label{lemma: outer_base_case}
\begin{align*}
    \P{\bars_1 \leq n - \frac{2m n \log d}{d} - \frac{2n\log d}{d^{2}} - 2\sqrt{m n}\log n} \leq \left(\frac{1}{n}\right)^{m\log n /4}.
\end{align*}
\end{lemma}
The proof of the lemma follows by showing that the drift of the Lyapunov function,
\begin{align*}
    V_1(\BFs) = n - \frac{2m n \log d}{d} - \frac{2n\log d}{d^{2}} - \sqrt{m n}\log n-s_1,
\end{align*}
is negative when $V_1(\BFs) \geq 0$. Thus, $s_1$ increases w.h.p. when $s_1 \leq n-\frac{2m n \log d}{d}(1+o(1))$, implying that $\bars_1 > n-\frac{2m n \log d}{d}(1+o(1))$ w.h.p. This completes the base case. Also note that the above lemma completes the proof of Theorem~\ref{theo: lower_bound} for $m=1$. Now, we consider $m \geq 2$ for the rest of the proof.

\textbf{Induction Step:} For some $k \in \{2, \hdots, m\}$, we define our induction hypothesis as follows. There exists $n_{IH} \in \bbZ_+$, independent of $k$ such that for all $n \geq n_{IH}$, we have
\begin{align}
   \P{\bars_i \leq n - 2m \frac{n \log d}{d^{k-i}} - 4m d^{i-1}\sqrt{m n}\log n -16m^3 \frac{n \log (d)^2}{d^{k-i+1}}} \leq \left(\frac{1}{n}\right)^{\frac{m \log n}{4}-4(k-1)m} \quad \forall i \in [k-1]. \tag{IH} \label{eq: outer_induction}
\end{align}
Note that, \eqref{eq: outer_induction} is equivalent to Step $k-1$ as in the proof outline given by Fig. \ref{fig: proof_outline_lb}. After completing the induction step, we improve the bound on $(n-\bars_i)$ by a factor of $d$ for all $i \in [k-1]$ and moreover, we introduce a lower bound on $\bars_k$ of the order $n - 2nm\log d/d$. Similar to the way Step 2 is proved in four parts as in Fig. \ref{fig: proof_outline_lb}, we present four lemmas that together completes the induction step. The first of the four lemmas corresponds to part $a$ of the induction step (Steps $2$ to $m$ in Fig. \ref{fig: proof_outline_lb}), wherein we obtain a lower bound on $\bars_k$. 
\begin{lemma}[Induction Part $a$] \label{lemma: outer_induction_weak_induction_step}
Consider the same setup as Theorem \ref{theo: informal}. Also, assume that \eqref{eq: outer_induction} holds for some $k \leq m$. Then, there exists $n_1 \in \bbZ_+$ such that for all $n \geq n_1$, we have
\begin{align*}
    \P{\bars_{k} \leq n - 6m^2 \frac{n\log d}{d}} \leq \left(\frac{1}{n}\right)^{\frac{m \log n}{4}-4\left(k-\frac{3}{4}\right)m}.
\end{align*}
\end{lemma}
Note that, the above lower bound is weak, as the term accompanying $n\log n/d$ is $6m^2$ which is larger than the required $2 m$ in the induction step. We improve this bound in part $c$ (Lemma \ref{lemma: base_case_of_induction_hypothesis}). Now, we present the proof sketch.
\proof{Proof Sketch of Lemma \ref{lemma: outer_induction_weak_induction_step}} The proof of this lemma is induction based. To state the induction hypothesis, consider a Lyapunov function of the following form (the exact expression is presented in the proof):
\begin{align*}
    L_{ik}(\BFs) = \min \left\{s_i - n + \Theta\left(\frac{m^2 n\log d}{d^{k-i+1}}\right), (k-i+1)n - \Theta\left(\frac{m^2 n\log d}{d}\right)-\sum_{l=i+1}^k s_l\right\} \quad \forall i \in [k-1
    ].
\end{align*}
Now, for some $i \in [k-1]$, the induction hypothesis is given as follows, which runs backwards on the index $i$: $L_{ik}(\bbars) = O(\sqrt{m n}\log n).$
To prove the induction step, we show that the drift of the Lyapunov function $L_{i-1,k}(\BFs)$ is negative for all states such that $L_{i, k}(\BFs) = O(\sqrt{m n}\log n)$ and $L_{i-1,k}(\BFs) \geq 0$. Thus, by applying the iterative SSC framework with $V=L_{ik}$ and, $\calE = \left\{L_{i+1, k}(\BFs) = O(\sqrt{m n}\log n)\right\}$, we conclude that $L_{i-1, k}(\bbars) = O(\sqrt{m n}\log n)$ w.h.p. Thus, the induction implies that $L_{1k}(\bbars)=O(\sqrt{m n}\log n)$ w.h.p. Note that $L_{i-1, k}(\BFs)$ is small implies that either $s_{i-1}$ is small or $\sum_{l=i}^k s_l$ is large. This form of state space collapse is reminiscent of the first sub-task of Step 2a as depicted in Fig. \ref{fig: lower_bound_step2a_1} in which we showed that either $\bars_1$ is small or $\sum_{l=2}^b \bars_l$ is large. 

To complete the proof, we further consider the following Lyapunov function:
\begin{align*}
    L_{0k}(\BFs) = kn - \Theta\left(\frac{m^2 n\log d}{d}\right) - \sum_{l=1}^k s_l.
\end{align*}
Using the fact that $L_{1k}(\bbars) = O(\sqrt{m n}\log n)$ w.h.p., we show that $L_{0k}(\bbars) = O(\sqrt{m n}\log n)$ w.h.p. In particular, we show that the drift of $L_{0k}(\BFs)$ is negative for all states such that $L_{1k}(\BFs) = O(\sqrt{m n}\log n)$ and $L_{0k}(\BFs) \geq 0$. Thus, by Lemma \ref{lemma: iterative_ssc} with $V = L_{0k}$, $\calE = \{L_{1k} = O(\sqrt{m n}\log n)\}$, $B=0$, and $j=\sqrt{m n}\log n$, we conclude that $\sum_{l=1}^k \bars_l \geq kn - \Theta\left(m^2 n\log d /d\right)$ w.h.p. Observing that $\bars_l \leq n$ w.p. 1, we get $\bars_k \geq n-\Theta\left(m^2n\log d/d\right)$ w.h.p. This completes the proof. Note that, analyzing the drift of $L_{0k}(\BFs)$ is equivalent to the second sub-task of Step 2a as depicted in Fig. \ref{fig: lower_bound_step2a_2}. \hfill $\square$
\endproof
In the next lemma, we prove weak lower bounds on $\bars_i$ for all $i \in [k-1]$ completing part $b$ of steps $2-m$, analogous to Step 2b as outlined in Fig. \ref{fig: proof_outline_lb}.
\begin{lemma}[Induction Part $b$] \label{lemma: remaining_weak_outer_induction}
Consider the same setup as Theorem \ref{theo: informal}. Also, assume that $m \geq 2$ and \eqref{eq: outer_induction} holds for some $k \leq m$. Then, there exists $n_2 \in \bbZ_+$ such that for all $n \geq n_2$, we have
\begin{align*}
    \P{\bars_i \leq n - 9m^2\frac{n \log d}{d^{k-i+1}}} \leq \left(\frac{1}{n}\right)^{\frac{m \log n}{4}-4\left(k-\frac{1}{2}\right)m} \quad \forall i \in [k].
\end{align*}
\end{lemma}
Note that, the bound obtained in the above lemma is weaker compared to the required bound for the induction step as $9m^2 > 2m$. We improve on this bound in part $d$ to obtain the induction step. We now present the sketch of the proof.
\proof{Proof Sketch of Lemma \ref{lemma: remaining_weak_outer_induction}}
We again use induction on $i$ to prove bounds on $\bars_i$ for all $i \in [k]$ by going backwards over the index $i$. The base case $(i=k)$ is already proved in Lemma \ref{lemma: outer_induction_weak_induction_step}. To prove the induction step, we analyze the drift of the following Lyapunov function:
\begin{align*}
    W_i(\BFs) = n - \Theta\left(\frac{m^2 n\log d}{d^{k-i+1}}\right) - s_i.
\end{align*}
We consider the induction hypothesis $W_{i+1}(\bbars)=O(\sqrt{mn}\log n)$ w.h.p. which is equivalent to a high probability lower bound on $\bars_{i+1}$. We show that the drift of $W_i(\BFs)$ is negative when $W_{i+1}(\BFs)=O(\sqrt{mn}\log n)$ and $W_i(\BFs) \geq 0$. Thus, by Lemma \ref{lemma: iterative_ssc}, we conclude that $W_i(\bbars)=O(\sqrt{mn}\log n)$ w.h.p. This provides a lower bound on $\bars_i$ which completes the proof. Note that, the bound on $W_{i+1}(\BFs)$ is crucial in obtaining the negative drift for $W_i(\BFs)$. Thus, the iterative version of SSC is an integral part of the proof. \hfill $\square$
\endproof
Now, we prove the required lower bound on $\bars_k$ completing part $c$ of steps $2-m$, analogous to Step 2c as in Fig. \ref{fig: proof_outline_lb}.
\begin{lemma}[Induction Part $c$] \label{lemma: base_case_of_induction_hypothesis}
Consider the same setup as Theorem \ref{theo: informal}. Also, assume that $m \geq 2$ and \eqref{eq: outer_induction} holds for some $k \leq m$. Then, there exists $n_3 \in \bbZ_+$ such that for all $n \geq n_3$, we have
\begin{align*}
    \P{\bars_k \leq n - 2m \frac{n \log d}{d} - 3m d^{k-1}\sqrt{m n}\log n -12m^2 \frac{n \log d}{d^2}} \leq \left(\frac{1}{n}\right)^{\frac{m \log n}{4}-4\left(k-\frac{1}{4}\right)m}.
\end{align*}
\end{lemma}
Note that the above lemmas replaces the coefficient of $n \log d /d$ from $6m^2$ in Lemma \ref{lemma: outer_induction_weak_induction_step} with the required $2m$ for the induction step. The proof of the above lemma follows similarly to the proof of Lemma \ref{lemma: outer_induction_weak_induction_step} but leverages the lower bound on $\bars_{k-1}$ proved in Lemma \ref{lemma: remaining_weak_outer_induction}. As the lower bound provided by Lemma \ref{lemma: remaining_weak_outer_induction} is sharper than the induction hypothesis \eqref{eq: outer_induction}, we obtain a better lower bound on $\bars_k$. Turns out, applying this bootstrapping step once suffices to prove the induction step for $\bars_k$. Now, using this bound, we obtain the required lower bound on $\bars_i$ for all $i \in [k-1]$ which completes the part $d$ of steps $2-m$, analogous to Step 2d as outlined in Fig. \ref{fig: proof_outline_lb}.
\begin{lemma}[Induction Part $d$] \label{lemma: induction_step}
Consider the same setup as Theorem \ref{theo: informal}. Also, assume that $m \geq 2$ and \eqref{eq: outer_induction} holds for some $k \leq m$. Then, there exists $n_4 \in \bbZ_+$ such that for all $n \geq n_4$, we have
\begin{align*}
    \P{\bars_i \leq n - 2m \frac{n \log d}{d^{k-i+1}} - 4m d^{i-1}\sqrt{m n}\log n -16m^3 \frac{n \log (d)^2}{d^{k-i+2}}} \leq \left(\frac{1}{n}\right)^{\frac{m \log n}{4}-4km} \quad \forall i \in [k]. \numberthis \label{eq: induction_step_lb}
\end{align*}
\end{lemma}
Note that the above lemma proves the coefficient of $n \log d /d^{k-i+1}$ from $9m^2$ in Lemma \ref{lemma: remaining_weak_outer_induction} to the required $2m$ for the induction step. The proof follows similarly to the proof of Lemma \ref{lemma: remaining_weak_outer_induction} but leverages a sharper lower bounds on $\bars_i$ for all $i \in [k-1]$ proved in Lemma \ref{lemma: remaining_weak_outer_induction}. Note that the above lemma essentially completes the induction because \eqref{eq: induction_step_lb} is same as \eqref{eq: outer_induction} with $k-1$ replaced by $k$. Now, we use the above lemmas below to prove Theorem \ref{theo: lower_bound}.
\proof{Proof of Theorem \ref{theo: lower_bound}}
Consider the induction hypothesis as defined in \eqref{eq: outer_induction}. The base case follows by Lemma \ref{lemma: base_case_lower_bound} for $n_{IH} \geq \max_{l \in \{0, \hdots, 4\}}\{n_l\}$, and the induction step follows by Lemma \ref{lemma: induction_step}. Thus, for all $n \geq n_{IH}$, we have
\begin{align*}
     \P{\bars_i \leq n - 2m \frac{n \log d}{d^{m-i+1}} - 4m d^{i-1}\sqrt{m n}\log n -16m^3 \frac{n \log (d)^2}{d^{m-i+2}}} &\leq \left(\frac{1}{n}\right)^{\frac{m \log n}{4}-4m^2} \\
     &\leq  \left(\frac{1}{n}\right)^{\frac{m \log n}{5}}\quad \forall i \in [m],
\end{align*}
where the last inequality follows for all $n \geq n_5$ for some $n_5 \in \bbZ_+$ by noting that $m=o(\log n)$. Now, by defining $n_{LB}\overset{\Delta}{=} \max\{n_5, n_{IH}\}$, the proof is complete. \hfill $\square$ 
\endproof
\section{Upper Bound}
Similar to the previous section, the analysis for the upper bound follows the iterative SSC framework. We first prove high probability upper bounds on $\bars_{m+1}$ and $\sum_{l=m+2}^b \bars_l$ and then inductively prove high probability upper bounds on $\{\bars_i\}_{i \in [m]}$. We start by focusing on the special case of $m=1$ to provide intuition behind the methodology.
\subsection{Special Case \texorpdfstring{$(m=1)$}{}} \label{sec: special_case_upper_bound}
The geometric intuition for the steps to prove the special case is given in Fig. \ref{fig: upper_bound_ssc}. In particular, Fig. \ref{fig: step0_ub} corresponds to Theorem \ref{theo: lower_bound} (lower bound on $s_1$) for the special case of $m=1$. The other two sub-figures correspond to proving matching upper bounds. Now, we elucidate the steps to obtain the upper bound for $m=1$, as outlined in Fig. \ref{fig: proof_sketch_ub}.  
\begin{itemize}
    \item \textbf{Step 1a:} We first show that $\sum_{l=2}^b \bars_l = O\left(b\sqrt{n}\log n\right)$ w.h.p. Noting that $\bars_i \geq 0$ for all $i \in \{3, \hdots, b\}$, we get $\bars_2=O\left(b\sqrt{n}\log n\right)=o(n)$ w.h.p. This is a weaker bound than what is required for Theorem \ref{theo: upper_bound}. We improve it further in Step 1c by first proving a high probability upper bound on $\sum_{l=m+2}^b \bars_l$ in the next step.
    \item \textbf{Step 1b:} The probability of an incoming job joining a queue with length at least two is equal to the probability of sampling $d$ queues with length at least two. This is equal to $\left(\frac{\bars_2}{n}\right)^d = o(1)^d$ w.h.p. which converges to zero very fast. Consistent with the intuition, we show that $\sum_{l=3}^b \bars_l=o(1)$ w.h.p.
    \item \textbf{Steps 1c and 1d:} Using the upper bound proved in Step 1b, we show that $\bars_2 = O\left(\sqrt{n}\log n\right)$ and $\bars_1 \leq \lambda + O(\sqrt{n}\log n)$ w.h.p. This is obtained by first showing \eqref{eq: min_lyapunov _function} holds, which allows us to prove \eqref{eq: rest_of_the_ssc} as shown below.    \begin{subequations}
    \begin{align}
        \min{}&\left\{\sum_{l=2}^b \bars_l - O(\sqrt{n}\log n), \lambda + O(\sqrt{n}\log n)-\bars_1\right\} \leq 0 \quad \textit{w.h.p.} \label{eq: min_lyapunov _function}\\
        \sum_{l=1}^b{}& \bars_l \leq \lambda + O(\sqrt{n}\log n) \quad \textit{w.h.p.} \label{eq: rest_of_the_ssc}
    \end{align}
    \end{subequations}
    Eq. \eqref{eq: rest_of_the_ssc} immediately implies that $\bars_1 \leq \lambda + O(\sqrt{n}\log n)$ as $\sum_{l=2}^b \bars_l \geq 0$. In addition, by using the lower bound on $\bars_1$ given by Theorem \ref{theo: lower_bound} in \eqref{eq: rest_of_the_ssc}, we obtain that $\sum_{l=2}^b \bars_l \leq O(\sqrt{n}\log n)$ completing Step 1c and 1d. The SSC corresponding to \eqref{eq: min_lyapunov _function} is depicted in Fig. \ref{fig: step1a_ub_1} which shows that either $s_1$ is large or $s_2$ is small. In particular, if $s_1$ is small and $s_2$ is large, then most of incoming jobs will join an empty queue resulting in $s_1$ increasing and $s_2$ decreasing. Further, the SSC corresponding to \eqref{eq: rest_of_the_ssc} is depicted in Fig. \ref{fig: step1a_ub_2} which upper bounds $\sum_{l=1}^b \bars_l$ w.h.p.
\end{itemize}
Tight characterization of $\bbars$ for the special case of $m=1$ was first obtained in \cite{sub_halfin_whitt_lei}. The authors used Stein's method in conjunction with SSC to prove the result. In particular, the SSC proved in \cite{sub_halfin_whitt_lei} is equivalent to \eqref{eq: min_lyapunov _function} as depicted in Fig. \ref{fig: step1a_ub_1}. This alone is not sufficient to characterize the complete stationary distribution. So, Stein's method was used along with \eqref{eq: min_lyapunov _function} to obtain the stationary distribution. On the other hand, we circumvent the use of Stein's method by using the iterative SSC framework, to further improve the SSC and obtain tight bounds on the stationary distribution. A takeaway from our paper is that one can simply use the iterative SSC approach to obtain tight bounds on the queue lengths if the stochastic system lives close to the fixed point of the corresponding dynamical system given by a set of ODEs.
\begin{figure}
    \FIGURE{
        \begin{minipage}[b]{0.32\textwidth}
        \centering
            \begin{tikzpicture}[scale=0.6]
                \fill[blue!20] (3.8, 0) -- (5.5, 0) -- (5.5, 5.5) -- (3.8, 3.8);
                \fill[pattern=crosshatch, pattern color=red!45] (0, 0) -- (3.8, 0) -- (3.8, 3.8);
                \draw[black, thick, ->] (0,0) -- (0,6);
                \draw[black, thick, ->] (0,0) -- (7,0);
                \draw[black, thick] (5.5,0) -- (5.5,5.5);
                \draw[black, thick] (0,0) -- (5.5,5.5);
                \draw[black, thick] (3.8, 0) -- (3.8, 3.8);
                \draw[blue, very thick, ->] (1.8, 1.3) -- (3.6, 0.7);
                \draw[blue, very thick, ->] (1, 0.5) -- (3.6, 0.2);
                \draw[blue, very thick, ->] (2.4, 2) -- (3.6, 1.25);
                \filldraw[black] (4.1,0) circle (2.5pt) node[anchor=north]{$\lambda$};
                \draw[black, thick, dashed] (3.8, 3.8) -- (3.8, 6.5);
                \draw[black, thick, ->] (3.8, 5.8) -- (5.5, 5.8) node at (5.1, 6.3) {Large $s_1$};
            \end{tikzpicture}
            \subcaption{$s_1 \uparrow$ in the red region until it enters the blue region}
            \label{fig: step0_ub}
        \end{minipage}
        \begin{minipage}[b]{0.32\textwidth}
        \centering
            \begin{tikzpicture}[scale=0.6]
                \fill[blue!20] (3.8, 0) -- (3.8, 1) -- (4.4, 1) -- (4.4, 4.4) -- (5.5, 5.5) -- (5.5, 0);
                \fill[pattern=crosshatch, pattern color=red!45] (3.8, 1) -- (3.8, 3.8) -- (4.4, 4.4) -- (4.4, 1);
                \draw[black, thick, ->] (0,0) -- (0,6);
                \draw[black, thick, ->] (0,0) -- (7,0);
                \draw[black, thick] (5.5,0) -- (5.5,5.5);
                \draw[black, thick] (0,0) -- (5.5,5.5);
                \draw[black, thick] (3.8, 0) -- (3.8, 3.8);
                \draw[blue, very thick, ->] (3.9, 3.5) -- (4.3, 3);
                \draw[blue, very thick, ->] (3.9, 2.5) -- (4.3, 2.2);
                \draw[blue, very thick, ->] (3.9, 1.6) -- (4.3, 1.5);
                 \draw[black, thick] (4.4, 1) -- (4.4, 4.4);
                 \draw[black, thick] (3.8, 1) -- (4.4, 1);
                \filldraw[black] (4.1,0) circle (2.5pt) node[anchor=north]{$\lambda$};
                \draw[black, thick, dashed] (4.4, 4.4) -- (4.4, 6.5);
                \draw[black, thick, ->] (4.4, 5.8) -- (6.1, 5.8) node at (5.7, 6.3) {Large $s_1$};
                \draw[black, thick, dashed] (1.5, 1) -- (6.5, 1);
                \draw[black, thick, ->] (3.2, 1) -- (3.2, 0.1) node at (1.9, 0.6) {Small $s_2$};
            \end{tikzpicture}
            \subcaption{$s_1 \uparrow$ \& $s_2 \downarrow$ in the red region until it enters the blue region}
            \label{fig: step1a_ub_1}
        \end{minipage}
        \begin{minipage}[b]{0.32\textwidth}
        \centering
             \begin{tikzpicture}[scale=0.6]
                \fill[blue!20] (3.8, 0) -- (3.8, 1) -- (4.4, 1) -- (4.4, 1.4) -- (4.8, 0);
                \fill[pattern=crosshatch, pattern color=red!45] (4.8, 0) -- (5.5, 0) -- (5.5, 5.5) -- (4.4, 4.4) -- (4.4, 1.4);
                \draw[black, thick, ->] (0,0) -- (0,6);
                \draw[black, thick, ->] (0,0) -- (7,0);
                \draw[black, thick] (5.5,0) -- (5.5,5.5);
                \draw[black, thick] (0,0) -- (5.5,5.5);
                \draw[black, thick] (3.8, 0) -- (3.8, 3.8);
                \draw[blue, very thick, ->] (5.4, 0.2) -- (4.9, 0.7);
                \draw[blue, very thick, ->] (5.4, 1.2) -- (5.1, 1.5);
                \draw[blue, very thick, ->] (4.6, 4.2) -- (4.8, 3.6);
                \draw[blue, very thick, ->] (4.7, 3.3) -- (4.9, 3);
                \draw[blue, very thick, ->] (5.3, 4.6) -- (5.35, 3.5);
                \draw[blue, very thick, ->] (5.4, 3.3) -- (5.2, 2.4);
                \draw[blue, very thick, ->] (5.25, 2.1) -- (5, 1.8);
                \draw[blue, very thick, ->] (4.9, 1.55) -- (4.6, 1.35);
                 \draw[black, thick] (4.4, 1) -- (4.4, 4.4);
                 \draw[black, thick] (3.8, 1) -- (4.4, 1);
                 \draw[black, thick] (4.4, 1.4) -- (4.8, 0);
                \filldraw[black] (4.1,0) circle (2.5pt) node[anchor=north]{$\lambda$};
                \draw[black, thick, dashed] (3.2, 5.4) -- (4.8, 0);
                \draw[black, thick, ->] (3.4, 4.725) -- (2, 4.3) node[rotate=17] at (1.8, 5.3) {Small};
                \node[rotate=17] at (2.1, 4.8) {$s_1+s_2$};
            \end{tikzpicture}
            \subcaption{$s_1+s_2 \downarrow$ in the red region until it enters the blue region}
            \label{fig: step1a_ub_2}
        \end{minipage}}{
 Graphical representation of iterative SSC for lower and upper bound for $m=1$: The red hatched region is shown to have low steady-state probability leading to the collapse into the solid blue region. In addition, the arrows represents the drift -  $(ds_1/dt, ds_2/dt)$.
    \label{fig: upper_bound_ssc}}{}
\end{figure}

\subsection{General Case} \label{sec: general_case_upper_bound}
To present the required intermediate results in a concise manner, define the following:
\begin{align*}
    B_i = 18m d^{i-1}\sqrt{m n}\log n + 48m^3 \frac{n \log (d)^2}{d^{m -i+2}}+ \frac{n^{1-\gamma}}{d^{m -i}}\mathbbm{1}\left\{m >1\right\}. \numberthis \label{eq: B_j_upper_bound}
\end{align*}
Note that, $B_i$ consists of lower order terms. In particular, $B_i = o\left(\frac{mn \log d}{d^{m -i +1}}\right)$ by \eqref{eq: lower_order_terms}. Now, corresponding to Step 1a, we provide a high probability upper bound on $\bars_{m + 1}$ in the following lemma.
\begin{lemma}[Step 1a] \label{lemma: weak_bound_s_mplus1} There exists $\tilde{n}_1 \in \bbZ_+$ such that for all $n \geq \tilde{n}_1$, we have \label{lemma: weak_upper_bound_smplus1}
\begin{align*} 
    \P{\bars_{m+1} \geq b B_{m}} &\leq \left(\frac{1}{n}\right)^{(m\log n)/6}.
\end{align*}
\end{lemma}
\proof{Proof Sketch of Lemma \ref{lemma: weak_upper_bound_smplus1}} To prove the lemma, we first consider a family of functions defined as follows:
\begin{align*}
    U_i(\BFs) = \min\left\{\sum_{l=i+1}^b s_l - (m-i)n -o\left(\frac{2mn\log d}{d}b\right), n - \frac{2mn\log d}{d^{m-i+1}}(1-o(1))-s_i \right\} \quad \forall i \in [m].
\end{align*}
In Section~\ref{sec: special_case_upper_bound}, the idea was to show $\sum_{i=2}^b s_i = o(n)$ and $s_1 \approx \lambda$ by establishing SSC as in Fig~\ref{fig: upper_bound_ssc} (b) and (c) using the Lyapunov function given by \eqref{eq: min_lyapunov _function} and \eqref{eq: rest_of_the_ssc}. More generally, we expect $s_{m+1}$ to be small and so a natural generalization of \eqref{eq: min_lyapunov _function} is $U_m(\BFs)$. However, simply using $U_m(\BFs)$ does not suffice to obtain an appropriate bound of $\sum_{l=2}^b s_l \approx (m-1)n + o(n)$ similar to \eqref{eq: rest_of_the_ssc}. Thus, we do induction on $U_i(\BFs)$ to translate the bound on $U_m(\BFs)$ to $U_1(\BFs)$ which establishes an SSC similar to Fig.~\ref{fig: upper_bound_ssc} (b). In particular, $U_1(\BFs) = o(n)$ implies either $s_1 \approx \lambda + o(n)$ or $\sum_{l=2}^b s_l \approx (m-1)n + o(n)$, which are the correct orders for $s_1$ and $\sum_{l=2}^b s_l$. Now, it remains to translate this bound to a useful bound on $\sum_{l=1}^b s_l$ similar to Fig.~\ref{fig: upper_bound_ssc} (c). To implement this step, we consider a Lyapunov function similar to that of \eqref{eq: rest_of_the_ssc} given as follows:
\begin{align*}
    U_0(\BFs) = \sum_{l=1}^b s_l - mn-o\left(\frac{2mn\log d}{d}b\right).
\end{align*}
Now, using the fact that $U_1(\bbars)$ is small, we show that $U_0(\bbars)$ is small, which provides the required bound of $\sum_{i=1}^b s_i \approx mn+o\left(\frac{2mn\log d}{d}b\right)$. Finally, as $s_l = \Omega(n)$ for all $l \in [m]$ by the lower bounds established in Theorem~\ref{theo: lower_bound}, we immediately obtain
\begin{align*}
    \sum_{l=m+1}^b \bars_l = o\left(\frac{2mn\log d}{d}b\right).
\end{align*}
More precisely, we show that the lower order term $o\left(\frac{2mn\log d}{d}b\right)$ is equal to $b B_m$. This completes the proof of the lemma. Observe that the above bound is equivalent to Step 1a in Section~\ref{sec: special_case_upper_bound}.
\hfill $\square$
\endproof
Note that, the bound in Lemma \ref{lemma: weak_bound_s_mplus1} only makes sense if $b B_{m} = o(n)$. One can verify that $b B_{m} = o(n)$ by substituting $b \leq \log (n)^3$ and $d \geq \log (n)^3$. This implies that an appropriate upper bound on $b$ is essential for the proof to work. It turns out that $b \leq \log (n)^3$ works for all $d \geq \log (n)^3$. The result can be easily extended for $b \leq \min\{n^{0.5-\gamma}, d\}$. 

Now, the next lemma corresponds to Step 1b in Fig. \ref{fig: proof_sketch_ub} and proves an $o(1)$ high probability upper bound on $\sum_{l=m+2}^b \bars_l$ by using the high probability upper bound $\bars_{m+1} = o(n)$ that was proved in Lemma \ref{lemma: weak_bound_s_mplus1}.
\begin{lemma}[Step 1b]  There exists $\tilde{n}_2 \in \bbZ_+$ such that for all $n \geq \tilde{n}_2$, we have \label{lemma: very_light_tail}
\begin{align*}
    \P{\sum_{l=m +2}^b \bars_{l} \geq 1} \leq \left(\frac{1}{n}\right)^{(m\log n)/7}.
\end{align*}
\end{lemma}
As $\bars_{m+1} = o(n)$ w.h.p., we have $\bars_{m + 1} \leq n/2$ w.h.p. for large enough $n$. Thus, the probability that an incoming customer will be matched with a queue with at least $m + 1$ customers is at most $0.5^d \leq (1/n)^{\log (n)^2}$ w.h.p. as $d \geq \Omega(\log (n)^3)$. We analyze the drift of $\sum_{l=m+2}^b \bars_l$ to obtain a high probability upper bound on itself.  

Next, we employ this bound to improve the upper bound on $\bars_{m+1}$, corresponding to Step 1c in the following lemma:
\begin{lemma}[Step 1c] There exists $\tilde{n}_3 \in \bbZ_+$ such that for all $n \geq \tilde{n}_3$, we have \label{lemma: strong_upper_bound_smplus1}
\begin{align*}
     \P{\bars_{m+1} \geq B_{m}} &\leq \left(\frac{1}{n}\right)^{(m\log n)/8}.
\end{align*}
\end{lemma}
The proof of the lemma is a more fine tuned version of the proof of Lemma \ref{lemma: weak_bound_s_mplus1}. In particular, the proof of Lemma \ref{lemma: weak_bound_s_mplus1} uses the coarse bound $\sum_{l=m+1}^b s_l \leq bs_{m+1}$. Lemma \ref{lemma: very_light_tail} improves this bound by showing $\sum_{l=m+2}^b \bars_l = o(1)$ w.h.p., which implies $\sum_{l=m+1}^b s_l \leq s_{m+1}+o(1)$. With this refinement, we repeat the steps of the proof of Lemma \ref{lemma: weak_bound_s_mplus1}, to get Lemma \ref{lemma: strong_upper_bound_smplus1}.

The rest of the proof of Theorem \ref{theo: upper_bound} is based on induction. In particular, we inductively prove upper bounds on $\{\bars_i : i \in [m]\}$, where the induction on $i$ runs backward. The induction hypothesis is given as follows: There exists $\tilde{n}_{IH} \in \bbZ_+$ such that for all $n \geq \tilde{n}_{IH}$, we have
\begin{align}
       \P{\bars_{i} \geq n - 2m \frac{n\log d}{d^{m + i-1}}+B_i + 2(m -i)m \frac{n \log d}{d^{m -i +2}}+\sqrt{m n}\log n} \leq \left(\frac{1}{n}\right)^{(m\log n)/8-(m-i)}. \tag{IH2} \label{eq: induction_upper_bound}
\end{align}
The base case is to prove an upper bound on $\bars_{m}$ that corresponds to Step 1d in Fig. \ref{fig: proof_sketch_ub}. This is done in the lemma below.
\begin{lemma}[Step 1d] There exists $\tilde{n}_4 \in \bbZ_+$ such that for all $n \geq \tilde{n}_4$, we have \label{lemma: base_case_upper_bound}
\begin{align*}
    \P{\bars_{m} \geq n - 2m \frac{n\log d}{d}+ B_{m}} \leq \left(\frac{1}{n}\right)^{(m\log n)/8}.
\end{align*}
\end{lemma}
Further, we prove the induction step, corresponding to Steps $2$ to $m$ in Fig. \ref{fig: proof_sketch_ub}.
\begin{lemma}[Steps $2$ to $m$] \label{lemma: induction_step_upper_bound}
Assume that \eqref{eq: induction_upper_bound} holds true for $i+1$. Then, \eqref{eq: induction_upper_bound} holds true for $i$.
\end{lemma}
\proof{Proof Sketch of Lemma \ref{lemma: induction_step_upper_bound}}
Consider the following family of functions:
\begin{align*}
    W_i(\BFs) = s_i - n + \frac{2m n \log d}{d^{m-i+1}} + o\left(\frac{2m n \log d}{d^{m-i+1}}\right) \quad \forall i \in [m].
\end{align*}
Note that Lemma~\ref{lemma: base_case_upper_bound}, Lemma~\ref{lemma: strong_upper_bound_smplus1}, and Theorem~\ref{theo: lower_bound} together provides tight upper and lower bounds on $s_i$ for all $i \geq m$. Now, to obtain an upper bound on $s_{m-1}$, one can simply use itself as the Lyapunov function. In particular, the drift of $s_{m-1} \propto W_{m-1}(\BFs)$ depends on $\{s_{m-1}, s_m, s_{m+1}\}$. As we have a tight characterization of $s_{m}$ and $s_{m+1}$, we can analyze the drift of $s_{m-1}$ to obtain an upper bound on $s_{m-1}$. In this fashion, we continue to inductively obtain upper bounds on $s_i$ for $i < m$. This completes the proof of Lemma~\ref{lemma: induction_step_upper_bound}. \endproof


\endproof
We conclude this section by presenting the proof of Theorem \ref{theo: upper_bound} using the results outlined above.
\proof{Proof of Theorem \ref{theo: upper_bound}}
Consider the induction hypothesis defined in \eqref{eq: induction_upper_bound}. By setting $\tilde{n}_{IH} \geq \tilde{n}_4$, the base case is complete by Lemma \ref{lemma: base_case_upper_bound}. In addition, the induction step is complete by Lemma \ref{lemma: induction_step_upper_bound}. Thus, for all $n \geq \tilde{n}_{IH}$ and $i \in [m]$,
\begin{align*}
    &\P{\bars_i \geq n - 2m \frac{n\log d}{d^{m-i+1}} + 19m d^{i-1}\sqrt{m n}\log n + 49m^3 \frac{n \log (d)^2}{d^{m-i+2}}+\frac{n^{1-\gamma}}{d^{m-i}}\mathbbm{1}\left\{m > 1\right\}} \\
    \leq{}& \P{\bars_{i} \geq n - 2m \frac{n\log d}{d^{m + i-1}}+B_i + 2(m -i)m \frac{n \log d}{d^{m -i +2}}+\sqrt{m n}\log n} \\ 
    \leq{}& \left(\frac{1}{n}\right)^{(m\log n)/8-(m-i)} 
    \leq \left(\frac{1}{n}\right)^{(m\log n)/9},
\end{align*}
where the last inequality follows for all $n \geq \tilde{n}_6$ for some $\tilde{n}_6 \in \bbZ_+$ as $m=o(\log n)$. Thus, by setting $n_{UB} \overset{\Delta}{=} \max_{k \in \{2, 3, 6\}}\{\tilde{n}_k, \tilde{n}_{IH}\}$, and using Lemma \ref{lemma: very_light_tail} and \ref{lemma: strong_upper_bound_smplus1}, completes the proof of Theorem \ref{theo: upper_bound}. \hfill $\square$
\endproof
\begin{figure}[h!]
    \FIGURE{
    \begin{tikzpicture}
        \draw[black!60, fill=green!5, very thick] (-2.9, 0.5) -- (2.9, 0.5) -- (2.9, -5.35) -- (-2.9, -5.35) -- (-2.9, 0.5);
        \node[rectangle, draw=black!60, fill=red!5, very thick, minimum size=7mm, text width=52mm, align=center] at (0, 0) (1a) {\textbf{Step 1a:} $\bars_{m+1} \leq b B_m$};
        \node[rectangle, draw=black!60, fill=red!5, very thick, minimum size=7mm, text width=52mm, align=center] [below= 7mm of 1a] (1b) {\textbf{Step 1b:} $\sum_{l=m+2}^b \bars_{l}=o(1)$};
        \node[rectangle, draw=black!60, fill=red!5, very thick, minimum size=7mm, text width=52mm, align=center] [below= 7mm of 1b] (1c) {\textbf{Step 1c:} $\bars_{m+1}\leq B_m$};
        \node[rectangle, draw=black!60, fill=red!5, very thick, minimum size=7mm, text width=52mm, align=center] [below= 7mm of 1c] (1d) {\textbf{Step 1d:} $\bars_{m} \leq n - \frac{2m n \log d}{d}(1-o(1))$};
         \node[rectangle, draw=black!60, fill=green!5, very thick, minimum size=7mm, text width=65mm, align=center] [right= 15mm of 1a] (2) {\textbf{Step 2:} $\bars_{m-1}\leq n-\frac{2m n \log d}{d^2}(1-o(1))$};
         \node[rectangle, draw=black!60, fill=green!5, very thick, minimum size=7mm, text width=65mm, align=center] [below= 14.2mm of 2] (k) {\textbf{Step $k$:} $\bars_{k}\leq n-\frac{2m n \log d}{d^{m-k+1}}(1-o(1))$};
         \node[rectangle, draw=black!60, fill=green!5, very thick, minimum size=7mm, text width=65mm, align=center] [below= 14.2mm of k] (m) {\textbf{Step $m$:} $\bars_{1}\leq n-\frac{2m n \log d}{d^m}(1-o(1))$};
         \draw[black, thick, ->] (1d.east) -- ([xshift=7.5mm] 1d.east) -- ([xshift=-7.5mm] 2.west) -- (2.west);
         \fill[black] ([yshift=-4.5mm] 2.south) circle (0.06);
         \fill[black] ([yshift=-7mm] 2.south) circle (0.06);
         \fill[black] ([yshift=-9.5mm] 2.south) circle (0.06);
          \fill[black] ([yshift=-4.5mm] k.south) circle (0.06);
         \fill[black] ([yshift=-7mm] k.south) circle (0.06);
         \fill[black] ([yshift=-9.5mm] k.south) circle (0.06);
    \end{tikzpicture}}{
Outline of the proof to establish high probability upper bounds, where $B_m$ is defined in \eqref{eq: B_j_upper_bound}
    \label{fig: proof_sketch_ub}}{}
\end{figure}
\section{Simulations}
In this section, we simulate the load balancing model for various values of $n$ and $d$. To avoid confusion, all the simulated variables are denoted with a dagger in the super-script. A Python script to simulate this system is available online \cite{pod_load_balancing}.
\subsection{Setup}
Load balancing under power-of-$d$ choices routing algorithm is governed by a continuous time Markov chain (CTMC) with transition rates $q_{\BFs^1, \BFs^2}$ given by 
\begin{align*}
    q_{\BFs^1, \BFs^2} = \begin{cases}
        \lambda\left(\left(\frac{s_i^1}{n}\right)^d - \left(\frac{s_{i+1}^1}{n}\right)^d\right) \quad &\textit{for}  \quad \BFs^2 = \BFs^1 + \mathbf{e}_{i+1} \ \forall i \in [b-1] \\
        s_i^1 - s_{i+1}^1 \quad \quad &\textit{for} \quad \BFs^2 = \BFs^1 - \mathbf{e}_i \ \forall i \in [b] \\
        0 &\textit{otherwise}.
    \end{cases}
\end{align*}
We fix the initial state $\BFs(0) = \mathbf{0}_{b}$, buffer size $b=10$, arrival rate $\lambda = n - n^{0.6}$ and carry out the simulation for various values of $n$ and $d$. Let the simulated trajectory be denoted as $s^\dagger_i(t)$. For each tuple, $(n, d)$, the CTMC is simulated until it approximately reaches the steady state. This is verified by plotting the evolution of $\BFs^\dagger(t)$ as a function of time, as in Fig. \ref{fig: transience_m_5}. The mean estimate, i.e. $\E{\bbars^\dagger}$ is calculated by considering only the last 75\% of the data to ensure that a steady state has already been reached. 

The goal of this section is to compare the theoretical bounds with the simulated steady-state expectations. Denote the fixed point of the ODE as $\bar{s}_i = \left(\frac{\lambda}{n}\right)^{\frac{d^i-1}{d-1}}$. We investigate the error of the simulated mean from the fixed point of the ODE defined as follows:
\begin{align}
    \text{Error} = \frac{1}{n}\max_{i \in [b]}|\E{\bars^\dagger_i} - \bars_i|. \label{eq: simulated_error}
\end{align}
In particular, the leading order terms in Theorem~\ref{theo: informal} is approximately equal to $\bar{s}_i = \left(\frac{\lambda}{n}\right)^{\frac{d^i-1}{d-1}}$ as seen in \eqref{eq: fixed_point_before_d}. And so, we expect the simulated mean to match closely with $\bars_i$. We report this error for different values of $(n, d)$ in the next sub-section.

\subsection{Results}
We simulate the system for $n = (10^3, 10^4, 10^5)$ and set $d$ to be the greatest integer such that $d \leq (2mn^\gamma \log d)^{1/m}$ for $m = (1, 2, 3, 4)$. We summarize the results in Tables \ref{tab: mean_queue} and \ref{tab: mean_queue_gamma}.
\begin{table}[tbh!]
    \TABLE{Simulated mean, error as in \eqref{eq: simulated_error}, and $\max\{i : \max_{t \in \bbZ_+}s_i^\dagger(t) > 0\}$ for different values of $n$ and $m$ with $\gamma=0.4$. Cell corresponding to $s_i^\dagger$ such that $\max_{t \in \bbZ_+}s_i^\dagger(t) = 0$ are highlighted in blue. The notation $a\operatorname{E}b$ denotes $a \times 10^b$ for $a, b \in \bbZ_+$.
    \label{tab: mean_queue}}{
    \begin{tabular}{c|c|c||c|c|c|c|c||c||c}
         &  &  & \multicolumn{5}{c||}{Simulated mean: $\E{\bars_i^\dagger} / n$} &  &  \\ \hline
        $m$ & $n$ & $d$ & $i=1$ & $i=2$ & $i=3$ & $i=4$ & $i=5$ & Error \eqref{eq: simulated_error} & $i : s_{\geq i} = 0$ \\ \hline \hline
        1 & 1000 & 161 & 0.94 & 0.01 & \cellcolor{blue!30}0.00 & \cellcolor{blue!30}0.00 & \cellcolor{blue!30}0.00 & 7E-03 & 3 \\ \hline
        2 & 1000 & 12 & 0.94 & 0.43 & 0.00 & \cellcolor{blue!30}0.00 & \cellcolor{blue!30}0.00 & 6E-03 & 4 \\ \hline
        3 & 1000 & 5 & 0.94 & 0.67 & 0.14 & 0.00 & \cellcolor{blue!30}0.00 & 3E-03 & 5 \\ \hline
        4 & 1000 & 3 & 0.94 & 0.77 & 0.43 & 0.08 & 0.00 & 2E-03 & 6 \\ \hline\hline
        1 & 10000 & 493 & 0.98 & 0.00 & \cellcolor{blue!30}0.00 & \cellcolor{blue!30}0.00 & \cellcolor{blue!30}0.00 & 2E-03 & 3 \\ \hline
        2 & 10000 & 22 & 0.97 & 0.56 & 0.00 & \cellcolor{blue!30}0.00 & \cellcolor{blue!30}0.00 & 2E-03 & 4 \\ \hline
        3 & 10000 & 7 & 0.97 & 0.81 & 0.23 & 0.00 & \cellcolor{blue!30}0.00 & 5E-03 & 5 \\ \hline
        4 & 10000 & 4 & 0.97 & 0.88 & 0.58 & 0.11 & 0.00 & 7E-03 & 6 \\ \hline\hline
        1 & 100000 & 1456 & 0.99 & 0.00 & \cellcolor{blue!30}0.00 & \cellcolor{blue!30}0.00 & \cellcolor{blue!30}0.00 & 2E-04 & 3 \\ \hline
        2 & 100000 & 38 & 0.99 & 0.68 & 0.00 & \cellcolor{blue!30}0.00 & \cellcolor{blue!30}0.00 & 2E-03 & 4 \\ \hline
        3 & 100000 & 11 & 0.99 & 0.89 & 0.26 & 0.00 & \cellcolor{blue!30}0.00 & 2E-03 & 5 \\ \hline
        4 & 100000 & 6 & 0.99 & 0.93 & 0.65 & 0.08 & 0.00 & 2E-03 & 6 \\
        \hline
    \end{tabular}}{}
\end{table}
\begin{table}[tbh!]
    \TABLE{Simulated mean, error as in \eqref{eq: simulated_error}, and $\max\{i : \max_{t \in \bbZ_+}s_i^\dagger(t) > 0\}$ for different values of $\gamma$ and $m$ with $n = 10^4$. Cell corresponding to $s_i^\dagger$ such that $\max_{t \in \bbZ_+}s_i^\dagger(t) = 0$ are highlighted in blue. The notation $a\operatorname{E}b$ denotes $a \times 10^b$ for $a, b \in \bbZ_+$.
    \label{tab: mean_queue_gamma}}{
    \begin{tabular}{c|c|c||c|c|c|c|c||c||c}
         &  &  & \multicolumn{5}{c||}{Simulated mean: $\E{\bars_i^\dagger} / n$} &  &  \\ \hline
        $m$ & $\gamma$ & $d$ & $i=1$ & $i=2$ & $i=3$ & $i=4$ & $i=5$ & Error \eqref{eq: simulated_error} & $i : s_{\geq i} = 0$ \\ \hline \hline
        1 & 0.1 & 12 & 0.60 & 0.00 & \cellcolor{blue!30}0.00 & \cellcolor{blue!30}0.00 & \cellcolor{blue!30}0.00 & 3E-04 & 3 \\ \hline
        2 & 0.1 & 3 & 0.60 & 0.13 & 0.00 & \cellcolor{blue!30}0.00 & \cellcolor{blue!30}0.00 & 7E-05 & 4 \\ \hline
        3 & 0.1 & 2 & 0.60 & 0.22 & 0.03 & 0.00 & 0.00 & 5E-04 & 6 \\ \hline \hline
        1 & 0.3 & 161 & 0.94 & 0.00 & \cellcolor{blue!30}0.00 & \cellcolor{blue!30}0.00 & \cellcolor{blue!30}0.00 & 2E-04 & 3 \\ \hline
        2 & 0.3 & 12 & 0.94 & 0.43 & 0.00 & \cellcolor{blue!30}0.00 & \cellcolor{blue!30}0.00 & 4E-04 & 4 \\ \hline
        3 & 0.3 & 5 & 0.94 & 0.68 & 0.13 & 0.00 & \cellcolor{blue!30}0.00 & 2E-03 & 5 \\ \hline
        4 & 0.3 & 3 & 0.94 & 0.77 & 0.43 & 0.07 & 0.00 & 7E-04 & 6 \\ \hline \hline
        1 & 0.5 & 1456 & 0.99 & 0.01 & \cellcolor{blue!30}0.00 & \cellcolor{blue!30}0.00 & \cellcolor{blue!30}0.00 & 1E-02 & 3 \\ \hline
        2 & 0.5 & 38 & 0.99 & 0.66 & 0.00 & \cellcolor{blue!30}0.00 & \cellcolor{blue!30}0.00 & 2E-02 & 4 \\ \hline
        3 & 0.5 & 11 & 0.99 & 0.89 & 0.27 & 0.00 & \cellcolor{blue!30}0.00 & 8E-03 & 5 \\ \hline
        4 & 0.5 & 6 & 0.99 & 0.93 & 0.65 & 0.08 & 0.00 & 8E-03 & 6 \\ \hline \hline
        1 & 0.6 & 4190 & 1.00 & 0.04 & \cellcolor{blue!30}0.00 & \cellcolor{blue!30}0.00 & \cellcolor{blue!30}0.00 & 4E-02 & 3 \\ \hline
        2 & 0.6 & 64 & 1.00 & 0.80 & 0.00 & \cellcolor{blue!30}0.00 & \cellcolor{blue!30}0.00 & 3E-02 & 4 \\ \hline
        3 & 0.6 & 16 & 1.00 & 0.93 & 0.31 & 0.00 & \cellcolor{blue!30}0.00 & 2E-02 & 5 \\ \hline
        4 & 0.6 & 8 & 1.00 & 0.97 & 0.76 & 0.12 & 0.00 & 2E-02 & 6 \\ \hline\hline
        1 & 0.7 & 10000 & 1.00 & 0.12 & \cellcolor{blue!30}0.00 & \cellcolor{blue!30}0.00 & \cellcolor{blue!30}0.00 & 1E-01 & 3 \\ \hline
        2 & 0.7 & 108 & 1.00 & 0.83 & 0.01 & \cellcolor{blue!30}0.00 & \cellcolor{blue!30}0.00 & 2E-02 & 4 \\ \hline
        3 & 0.7 & 22 & 1.00 & 0.97 & 0.49 & 0.00 & \cellcolor{blue!30}0.00 & 4E-02 & 5 \\ \hline
        4 & 0.7 & 10 & 1.00 & 0.98 & 0.84 & 0.21 & 0.00 & 4E-02 & 6 \\
        \hline
    \end{tabular}}{}
\end{table}
In particular, we document the simulated mean $\{\E{\bars_i^\dagger} / n\}_{i=1}^5$, the error defined in \eqref{eq: simulated_error}, and the maximum queue lengths that were observed in the simulation, i.e. $\min\{i : s^\dagger_i(t) = 0 \ \forall t \in \bbZ_+\}$. 
Furthermore, to understand the transient behavior, we plot the evolution of $\BFs^\dagger(t)$ with time for $n = 10^5$ and $m = (2, 3, 4)$ in Fig. \ref{fig: transience_m_5} and \ref{fig: transience_m_2_4}.
\begin{figure}[bth!]
    \FIGURE{
    \begin{tikzpicture}[scale=1.75]
        \node[scale=0.45] at (1, 0.65) {\input{m_4.0_n_100000_gamma_0.4_d_6.pgf}};
        \draw[fill=brown] (3.5, -0.9) rectangle (3.75, -0.65);
        \draw[fill=brown] (3.5, -0.6) rectangle (3.75, -0.35);
        \draw[fill=brown] (3.5, -0.3) rectangle (3.75, -0.05);
        \draw[fill=brown] (3.5, 0) rectangle (3.75, 0.25);
        \draw[fill=brown] (3.5, 0.3) rectangle (3.75, 0.55);
        \draw[fill=brown] (3.5, 0.6) rectangle (3.75, 0.85);
        \draw[fill=brown] (3.5, 0.9) rectangle (3.75, 1.15);
        \draw[fill=brown] (3.5, 1.2) rectangle (3.75, 1.45);
        \draw[fill=brown] (3.5, 1.5) rectangle (3.75, 1.75);
        \draw[fill=brown] (3.5, 1.8) rectangle (3.75, 2.05);
        \draw[fill=brown] (3.8, -0.9) rectangle (4.05, -0.65);
        \draw[fill=brown] (3.8, -0.6) rectangle (4.05, -0.35);
        \draw[fill=brown] (3.8, -0.3) rectangle (4.05, -0.05);
        \draw[fill=brown] (3.8, 0) rectangle (4.05, 0.25);
        \draw[fill=brown] (3.8, 0.3) rectangle (4.05, 0.55);
        \draw[fill=brown] (3.8, 0.6) rectangle (4.05, 0.85);
        \draw[fill=brown] (3.8, 0.9) rectangle (4.05, 1.15);
        \draw[fill=brown] (3.8, 1.2) rectangle (4.05, 1.45);
        \draw[fill=brown] (3.8, 1.5) rectangle (4.05, 1.75);
        \draw[fill=brown] (4.1, -0.9) rectangle (4.35, -0.65);
        \draw[fill=brown] (4.1, -0.6) rectangle (4.35, -0.35);
        \draw[fill=brown] (4.1, -0.3) rectangle (4.35, -0.05);
        \draw[fill=brown] (4.1, 0) rectangle (4.35, 0.25);
        \draw[fill=brown] (4.1, 0.3) rectangle (4.35, 0.55);
        \draw[fill=brown] (4.1, 0.6) rectangle (4.35, 0.85);
        \draw[fill=brown] (4.1, 0.9) rectangle (4.35, 1.15);
        \draw[fill=brown] (4.4, -0.9) rectangle (4.65, -0.65);
        \draw[fill=brown] (4.4, -0.6) rectangle (4.65, -0.35);
        \draw[thick, black] (3.3, -0.775) circle (0.125);
        \draw[thick, black] (3.3, -0.475) circle (0.125);
        \draw[thick, black] (3.3, -0.175) circle (0.125);
        \draw[thick, black] (3.3, 0.125) circle (0.125);
        \draw[thick, black] (3.3, 0.425) circle (0.125);
        \draw[thick, black] (3.3, 0.725) circle (0.125);
        \draw[thick, black] (3.3, 1.025) circle (0.125);
        \draw[thick, black] (3.3, 1.325) circle (0.125);
        \draw[thick, black] (3.3, 1.625) circle (0.125);
        \draw[thick, black] (3.3, 1.925) circle (0.125);
        \draw[ultra thin, black!60] (5.5, -0.925) -- (3.455, -0.925) -- (3.455, -0.625) -- (5.5, -0.625);
        \draw[ultra thin, black!60] (3.455, -0.625) -- (3.455, 2.075);
        \draw[ultra thin, black!60] (3.455, -0.325) -- (5.5, -0.325);
         \draw[ultra thin, black!60] (3.455, -0.025) -- (5.5, -0.025);
         \draw[ultra thin, black!60] (3.455, 0.275) -- (5.5, 0.275);
         \draw[ultra thin, black!60] (3.455, 0.575) -- (5.5, 0.575);
         \draw[ultra thin, black!60] (3.455, 0.875) -- (5.5, 0.875);
         \draw[ultra thin, black!60] (3.455, 1.175) -- (5.5, 1.175);
         \draw[ultra thin, black!60] (3.455, 1.475) -- (5.5, 1.475);
         \draw[ultra thin, black!60] (3.455, 1.775) -- (5.5, 1.775);
         \draw[ultra thin, black!60] (3.455, 2.075) -- (5.5, 2.075);
         \node at (6.3, 2) {Legend};
         \draw[ultra thick, color=blue1] (5.7, 1.62) -- (6.3, 1.62) node[black] at (6, 1.73) {$1$};
         \node at (6.63, 1.64) {$s_1^\dagger / n$};
         \draw[ultra thick, color=orange] (5.7, 1.3) -- (6.3, 1.3) node[black] at (6, 1.41) {$2$};
         \node at (6.63, 1.32) {$s_2^\dagger / n$};
         \draw[ultra thick, color=terquise] (5.7, 0.98) -- (6.3, 0.98) node[black] at (6, 1.09) {$3$};
         \node at (6.63, 1) {$s_3^\dagger / n$};
         \draw[ultra thick, color=red1] (5.7, 0.66) -- (6.3, 0.66) node[black] at (6, 0.77) {$4$};
         \node at (6.63, 0.68) {$s_4^\dagger / n$};
         \draw[ultra thick, color=pink] (5.7, 0.34) -- (6.3, 0.34) node[black] at (6, 0.45) {$5$};
         \node at (6.63, 0.36) {$s_5^\dagger / n$};
         \draw[very thick, dotted, color=brown1] (5.7, 0.02) -- (6.3, 0.02);
         \node at (6.63, 0.04) {$\bbars / n$};
         \draw[thick, black] (5.6, 2.15) -- (6.95, 2.15) -- (6.95, -0.15) -- (5.6, -0.15) -- (5.6, 2.15);
         \node at (3, 2) {$1$};
         \node at (3, 1.8) {$2$};
         \node at (3, 1.05) {$3$};
         \node at (3, -0.5) {$4$};
         \node at (3, -0.75) {$5$};
    \end{tikzpicture}}{
    Evolution of the load balancing CTMC with $n=10^5$, $\gamma=0.4$, and $d=6$. 
    \label{fig: transience_m_5}}{}
\end{figure}
\begin{figure}
    \FIGURE{
\begin{minipage}[b]{0.48\textwidth}
        \centering
        \pgfplotsset{
        width=0.7\textwidth,
        height=0.75\textwidth}
        \begin{tikzpicture}[scale=1.75]
        \node[scale=0.42] at (5, 1.3) {\input{m_3.0_n_100000_gamma_0.4_d_11.pgf}};
        \node at (5.625, 1.85) {Legend};
         \draw[very thick, dotted, color=brown1] (5.7, 1.3) -- (6.3, 1.3);
         \node at (6.57, 1.32) {$\bbars / n$};
         \draw[ultra thick, color=blue1] (4.55, 1.62) -- (5.15, 1.62) node[black] at (4.85, 1.73) {$1$};
         \node at (5.42, 1.64) {$s_1^\dagger / n$};
         \draw[ultra thick, color=orange] (4.55, 1.3) -- (5.15, 1.3) node[black] at (4.85, 1.41) {$2$};
         \node at (5.42, 1.32) {$s_2^\dagger / n$};
         \draw[ultra thick, color=terquise] (5.7, 1.62) -- (6.3, 1.62) node[black] at (6, 1.73) {$3$};
         \node at (6.57, 1.64) {$s_3^\dagger / n$};
         \draw[thick, black] (4.45, 2) -- (6.8, 2) -- (6.8, 1.15) -- (4.45, 1.15) -- (4.45, 2);
         \node at (6.73, 2.9) {$1$};
         \node at (6.73, 2.6) {$2$};
         \node at (6.73, 0.7) {$3$};
         \end{tikzpicture}
            \label{fig: transience_m_4}
        \end{minipage}
        \begin{minipage}[b]{0.48\textwidth}
        \centering
        \pgfplotsset{
        width=0.7\textwidth,
        height=0.75\textwidth}
        \begin{tikzpicture}[scale=1.75]
        \node[scale=0.42] at (5, 2) {\input{m_2.0_n_100000_gamma_0.4_d_38.pgf}};
        \node at (5.1, 1.95) {Legend};
         \draw[very thick, dotted, color=brown1] (4.55, 0.98) -- (5.15, 0.98);
         \node at (5.42, 1) {$\bbars / n$};
         \draw[ultra thick, color=blue1] (4.55, 1.62) -- (5.15, 1.62) node[black] at (4.85, 1.73) {$1$};
         \node at (5.42, 1.64) {$s_1^\dagger / n$};
         \draw[ultra thick, color=orange] (4.55, 1.3) -- (5.15, 1.3) node[black] at (4.85, 1.41) {$2$};
         \node at (5.42, 1.32) {$s_2^\dagger / n$};
         \draw[thick, black] (4.45, 2.15) -- (5.7, 2.15) -- (5.7, 0.8) -- (4.45, 0.8) -- (4.45, 2.15);
         \node at (6.73, 3.6) {$1$};
         \node at (6.73, 2.7) {$2$};
         \end{tikzpicture}
            \label{fig: transience_m_2}
        \end{minipage}}{
    Load balancing CTMC with $n=10^5$, $\gamma=0.4$, and $d=38$ (left) and $d=6$ (right).
    \label{fig: transience_m_2_4}}{}
\end{figure}
We now summarize the takeaways from Tables \ref{tab: mean_queue} and \ref{tab: mean_queue_gamma}, and Figures \ref{fig: transience_m_5} and \ref{fig: transience_m_2_4}.

As observed in Tables \ref{tab: mean_queue} and \ref{tab: mean_queue_gamma}, the fixed point $\mathbf{\bars}$ closely approximates the stationary mean even for $n$ as small as $10^3$. As expected, the approximation is tight for a wide range of values of $\gamma$ and $m$ except when $\gamma$ is large and $m$ is small. For example, we observe a non-trivial error of $0.1$ for $(\gamma, m) = (0.7, 1)$. These values of the parameters are out of the permissible range of $\gamma < 0.5$ as in Theorem~\ref{theo: informal}. Next, as observed in Tables \ref{tab: mean_queue} and \ref{tab: mean_queue_gamma}, no incoming customers are rejected due to a finite waiting space for $b \geq 8$. This suggests that the assumption $b = O(\log(n)^3)$ is not fundamental to the model. It is merely an artefact of the proof. Lastly, observe that $s_i \approx 0$ for all $i \geq m+1$ which verifies the bounds obtained in Theorem~\ref{theo: informal}.

As observed in Figures \ref{fig: transience_m_5} and \ref{fig: transience_m_4}, $\BFs^\dagger(t)$ stays close to $\bbars$ for all $t \geq 500$. In particular, $\BFs^\dagger(t) \approx \bbars$ w.h.p. in the steady-state as established in Theorem~\ref{theo: informal}. Lastly, one can also observe in the figures that the fluctuations around $\bars_i$ increase with $i$ which aligns with the error bound $(d^{i-1}\sqrt{n}\log(n))$ established in Theorems  \ref{theo: lower_bound} and \ref{theo: upper_bound}, which increases with $i$.

\section{Conclusion and Future Work}
In this paper, we characterized the performance of Power-of-$d$ choices routing algorithm for the sub-Halfin-Whitt regime in the load balancing model. We showed that if $d$ grows polynomially with $n$, then the jobs experience a finite delay. On the other hand, if $d$ grows only as Poly-Log$(n)$, then the jobs experience infinite asymptotic delay. In particular, we characterized the delay for Power-of-$d$ with $d \in [\Omega(\log (n)^3), n]$ and $\gamma \in (0, 0.5)$. Future work is to similarly understand the performance of Power-of-$d$ choices for other many-server-heavy-traffic regimes.
\section{Acknowledgement}
We thank Dr. Debankur Mukherjee for insightful discussions that helped in proving the result. In addition, the illustration of the fixed point, as shown in Fig. \ref{fig: fixed_point} is inspired from a figure in \cite{mukherjee2018universality}. Illustration of the asymptotic regime as a 2D graph in Fig. \ref{fig: regimes} is inspired by Dr. Lei Ying. 
\newpage

\bibliographystyle{informs2014} 
\bibliography{references.bib} 
\newpage
\renewcommand{\theHsection}{A\arabic{section}}
\begin{APPENDICES}
\section{Preliminary Lemmas}
\subsection{Taylor Series Based Inequalities}
In this section, we present a few inequalities based on Taylor's series expansion which will be useful later to bound some of the terms. The proofs of these lemmas are deferred to Appendix~\ref{app: taylor_series}.
\begin{lemma} \label{lemma: power_of_d_going_to_zero}
Let $f(d)$ and $r$ be such that $d f(d) \rightarrow 0$  and 
$r \log d/d \rightarrow 0$ as $d \rightarrow \infty$. Then, there exists $d_0$ such that for all $d \geq d_0$, we have
\begin{align*}
   \left(1 - r\frac{\log d}{d}+f(d)\right)^{\lfloor d \rfloor} \leq \frac{2}{d^r}.
\end{align*}
\end{lemma}
In the application of the above lemma, we always pick $r = O(m)$ and so we have $r \log d/d \leq \log n/d \rightarrow 0$ as $d = \Omega(\log (n)^3)$.
\begin{lemma} \label{lemma: power_of_d_not_going_to_zero}
Let $f(d)$ be such that $d f(d) \rightarrow 0$ as $d \rightarrow \infty$ and $f(d) \geq 0$ for $d \geq d_1$ for some $d_1 \in \bbR_+$. Then, there exists $d_2 \in \bbR_+$ such that for all $d \geq d_2$, we have
\begin{align*}
    1-df(d) \leq \left(1-f(d)\right)^{\lfloor d \rfloor} \leq 1-\lfloor d \rfloor f(d)+\frac{1}{2} d^2f(d)^2.
\end{align*}
\end{lemma}
\begin{lemma} \label{lemma: simplified_ssc_identity} There exists $n_a \in \bbZ_+$ such that for all $n \geq n_a$, we have
\begin{align*}
    \left(\frac{n}{n+\sqrt{m n}\log n}\right)^{\sqrt{m n} (\log n)/2} \leq \frac{1}{n^{(m \log n)/4}}.
\end{align*}
\end{lemma}
For the rest of the appendix, we consider $n$ to be large enough such that all the results in this section holds true. 
\subsection{Iterative State Space Collapse} \label{app: iterative_ssc}
We start by formally defining the drift of a Lyapunov function. Let $q_{\BFs, \BFs'}$ be the rate at which the CTMC transitions from $\BFs$ to $\BFs'$. 
\begin{definition} \label{def: lyapunov_function}
Consider a Lyapunov function $V: \calS \rightarrow \bbR$ and define the drift of $V$ at state $\BFs$ as
\begin{align*}
    \Delta V(\BFs) = \sum_{\BFs' \in \calS, \BFs' \neq \BFs} q_{\BFs, \BFs'} \left(V(\BFs') - V(\BFs)\right).
\end{align*}
\end{definition}
\textbf{State Space Collapse:} If the drift of $V(\cdot)$ is such that $\Delta V(\BFs) \leq -\gamma$ when $V(\BFs)\geq B$ for some $\gamma, B \in \bbR_+$, then, one can obtain high probability tail bounds on $V(\bbars)$ \cite{bertsimas_tail_bounds} that depends on $B, \gamma$ and the properties of the CTMC. Intuitively, every time the CTMC jumps to a state $\BFs$ such that $V(\BFs) \geq B$, due to a strong drift ($\Delta V(\BFs) \leq -\gamma$), the CTMC will quickly jump back to a state such that $V(\BFs) \leq B$. Loosely speaking, the stationary probability $\P{V(\bbars) \geq B + j}$ decreases exponentially in terms of $j$. This implies that the underlying CTMC $\{\BFs(t): t \geq 0\}$ ``collapses'' to a subset of the state space $\{\BFs \in \calS : V(\BFs) \leq B + j\}$ for a large enough $j \in \bbR_+$. 

\textbf{Iterative State Space Collapse:} The main challenge in obtaining SSC is to show that the Lyapunov function $V(\cdot)$ exhibits negative drift when $V(\BFs) \geq B$. The authors in \cite{lei_coxian_2_sub_halfin_whitt} ingeniously showed that if $\calE \subseteq \calS$ such that $\P{\bbars \notin \calE} \approx 0$, then, it suffices to show negative drift for $\BFs \in \calE : V(\BFs) \geq B$. Thus, by exploiting the properties of the stationary distribution of $\{\BFs(t): t \geq 0\}$, one needs to show negative drift for only a subset of the state space. We state this as a lemma below.
\begin{lemma} \label{lemma: iterative_ssc}
Consider a Lyapunov function $V : \calS \rightarrow \bbR$ such that $V(\BFs) \geq D$ for all $\BFs \in \calS$ for some $D \in \bbR$. Consider
\begin{align*}
    \nu_{\max} := \max_{\BFs, \BFs' \in \calS, q_{\BFs, \BFs'} > 0} \big| V(\BFs') - V(\BFs)| < \infty,
\end{align*}
and define
\begin{align*}
    q_{\max} := \max_{\BFs \in \calS} \sum_{\BFs' \in \calS: V(\BFs) < V(\BFs')} q_{\BFs, \BFs'}.
\end{align*}
Assume that there exists a set $\calE$ with $B>D$, $\gamma>0$, $\delta \geq 0$ such that the following conditions are satisfied.
\begin{itemize}
    \item $\Delta V(\BFs) \leq -\gamma$ when $V(\BFs) \geq B$ and $\BFs \in \calE$,
    \item $\Delta V(\BFs) \leq \delta$ when $V(\BFs) \geq B$ and $\BFs \notin \calE$.
\end{itemize}
Then, 
\begin{align*}
    \P{V(\bbars) \geq B + 2\nu_{\max} j} \leq \alpha^j + \beta \P{\bbars \notin \calE} \quad \forall j \in \bbZ_+, \numberthis \label{eq: iterative_ssc}
\end{align*}
with
\begin{align*}
    \alpha = \frac{q_{\max}\nu_{\max}}{q_{\max}\nu_{\max} + \gamma} \quad \textit{and} \quad \beta = \frac{\delta}{\gamma} + 1.
\end{align*}
\end{lemma} 
The above lemma is obtained by directly using \cite[Lemma 10]{lei_coxian_2_sub_halfin_whitt} with $V' = V-D$ as \cite[Lemma 10]{lei_coxian_2_sub_halfin_whitt} requires the Lyapunov function to be non-negative. It'll be helpful to note that proof of almost all the following lemmas follows a four-part template. 
\begin{enumerate}
    \item Define a Lyapunov function $V(\BFs)$ depending on what is required to be proved.
    \item Show that $\Delta V(\BFs) \leq -\sqrt{m n}\log n$  for all $\BFs \in \calE$ such that $V(\BFs) \geq 0$, where $\P{\bbars \in \calE} \approx 1$. Lemma \ref{lemma: power_of_d_going_to_zero} and Lemma \ref{lemma: power_of_d_not_going_to_zero} will be useful in completing this step.
    \item Apply Lemma \ref{lemma: iterative_ssc} to obtain a high probability bound on $V(\bbars)$. In all the instances of application of Lemma \ref{lemma: iterative_ssc}, we use $B=0$, $j=\sqrt{m n}(\log n)/2$, $\nu_{\max}=1$, $q_{\max} = \delta = n$, and $\gamma=\sqrt{m n}\log n$. This also implies that $\beta \leq \sqrt{n}$ for large enough $n$. Also note that, Lemma \ref{lemma: simplified_ssc_identity} will be useful in simplifying the expression obtained. In particular, we have
    \begin{align*}
        \P{V(\bbars) \geq  \sqrt{m n}\log n} &\leq \left(\frac{n}{n+\sqrt{m n}\log n}\right)^{\sqrt{m n}(\log n)/2} + \sqrt{n}\P{\bbars \notin \calE} \\
        &\leq \left(\frac{1}{n}\right)^{(m \log n)/4} + \sqrt{n}\P{\bbars \notin \calE}.
    \end{align*}
    \item Translate the high probability bound on $V(\bbars)$ to the required bound by employing basic results in probability like union bound, law of total probability, etc.
\end{enumerate}
\section{Proof of Lemmas for Lower Bound} \label{app: lb_lemmas}
\proof{Proof of Lemma \ref{lemma: outer_base_case} (Base Case)}
To prove the lemma, we consider the following Lyapunov function:
\begin{align*}
    V_1(\BFs) = n - \frac{2m n \log d}{d} - \frac{2n\log d}{d^{2}} -\sqrt{ m n}\log n - s_1.
\end{align*}
Now, we analyze the drift of $V_1(\BFs)$ when $V_1(\BFs) \geq 0$. Thus, we have $s_1 \leq n - \frac{2m n \log d}{d} - \frac{2n\log d}{d^{2}}  - \sqrt{m n}\log n$. Now, the drift is given as follows:
\begin{align*}
    \Delta V_1(\BFs) &= s_1 - s_2 - \lambda\left(1 - \left(\frac{s_1}{n}\right)^{\lfloor d \rfloor}\right) \\
     &\leq n - \frac{2m n \log d}{d}- \frac{2n\log d}{d^{2}} -\sqrt{m n}\log n - \lambda\left(1 - \left(1 - \frac{ 2m\log d}{d}\right)^{\lfloor d \rfloor}\right) \\
    &\overset{(a)}{\leq} n - \frac{2m n \log d}{d}- \frac{2n\log d}{d^{2}} -\sqrt{m n}\log n - \lambda\left(1 - \frac{2}{d^{2m}}\right) \\
    &\leq  - \frac{2m n \log d}{d}- \frac{2n\log d}{d^{2}} -\sqrt{m n}\log n + \frac{2n}{d^{2m}} + n^{1-\gamma} \\
    &\overset{(b)}{\leq}  -\sqrt{m n}\log n,
\end{align*}
where $(a)$ follows for all $n \geq n_0$ for some $n_0 \in \bbZ_+$ by Lemma \ref{lemma: power_of_d_going_to_zero}. Next, $(b)$ follows by the following observations: $\frac{2m n \log d}{d} \geq \frac{2m n \log d}{d^{m}} = n^{1-\gamma}$, and $\frac{2n}{d^{2m}} \leq \frac{2n}{d^2}$. Thus, we have $V_1(\BFs) \leq -\sqrt{m n}\log n$ for all $n \geq n_0$. By Lemma \ref{lemma: iterative_ssc}, we get
\begin{align*}
    \P{V_1(\bbars) \geq \sqrt{m n}\log n} \leq \left(\frac{n}{n+\sqrt{m n}\log n}\right)^{\sqrt{m n}(\log n)/2} \leq \frac{1}{n^{(m\log n)/4}},
\end{align*}
where the last inequality follows by Lemma \ref{lemma: simplified_ssc_identity}. This completes the proof. \hfill $\square$
\endproof
\proof{Proof of Lemma \ref{lemma: outer_induction_weak_induction_step}}
Define a family of functions $\{V_{ik}(\BFs): i \in [k]\}$ for $k \geq 2$ as follows:
\begin{align*}
    V_{ik}(\BFs) &= s_i - n + 3im \frac{n\log d}{d^{k-i+1}} \quad \forall i \in [k-1]\\
    V_{kk}(\BFs) &= n - 3km  \frac{n\log d}{d} - s_k. 
\end{align*}
Now, using the above family of functions, we define the Lyapunov functions $\{L_{lk}: l \in [k-1]\}$ as follows:
\begin{align*}
    L_{lk}^{(1)} &= V_{kk} - \sum_{j=l+1}^{k-1} V_{jk}, \quad L_{lk}^{(2)} = V_{lk} \quad \forall l \in [k-1] \\
    L_{lk} &= \min\left\{L_{lk}^{(1)}, L_{lk}^{(2)} \right\} \quad \forall l \in [k-1].
\end{align*}
To prove the lemma, we make use of the following claim:
\begin{claim} There exists $n_{c1} \in \bbZ_+$ such that for all $n\geq n_{c1}$, we have \label{claim: induction_for_weak_induction_step}
\begin{align*}
    \P{L_{1k}(\bbars) \geq \sqrt{m n}\log n} \leq \left(\frac{1}{n}\right)^{(m\log n)/4-4(k-1)m-(k-1) }.
\end{align*}
\end{claim}
We defer the proof of the claim to Appendix \ref{app: claim_lb} and continue with the proof of Lemma \ref{lemma: outer_induction_weak_induction_step}. We analyze the drift of $L_{0k}^{(1)}(\BFs)$ resulting in a high probability upper bound on $L_{0k}^{(1)}(\BFs)$ which will imply a high probability lower bound on $\bbars_k$. We start by analyzing the drift of $L_{0k}^{(1)}(\BFs)$ when $L_{0k}^{(1)}(\BFs) \geq 0$ and $\BFs \in \calC_{1, k}^{(1)} \cap \bigcap_{l=1}^{k-1} \calD_l^{(1)}$ where
\begin{subequations} \label{eq: sets_weak_induction_base_case}
\begin{align}
    \calC_{l, k}^{(1)} &= \left\{L_{lk} \leq \sqrt{m n}\log n\right\} \quad \forall l \in [k-1] \\
    \calD_l^{(1)} &= \left\{s_l \geq n - \frac{5m n\log d}{2d^{k -l }}\right\} \quad \forall l \in [k-1].
\end{align}
\end{subequations}
First, by using that $\BFs \in \calC_{1, k}^{(1)}$, we get a useful upper bound on $s_1$ as follows:
\begin{align*}
    L_{0k}^{(1)}(\BFs) \geq 0 &\Rightarrow V_{kk}(\BFs) - \sum_{j=2}^k V_{jk}(\BFs) \geq V_{1k}(\BFs) \\
    &\overset{(*)}{\Rightarrow} V_{1k}(\BFs) \leq \sqrt{m n}\log n \quad \textit{as } \BFs \in \calC_{1, k}^{(1)} \\
    &\Rightarrow s_1 \leq n - 3m \frac{n\log d}{d^k} + \sqrt{m n}\log n \leq n - 5m \frac{n\log d}{2d^k}, \numberthis \label{eq: upper_bound_s1}
\end{align*}
where $(*)$ follows as $L_{1k} = \min\left\{V_{kk}(\BFs) - \sum_{j=2}^k V_{jk}(\BFs), V_{1k}(\BFs)\right\} = V_{1k}$ as $V_{kk}(\BFs) - \sum_{j=2}^k V_{jk}(\BFs) \geq V_{1k}(\BFs)$ and so $\BFs \in \calC_{1, k}^{(1)}$ implies $L_{1k}(\BFs) = V_{1k}(\BFs) \leq \sqrt{mn}\log n$. Further, the last inequality holds because $m \frac{n \log d}{2d^k} \geq m \frac{n \log d}{2d^{m}} = n^{1-\gamma}/4 \geq \sqrt{m n}\log n$ for all $n \geq n_1^{(1)}$ for some $n_1^{(1)} \in \bbZ_+$ independent of $k$.
Next, we get a useful upper bound on $s_k$ as follows:
\begin{align*}
    L_{0k}^{(1)}(\BFs) \geq 0 &\Rightarrow V_{kk}(\BFs) - \sum_{j=1}^{k-1} V_{jk}(\BFs) \geq 0 \\ &\Rightarrow  s_k \leq n - 3km \frac{n \log d}{d} - \sum_{j=1}^{k-1} V_{jk}(\BFs) \\
    &\overset{(*)}{\Rightarrow} s_k \leq n - 3km \frac{n \log d}{d} + 2.5m n \log d\sum_{l=1}^{k-1} \frac{1}{d^l} \quad \textit{as } \BFs \in \bigcap_{l=1}^{k-1} \calD_l^{(1)} \\
     &\Rightarrow s_k \leq n - 3m \frac{n \log d}{d}, \numberthis \label{eq: upper_bound_sk}
\end{align*}
where $(*)$ follows as $\calD_l^{(1)}$ implies $s_l \geq n - \frac{5mn \log d}{2d^{k-l}}$ which further implies $V_{lk}(\BFs) \geq - \frac{5mn \log d}{2d^{k-l}} + 3lm\frac{n \log d}{d^{k-l+1}} \geq - \frac{5mn \log d}{2d^{k-l}}$. Further,
 the last assertion follows by using the bound $1/d^l \leq 1/d$ for $l \in \{1, 2, \hdots, k-1\}$. Now, the drift is given as follows:
\begin{align*}
\Delta L_{0k}^{(1)}(\BFs)
    ={}& s_1 - s_{k+1} - \lambda\left(1 - \left(\frac{s_k}{n}\right)^{\lfloor d \rfloor}\right) \\
    \overset{(a)}{\leq}{}& n - 5m \frac{n\log d}{2d^{k}}-\lambda\left(1 - \left(1 - 3m \frac{\log d}{d}\right)^{\lfloor d \rfloor}\right) \\
    \overset{(b)}{\leq}{}& n - 5m \frac{n\log d}{2d^{k}}-\lambda\left(1 - \frac{2}{d^{3m}}\right) \\
    \leq{}& - 5m \frac{n\log d}{2d^{k}}   + \frac{2n}{d^{3m}} + n^{1-\gamma}  \\
    \overset{(c)}{\leq}{}& - m \frac{n\log d}{2d^{k}} + \frac{2n}{d^{3m}} \\
    \overset{(d)}{\leq}{}& - m \frac{n\log d}{4d^{k}}
    \leq - m \frac{n\log d}{4d^{m }}\\
    \leq{}& -\frac{1}{8}n^{1-\gamma} \overset{(e)}{\leq} -\sqrt{m n}\log n,
\end{align*}
where $(a)$ follows by lower bounding $s_{k+1}$ by zero and using the upper bounds on $s_1$ and $s_k$ given by \eqref{eq: upper_bound_s1} and \eqref{eq: upper_bound_sk} respectively. Next, $(b)$ follows by Lemma \ref{lemma: power_of_d_going_to_zero}. Further, $(c)$ follows as $2m \frac{n\log d}{d^{k}} \geq 2m \frac{n\log d}{d^{m}} = n^{1-\gamma}$. Now, $(d)$ follows as $m \frac{n\log d}{4d^{k}} \geq m \frac{n\log d}{4d^{m}}$. Thus, there exists $n_1^{(2)} \in \bbZ_+$ independent of $k$ such that for all $n \geq n_1^{(2)}$, we have $m \frac{n\log d}{4d^{k}} \geq 2n/d^{3m}$ as $d$ increases with $n$. Lastly, $(e)$ follows as there exists $n_1^{(3)} \in \bbZ_+$ such that for all $n \geq n_1^{(3)}$, we have $n^{1-\gamma}/8 \geq \sqrt{m n}\log n$. Thus, for all $n \geq \max_{k \in [3]}\{n_1^{(k)}\}$, we have $\Delta L_{0k}^{(1)}(\BFs) \leq -\sqrt{m n}\log n$ when  $L_{0k}^{(1)}(\BFs) \geq 0$ and $\BFs \in \calC_{1, k}^{(1)} \cap \bigcap_{l=1}^{k-1} \calD_l^{(1)}$. Now, by using Lemma \ref{lemma: iterative_ssc}, we get a high probability upper bound on $L_{0k}^{(1)}(\bbars)$ as follows:
\begin{align*}
    &\P{L_{0k}^{(1)}(\bbars) \geq \sqrt{m n}\log n} \\
    \leq{}& \left(\frac{n}{n+\sqrt{m n}\log n}\right)^{(\sqrt{m n}\log n)/2} + \sqrt{n}\P{\bbars \notin \calC_{1, k}^{(1)} \cap \bigcap_{l=1}^{k-1} \calD_l^{(1)}} \\
     \overset{(a)}{\leq}{}& \left(\frac{1}{n}\right)^{(m\log n)/4} + \sqrt{n}\left(\P{\bbars \notin \calC_{1, k}^{(1)}} + \sum_{l=1}^{k-1} \P{\bbars \notin \calD_l^{(1)}}\right) \\
     \overset{(b)}{\leq}{}&  \left(\frac{1}{n}\right)^{(m\log n)/4} + \sqrt{n}\left(\frac{1}{n}\right)^{(m\log n)/4-4(k-1)m-(k-1)} + \sqrt{n}(k-1)\left(\frac{1}{n}\right)^{(m\log n)/4-4(k-1)m} \\
    \overset{(c)}{\leq}{}& \left(\frac{1}{n}\right)^{(m\log n)/4-4(k-1)m-k} \leq \left(\frac{1}{n}\right)^{(m\log n)/4-4(k-1)m-m},
\end{align*}
where $(a)$ follows by Lemma \ref{lemma: simplified_ssc_identity}. Next, $(b)$ follows by bounding  $\P{\bbars \notin \calC_{1, k}^{(1)}}$ using Claim \ref{claim: induction_for_weak_induction_step} and $\P{\bbars \notin \calD_l^{(1)}}$ is bounded by using \eqref{eq: outer_induction} and noting that the lower order terms in \eqref{eq: outer_induction} are upper bounded by $m n \log d/(2d^{k-l})$. In particular, similar to \eqref{eq: lower_order_terms} there exists $n_1^{(4)} \in \bbZ_+$ independent of $k$ such that for all $n \geq n_1^{(4)}$, we have
\begin{align}
    m \frac{n\log d}{2d^{k-l}} \geq 16m^3 \frac{n \log(d)^2}{d^{k-l+1}}+4m d^{l-1}\sqrt{m n}\log n \quad \forall l \in [k-1]. \label{eq: converting_to_leading_term}
\end{align}
Lastly, $(c)$ holds for all $n \geq n^{(5)}_1 \in \bbZ_+$ for some $n \geq n^{(5)}_1 \in \bbZ_+$. Now, we use the above probability bound to obtain the required result for the lemma as follows.
\begin{align*}
    \left\{L_{0k}^{(1)}(\BFs) \geq \sqrt{m n}\log n\right\} &= \left\{kn - 3km \frac{n\log d}{d} - 3m n \log d\sum_{l=1}^{k-1} \frac{l}{d^{k-l+1}}-\sum_{l=1}^k s_l \geq \sqrt{m n}\log n\right\} \\
    &\supseteq \left\{s_k \leq n - 3km \frac{n\log d}{d} - 3m n \log d\sum_{l=1}^{k-1} \frac{l}{d^{k-l+1}}-\sqrt{m n}\log n\right\} \\
    &\supseteq \left\{s_k \leq n - 6m^2 \frac{n\log d}{d}\right\},
\end{align*}
where the last assertion follows as there exists $n_1^{(6)} \in \bbZ_+$ independent of $k$ such that for all $n \geq n_1^{(6)}$, we have
\begin{align*}
    3m n \log d\sum_{l=1}^{k-1} \frac{l}{d^{k-l+1}}+\sqrt{m n}\log n &\overset{(a)}{\leq} 3m  n \log d\left(\frac{m}{d^2} + \frac{m^2}{d^3}\right)+\sqrt{m n}\log n
    \\
    &\overset{(b)}{\leq} 6m^2 \frac{n \log d}{d^2} + \sqrt{m n}\log n \overset{(c)}{\leq} 3m^2 \frac{n \log d}{d},
\end{align*}
where $(a)$ follows by using the bounds $k \leq m$ and $1/d^l \leq 1/d^3$ for all $l \in \{3, \hdots, k-1\}$. Next, $(b)$ follows as $m^2 / d^3 \leq m / d^2$ for large enough $n$. Lastly, $(c)$ follows as $\sqrt{m n}\log n \leq n^{1-\gamma} = 2mn\log d/d^m \leq 2mn\log d/d \leq 2m^2 n \log d/d$ and $6/d^2 \leq 1/d$ for large enough $n$ as $d\rightarrow \infty$ as $n \rightarrow \infty$. Thus, by defining $n_1 \overset{\Delta}{=} \max_{k \in [6]}\left\{n_1^{(k)}, n_{c1}\right\}$, the proof is complete. \hfill $\square$
\endproof
\proof{Proof of Lemma \ref{lemma: remaining_weak_outer_induction}}
We will prove this using induction. For some $j \in \{1, \hdots, k-1\}$, the induction hypothesis is given as follows. There exists $n^{(1)}_2 \in \bbZ_+$ such that for all $n \geq n^{(1)}_2$, we have
\begin{align}
    \P{\bars_{j} \leq n - 3\left(2m+k-j\right)m\frac{n \log d}{d^{k-j+1}} - \sqrt{m n}\log n} \leq \left(\frac{1}{n}\right)^{\frac{1}{4}m\log n - 4\left(k-\frac{3}{4}\right)m-(k-j)}. \tag{IH2} \label{eq: inner_induction}
\end{align}
Note that, the above expression would directly imply the required result as shown later in \eqref{eq: 9msquare}. The base case $(j=k)$ is satisfied for all $n \geq n_1$ by Lemma \ref{lemma: outer_induction_weak_induction_step}. Now, we show the induction step for $j$. Define the family of Lyapunov functions $\{W_l(\BFs) : l \in [k-1]\}$ for $k \geq 2$ as follows:
\begin{align*}
    W_l(\BFs) = n - 3\left(2m+k-l\right)m\frac{n \log d}{d^{k-l+1}} - s_l.
\end{align*}
We analyze the drift of $W_j(\BFs)$ when $W_j(\BFs) \geq 0$ and $\BFs \in \calC^{(2)}_{j+1} \cap \calD^{(2)}_{j-1}$ where
\begin{align*}
    \calC^{(2)}_l &= \left\{W_l(\BFs) \leq \sqrt{m n}\log n\right\} \quad \forall l \in [k-1] \\
    \calD_l^{(2)} &= \left\{s_l \geq n - \frac{5m n\log d}{2d^{k-l}}\right\} \quad \forall l \in [k-1].
\end{align*}
The drift is given as follows:
\begin{align*}
    &\Delta W_j(\BFs) \\
    ={}& s_j - s_{j+1} - \lambda\left(\left(\frac{s_{j-1}}{n}\right)^{\lfloor d \rfloor} - \left(\frac{s_j}{n}\right)^{\lfloor d \rfloor}\right) \\
    \overset{(a)}{\leq}{}&  3\left(2m+k-j-1\right)m\frac{n \log d}{d^{k-j}} + \sqrt{m n}\log n \\
    &- \lambda\left(\left(1 - \frac{5m\log d}{2d^{k-j+1}}\right)^{\lfloor d \rfloor}-\left(1- 3\left(2m+k-j\right)m\frac{ \log d}{d^{k-j+1}}\right)^{\lfloor d \rfloor}\right) \\
    \overset{(b)}{\leq}{}&3\left(2m+k-j-1\right)m\frac{n \log d}{d^{k-j}} + \sqrt{m n}\log n \\
    &- \lambda\left( -\frac{5m\log d}{2d^{k-j}}+3\left(2m+k-j\right)m\frac{\lfloor d \rfloor \log d}{d^{k-j+1}}-9\left(2m+k-j\right)^2m^2\frac{ \log (d)^2}{2d^{2k-2j}}\right) \\
    \leq{}& - m\frac{n \log d}{2d^{k-j}}+ 9m^2 \frac{n\log d}{d^{k-j+1}} + \sqrt{m n}\log n+3\left(2m+k-j\right)m\frac{ n^{1-\gamma}\log d}{d^{k-j}}+9\left(2m+k-j\right)^2m^2\frac{ n\log (d)^2}{2d^{2k-2j}} \\
    \overset{(c)}{\leq}{}& - m\frac{n \log d}{4d^{k-j}} \overset{(d)}{\leq} -\sqrt{m n}\log n.
\end{align*}
where $(a)$ follows by bounding $s_{j-1}, s_j$ and $s_{j+1}$ using $\BFs \in \calC_{j+1}^{(2)} \cap \calD_{j-1}^{(2)}$ and $W_j(\BFs) \geq 0$. Next, $(b)$ follows by Lemma \ref{lemma: power_of_d_not_going_to_zero}.

Further, $(c)$ follows as there exists $n_2^{(2)} \in \bbZ_+$, independent of $j$ and $k$ such that for all $n \geq n_2^{(2)}$, we have
\begin{align*}
    m\frac{n \log d}{16d^{k-j}} &\overset{(c_1)}{\geq} m\frac{n \log d}{16d^{m}} \geq \frac{1}{32}n^{1-\gamma} \overset{(c_2)}{\geq} \sqrt{m n}\log n \\
    m\frac{n \log d}{16d^{k-j}} &\overset{(c_3)}{\geq} m\frac{n \log d}{16d^{k-j}} \times 144m n^{-\gamma} \overset{(c_4)}{\geq} 3\left(2m+k-j\right)m\frac{ n^{1-\gamma}\log d}{d^{k-j}} \\
    m\frac{n \log d}{16d^{k-j}} &\overset{(c_5)}{\geq} m\frac{n \log d}{16d^{k-j}} \times \frac{648m^3\log d}{d} \overset{(c_6)}{\geq} 9\left(2m+k-j\right)^2m^2\frac{ n\log (d)^2}{2d^{2k-2j}}, \\
    m\frac{n \log d}{16d^{k-j}} &\geq m\frac{n \log d}{16d^{k-j}} \times \frac{144m}{d} = 9m^2 \frac{\log d}{d^{k-j+1}}
\end{align*}
where $(c_1)$ follows as $k-j \leq k \leq m$, and $(c_2)$ follows as $\gamma < 0.5$. Next, $(c_3)$ follows as $\gamma >0$ and $m \leq \log n$, and $(c_4)$ follows as $3\left(2m+k-j\right)\leq 9m$. Lastly, by \eqref{eq: m}, $(c_5)$ follows as $m ^3\log d /d \leq \gamma^3\log (n)^3(1+o(1))/(d\log (d)^2) \rightarrow 0$ as $n \rightarrow \infty$, and $(c_6)$ follows as $3\left(2m+k-j\right) \leq 9m$. 

Further, $(d)$ follows as there exists $n_2^{(3)} \in \bbZ_+$, independent of $j$ and $k$ such that for all $n \geq n_2^{(3)}$ we have $m n \log d / 4d^{m} \geq \sqrt{m n}\log n$. Thus, for all $n \geq \max\{n_2^{(2)}, n_2^{(3)}\}$, we have $\Delta W_j(\BFs) \leq -\sqrt{m n}\log n$ when $W_j(\BFs) \geq 0$ and $\BFs \in \calC^{(2)}_{j+1} \cap \calD^{(2)}_{j-1}$. Now, by using Lemma \ref{lemma: iterative_ssc}, we get a high probability upper bound on $\bars_j$ as follows:
\begin{align*}
    &\P{\bars_j \leq n - 9 m^2 \frac{n \log d}{d^{k-j+1}}} \\
    \overset{(a)}{\leq}{}& \P{W_j(\bbars) \geq \sqrt{m n}\log n} \\
    \overset{(b)}{\leq}{}& \left(\frac{n}{n+\sqrt{m n}\log n}\right)^{(\sqrt{m n}\log n)/2} + \sqrt{n}\P{\bbars \notin \calC^{(2)}_{j+1} \cap \calD^{(2)}_{j-1}} \\
    \overset{(c)}{\leq}{}& \left(\frac{1}{n}\right)^{(m\log n)/4} +\sqrt{n}\left(\P{\bbars \notin \calC^{(2)}_{j+1}} + \P{\bbars \notin \calD^{(2)}_{j-1}}\right)  \\
    \overset{(d)}{\leq}{}& \left(\frac{1}{n}\right)^{(m\log n)/4} +\sqrt{n}\left(\frac{1}{n}\right)^{(m\log n)/4 - 4(k-1)m-m-(k-j)} + \sqrt{n}\left(\frac{1}{n}\right)^{(m\log n)/4 - 4(k-1)m} \\
    \overset{(e)}{\leq}{}& \left(\frac{1}{n}\right)^{(m\log n)/4 - 4(k-1)m-m-(k-j+1)},
\end{align*}
where $(a)$ follows for all $n \geq n_2^{(4)}$ for some $n_2^{(4)} \in \bbZ_+$ as 
\begin{align*}
    &n-3(2m +k-j)m \frac{n \log d}{d^{k-j+1}}-\sqrt{m n}\log n \\
    \geq{}&  n-9m^2 \frac{n \log d}{d^{k-j+1}}+3m \frac{n \log d}{d^{k-j+1}}-\sqrt{m n}\log n \\
    \geq{}& n-9m^2 \frac{n \log d}{d^{k-j+1}}+3m \frac{n \log d}{d^{m}}-\sqrt{m n}\log n \geq n-9m^2 \frac{n \log d}{d^{k-j+1}}. \numberthis \label{eq: 9msquare}
\end{align*}
Next, $(b)$ follows by Lemma \ref{lemma: iterative_ssc} and $(c)$ follows by Lemma \ref{lemma: simplified_ssc_identity}. Now, $(d)$ follows by upper bounding $\P{\bbars \notin \calC^{(2)}_{j+1}}$ using \eqref{eq: inner_induction}. In addition, $\P{\bbars \notin  \calD_{j-1}^{(2)}}$ is upper bounded by using \eqref{eq: outer_induction} and noting that the lower order terms in \eqref{eq: outer_induction} are upper bounded by $m n \log d/(2d^{k-j+1})$ using \eqref{eq: converting_to_leading_term} for all $n \geq n_2^{(5)}$ for some $n_2^{(5)} \in \bbZ_+$. Lastly, $(e)$ follows for all $n \geq n_2^{(6)}$ for some $n_2^{(6)} \in \bbZ_+$. Now, by considering $n_2 \overset{\Delta}{=} n_2^{(1)} \geq \max_{k \in [6]}\{n_2^{(k)}, n_1\}$, the induction step is complete. Thus, we have
\begin{align*}
    \P{\bars_l \leq n - 9 m^2 \frac{n \log d}{d^{k-l+1}}} \leq \left(\frac{1}{n}\right)^{(m\log n)/4 - 4(k-1)m-m-(k-l+1)} \leq \left(\frac{1}{n} \right)^{(m\log n)/4-4(k-0.5)m} \ \forall l \in [k].
\end{align*}
This completes the proof of the lemma. \hfill $\square$
\endproof
\proof{Proof of Lemma \ref{lemma: base_case_of_induction_hypothesis}}
Define the family of functions $\{W_{ik}(\BFs): i \in [k]\}$ as follows:
\begin{align*}
    W_{ik}(\BFs) &= s_{i} - n + 2m \frac{n\log d}{d^{k-i+1}} +  (2i+1)d^{i-1}\sqrt{m n}\log n + 3(i+1)m\frac{n\log d}{d^{k-i+2}} \quad \forall i \in \{1, \hdots, k-1\} \\
    W_{kk}(\BFs) &= n - 2m \frac{n \log d}{d} - 2k d^{k-1}\sqrt{m n}\log n -10m^2 \frac{n \log d}{d^2} - s_k.
\end{align*}
Now, we use the above family of functions to define the Lyapunov functions $\{Z_{ik}: i \in [k]\}$ as follows:
\begin{subequations}
\label{eq: z_lyapunov_functions}
\begin{align}
    Z_{ik}^{(1)} &= W_{kk} - \sum_{l=i+1}^{k-1} W_{lk}, \quad Z_{ik}^{(2)} = W_{ik} \quad \forall i \in [k-1] \\
    Z_{ik} &= \min\left\{Z_{ik}^{(1)}, Z_{ik}^{(2)}\right\} \quad \forall i \in [k-1].
\end{align}
\end{subequations}
To prove the lemma, we make use of the following claim:
\begin{claim} There exists $n_{c2} \in \bbZ_+$ such that for all $n \geq n_{c2}$, we have \label{claim: base_case_of_induction_hypothesis}
\begin{align*}
    \P{Z_{1k}(\bbars) \geq \sqrt{m n}\log n} \leq \left(\frac{1}{n}\right)^{(m \log n)/4 - 4(k-0.5)m-(k-1)}.
\end{align*}
\end{claim}
We defer the proof of claim to Appendix \ref{app: claim_lb} and continue with the proof of Lemma \ref{lemma: base_case_of_induction_hypothesis}. We analyze the drift of $Z_{0k}^{(1)}(\BFs)$ when $Z_{0k}^{(1)}(\BFs) \geq 0$ and $\BFs \in \calC^{(3)}_{1,k} \cap \bigcap_{l=1}^{k-1}\calD^{(3)}_l$ where
\begin{subequations} \label{eq: sets_induction_step_base_case}
\begin{align}
    \calC_{l, k}^{(3)} &= \left\{Z_{l,k}(\BFs) \leq \sqrt{m n}\log n\right\} \quad \forall l \in [k-1] \\
    \calD_l^{(3)} &= \left\{s_l \geq n - \frac{9m^2 n \log d}{d^{k-l+1}}\right\} \quad \forall l \in [k-1].
\end{align}
\end{subequations}
We first obtain a useful upper bound on $s_1$ as follows.
\begin{align*}
    Z_{0k}^{(1)}(\BFs) \geq 0 &\Rightarrow W_{kk}(\BFs) - \sum_{l=2}^{k-1} W_{lk}(\BFs) \geq W_{1k}(\BFs) \\
    &\Rightarrow W_{1k}(\BFs) \leq \sqrt{m n}\log n \quad \textit{as } \BFs \in \calC_{1,k}^{(3)} \\
    &\Rightarrow s_1 \leq n - 2m \frac{n\log d}{d^{k}}-  2\sqrt{m n}\log n - 6m\frac{n\log d}{d^{k+1}}. \numberthis \label{eq: precise_upper_bound_s1}
\end{align*}
Next, we will get a useful upper bound on $s_k$ as follows:
\begin{align*}
    Z_{0k}^{(1)}(\BFs) \geq 0 &\Rightarrow s_k \leq n - 2m \frac{n \log d}{d} - 2k d^{k-1}\sqrt{m n}\log n -10m^2 \frac{n \log d}{d^2} - \sum_{l=1}^{k-1} W_{lk}(\BFs) \\
    &\overset{(*)}{\Rightarrow} s_k \leq n - 2m \frac{n \log d}{d} - 2k d^{k-1}\sqrt{m n}\log n -10m^2 \frac{n \log d}{d^2} + 9m^2 n \log d\sum_{l=2}^k \frac{1}{d^l} \\
    &\overset{(**)}{\Rightarrow} s_k \leq n - 2m \frac{n \log d}{d}, \numberthis \label{eq: precise_upper_bound_sk}
\end{align*}
where $(*)$ follows as $\BFs \in \bigcap_{l=1}^{k-1}\calD_l^{(3)}$ and $(**)$ follows as there exists $n_3^{(1)} \in \bbZ_+$ such that for all $n \geq n_3^{(1)}$ we have
\begin{align*}
    9m^2 n \log d\sum_{l=2}^k \frac{1}{d^l} \leq 9m^2 n \log d \left(\frac{1}{d^2} + \frac{m}{d^3}\right) \leq 10m^2 \frac{n \log d}{d^2}.
\end{align*}
where the last inequality follows by noting that $m / d \rightarrow 0$ as $n \rightarrow \infty$. Now, the drift of $Z_{0k}^{(1)}(\BFs)$ is given as follows:
\begin{align*}
\Delta Z_{0k}^{(1)}(\BFs)
    ={}& s_1 - s_{k+1} - \lambda\left(1 - \left(\frac{s_k}{n}\right)^{\lfloor d \rfloor}\right) \\
    \overset{(a)}{\leq}{}& n - 2m \frac{n\log d}{d^{k}}-  2\sqrt{m n}\log n - 6m\frac{n\log d}{d^{k+1}} - \lambda\left(1 - \left(1 - 2m \frac{\log d}{d} \right)^{\lfloor d \rfloor}\right) \\
    \overset{(b)}{\leq}{}& n - 2m \frac{n\log d}{d^{k}}-  2\sqrt{m n}\log n - 6m\frac{n\log d}{d^{k+1}} - \lambda\left(1-\frac{2}{d^{2m}}\right) \\
    \leq{}& - 2m \frac{n\log d}{d^{k}}-  2\sqrt{m n}\log n - 6m\frac{n\log d}{d^{k+1}} + n^{1-\gamma} + \frac{2n}{d^{2m}} \\
    \overset{(c)}{\leq}{}& -  2\sqrt{m n}\log n \leq -\sqrt{m n}\log n.
\end{align*}
where $(a)$ follows as $s_{k+1} \geq 0$, $s_1$ is upper bounded as in \eqref{eq: precise_upper_bound_s1}, and $s_k$ is upper bounded as in \eqref{eq: precise_upper_bound_sk}. Next, by Lemma \ref{lemma: power_of_d_going_to_zero}, there exists $n_3^{(2)} \in \bbZ_+$ independent of $k$ such that for all $n \geq n_3^{(2)}$, $(b)$ follows. Lastly, $(c)$ follows by noting that
\begin{align*}
    2m \frac{n\log d}{d^{k}} &\geq 2m \frac{n\log d}{d^{m}} = n^{1-\gamma}, \\
    6m\frac{n\log d}{d^{k+1}} &\geq 6m\frac{n\log d}{d^{m+1}} \geq 6m\frac{n\log d}{d^{2m}} \geq \frac{2n}{d^{2m}},
\end{align*}
where the last set of inequalities follows as $m \geq 2$ and there exists $n_3^{(3)} \in \bbZ_+$ such that for all $n \geq n_3^{(3)}$, we have $\log d \geq 1$. Thus, for all $n \geq \max_{k \in [3]}\left\{n_3^{(k)}\right\}$, we have $\Delta Z_{0k}^{(1)}(\BFs) \leq -\sqrt{m n}\log n$ when $Z_{0k}^{(1)}(\BFs) \geq 0$ and $\BFs \in \calC_{1k}^{(3)} \cap \bigcap_{l=1}^{k-1} \calD_l^{(3)}$. Now, by using Lemma \ref{lemma: iterative_ssc}, we obtain a high probability upper bound on $Z_{0k}^{(1)}(\bbars)$ as follows:
\begin{align*}
    &\P{Z_{0k}(\bbars) \geq \sqrt{m n}\log n} \\
    \leq{}& \left(\frac{n}{n+\sqrt{m n}\log n}\right)^{(\sqrt{m n}\log n)/2}+\sqrt{n} \P{\bbars \notin \calC_{1k}^{(3)} \cap \bigcap_{l=1}^{k-1} \calD_l^{(3)}} \\
    \overset{(a)}{\leq}{}& \left(\frac{1}{n}\right)^{(m\log n)/4} + \sqrt{n} \left(\P{\bbars \notin \calC_{1k}^{(3)}} + \sum_{l=1}^{k-1} \P{\bbars \notin \calD_l^{(3)}}\right) \\
    \overset{(b)}{\leq}{}& \left(\frac{1}{n}\right)^{(m\log n)/4}  + \sqrt{n}\left(\frac{1}{n}\right)^{(m\log n)/4-4(k-0.5)m-(k-1)} + \sqrt{n}(k-1)\left(\frac{1}{n}\right)^{(m\log n)/4-4(k-0.5)m} \\
    \overset{(c)}{\leq}{}&  \left(\frac{1}{n}\right)^{(m\log n)/4-4(k-0.5)m-k} \leq \left(\frac{1}{n}\right)^{(m\log n)/4-4(k-0.5)m-m},
\end{align*}
where $(a)$ follows by Lemma \ref{lemma: simplified_ssc_identity}. Next, $(b)$ follows by upper bounding $\P{\bbars \notin \calC_{1k}^{(3)}}$ by Claim \ref{claim: base_case_of_induction_hypothesis}, and upper bounding $\P{\bbars \notin \calD_l^{(3)}}$ for all $n \geq n_2$ by Lemma \ref{lemma: remaining_weak_outer_induction}. Lastly, $(c)$ follows for all $n \geq n_3^{(4)}$ for some $n_3^{(4)} \in \bbZ_+$ independent of $k$. Now, to complete the proof, note that
\begin{align*}
    &\left\{Z_{0k}(\bbars) \geq \sqrt{m n}\log n\right\} \\
    ={}& \bigg\{kn - 2m n \log d \sum_{l=1}^k \frac{1}{d^l}-\sqrt{m n}\log n\sum_{l=1}^{k} (2l+1)d^{l-1}-10m^2 \frac{n\log d}{d^2} - 3m n \log d \sum_{l=1}^{k-1} \frac{(l+1)}{d^{k-l+2}}-\sum_{l=1}^k \bars_l \\
    &\geq \sqrt{m n}\log n\bigg\} \\
    \supseteq{}& \left\{\bars_k \leq n - 2m \frac{n\log d}{d} - 3m d^{k-1} \sqrt{m n}\log n - 12 m^2 \frac{n\log d}{d^2}\right\},
\end{align*}
where the last assertion follows as there exists $n_3^{(5)} \in \bbZ_+$ independent of $k$ such that for all $n \geq n_3^{(5)}$, we have
\begin{align*}
    3m n \log d \sum_{l=1}^{k-1} \frac{(l+1)}{d^{k-l+2}} &\leq 3m^3 \frac{n \log d}{d^3} \leq m^2 \frac{n\log d}{2d^2} \\
    2m n\log d\sum_{l=2}^k \frac{1}{d} &\leq 2m n\log d\left(\frac{1}{d^2}+\frac{m}{d^3}\right) \leq \frac{3mn\log d}{d^2} \overset{(*)}{\leq} 3m^2 \frac{n\log d}{2d^2} \\
    \sqrt{m n}\log n\sum_{l=1}^{k} (2l+1)d^{l-1} + \sqrt{m n}\log n &\leq (2m +1)\left(d^{k-1}+m d^{k-2}+1\right) \sqrt{m n}\log n \\
    &\leq 3m d^{k-1} \sqrt{m n}\log n,
\end{align*}
where we used the fact that $k \leq m$ and $m^2/d \rightarrow 0$ as $n\rightarrow \infty$. Note that, $(*)$ follows as $m \geq 2$. Now, by defining $n_3 \overset{\Delta}{=} \max_{k \in [5]}\{n_3^{(k)}, n_2, n_{c2}\}$, the proof is complete. \hfill $\square$
\endproof
\proof{Proof of Lemma \ref{lemma: induction_step}}
We will prove this using induction. For some $j \in [k]$, the induction hypothesis is given as follows: there exists $n_4^{(1)} \in \bbZ_+$ such that for all $n \geq n_4^{(1)}$, we have
\begin{align}
    &\P{\bars_{j} \leq  n - 2m \frac{n \log d}{d^{k-j+1}} - (4m-j) d^{j-1}\sqrt{m n}\log n - 16(k-j+1)m^2 \frac{n \log (d)^2}{d^{k-j+2}}-\sqrt{m n}\log n} \nonumber \\
    \leq{}& \left(\frac{1}{n}\right)^{(m\log n)/4-4(k-0.25)m-(k-j)}. \tag{IH3} \label{eq: precise_inner_induction}
\end{align}
The base case $(j=k)$ is satisfied for all $n \geq n_3$ by Lemma \ref{lemma: base_case_of_induction_hypothesis} as
\begin{align*}
    (4m-k) d^{k-1}\sqrt{m n}\log n + 16m^2 \frac{n \log (d)^2}{d^{2}}+\sqrt{m n}\log n \geq 3m d^{k-1}\sqrt{m n}\log n+12m^2 \frac{n \log d}{d^{2}}.
\end{align*}
Now, we show the induction step corresponding to $j \in [k]$. Consider the Lyapunov function
\begin{align*}
    \tilde{W}_j(\BFs) = n - 2m \frac{n \log d}{d^{k-j+1}} - (4m-j) d^{j-1}\sqrt{m n}\log n -16(k-j+1)m^2 \frac{n \log (d)^2}{d^{k-j+2}} - s_j.
\end{align*}
We analyze the drift of $\tilde{W}_j(\BFs)$ when $\tilde{W}_j(\BFs) \geq 0$ and $\BFs \in \calC^{(4)}_{j+1} \cap \calD^{(4)}_{j-1}$, where
\begin{align*}
    \calC^{(4)}_l &= \left\{\tilde{W}_l \leq \sqrt{m n}\log n\right\} \quad \forall l \in [k] \\
    \calD^{(4)}_l &= \left\{s_{l} \geq n - 9m^2 \frac{n\log d}{d^{k-l+1}}\right\} \quad \forall l \in [k].
\end{align*}
Now, the drift is given as follows:
\begin{align*}
    &\Delta \tilde{W}_j(\BFs) \\
    ={}& s_j - s_{j+1} - \lambda\left(\left(\frac{s_{j-1}}{n}\right)^{\lfloor d \rfloor} - \left(\frac{s_j}{n}\right)^{\lfloor d \rfloor}\right) \\
    \overset{(a)}{\leq}{}& 2m \frac{n\log d}{d^{k-j}} + (4m-j-1) d^{j}\sqrt{m n}\log n +16(k-j)m^2 \frac{n \log (d)^2}{d^{k-j+1}}+\sqrt{m n}\log n\\
    &- \lambda\left(1 - 9m^2\frac{\log d}{d^{k-j+2}}\right)^{\lfloor d \rfloor}\hspace{-0.54pt}+\lambda\left(1- 2m \frac{\log d}{d^{k-j+1}}- (4m-j)\sqrt{m} d^{j-1}\frac{\log n}{\sqrt{n}} -16(k-j+1)m^2 \frac{\log (d)^2}{d^{k-j+2}}\right)^{\lfloor d \rfloor} \\
    \overset{(b)}{\leq}{}& 2m \frac{n \log d}{d^{k-j}} + (4m-j-1) d^{j}\sqrt{m n}\log n +16(k-j)m^2 \frac{n \log (d)^2}{d^{k-j+1}}+\sqrt{m n}\log n \\
    &-\lambda\left(- 9m^2\frac{\log d}{d^{k-j+1}}+ 2m \frac{\lfloor d \rfloor\log d}{d^{k-j+1}}+ (4m-j)\sqrt{m} d^{j-1}\frac{\lfloor d \rfloor\log n}{\sqrt{n}} +16(k-j+1)m^2 \frac{\lfloor d \rfloor\log (d)^2}{d^{k-j+2}}\right) \\
    &+\frac{3\lambda}{2}\left(4m^2 \frac{ \log (d)^2}{d^{2k-2j}}+ (4m-j)^2 m d^{2j}\frac{\log (n)^2}{n} +256(k-j+1)^2 m^4 \frac{\log (d)^4}{d^{2k-2j+2}}\right) \\
    \leq{}& -d^{j}\sqrt{m n}\log n-7m^2\frac{n \log (d)^2}{d^{k-j+1}}+\sqrt{m n}\log n+6m^2 \frac{n\log (d)^2}{d^{2k-2j}}+ 24m^3 d^{2j}\log (n)^2 +384m^6 \frac{n\log (d)^4}{d^{2k-2j+2}} \\
    &+ 2m \frac{n^{1-\gamma}\log d}{d^{k-j}}+ 4\sqrt{m}m d^{j}n^{0.5-\gamma}\log n +16m^3\frac{n^{1-\gamma}\log (d)^2}{d^{k-j+1}}+2m \frac{n\log d}{d^{k-j+1}}+ (4m-j) d^{j-1}\sqrt{mn}\log n \\
    &+16(k-j+1)m^2 \frac{n\log (d)^2}{d^{k-j+2}} \\
    \overset{(c)}{\leq}{}&-d^{j}\sqrt{m n}\log n-m^2\frac{n \log (d)^2}{d^{k-j+1}}+\sqrt{m n}\log n+ 24m^3 d^{2j}\log (n)^2 +384m^6 \frac{n\log (d)^4}{d^{2k-2j+2}}+ 2m \frac{n^{1-\gamma}\log d}{d^{k-j}} \\
    &+ 4\sqrt{m}m d^{j}n^{0.5-\gamma}\log n +16m^3\frac{n^{1-\gamma}\log (d)^2}{d^{k-j+1}}+2m \frac{n\log d}{d^{k-j+1}}+ 4m d^{j-1}\sqrt{mn}\log n + 16m^3\frac{n\log (d)^2}{d^{k-j+2}} \\
    \overset{(d)}{\leq}{}& -\frac{1}{2}d^{j}\sqrt{m n}\log n+4m d^{j-1}\sqrt{mn}\log n + \sqrt{m n}\log n \overset{(e)}{\leq} -\sqrt{m n}\log n,
\end{align*} 
where $(a)$ follows by the lower bound on $s_{j-1}, s_{j+1}$ and upper bound on $s_j$. In particular, as $\BFs \in \calD_{j-1}^{(4)}$, we get a lower bound on $s_{j-1}$, as $\tilde{W}_j(\BFs) \geq 0$, we get an upper bound on $s_j$, and as $\BFs \in \calC_{j+1}^{(4)}$, we get a lower bound on $s_{j+1}$. Next, $(b)$ follows by Lemma \ref{lemma: power_of_d_not_going_to_zero} and using the identity $(a+b+c)^2 \leq 3(a^2+b^2+c^2)$. Now, $(c)$ follows as $k-j \geq 1$ which implies that $k-j+1 \leq 2(k-j)$. In addition, we also use that $j \geq 1$ and $k \leq m$. Lastly, $(d)$ follows as there exists $n_4^{(2)} \in \bbZ_+$ such that for all $n \geq n^{(2)}_4$, we have
\begin{align*}
    \frac{1}{3}m^2\frac{n \log (d)^2}{d^{k-j+1}} &\overset{(d_1)}{\geq} \frac{1}{3}m^2\frac{n \log (d)^2}{d^{k-j+1}} \times \frac{1152m^4\log (d)^2}{d^{k-j+1}} =  384m^6 \frac{n\log (d)^4}{d^{2k-2j+2}} \\
    \frac{1}{3}m^2\frac{n \log (d)^2}{d^{k-j+1}} &\overset{(d_2)}{\geq} \frac{1}{6}m\frac{n^{1-\gamma} \log d}{d^{k-j+1-m}} \overset{(d_3)}{\geq} 3m \frac{n^{1-\gamma}\log d}{d^{k-j}} \overset{(d_4)}{\geq} 2m \frac{n^{1-\gamma}\log d}{d^{k-j}}+16m^3\frac{n^{1-\gamma}\log (d)^2}{d^{k-j+1}}\\
    \frac{1}{3}m^2\frac{n \log (d)^2}{d^{k-j+1}} &\geq 3m \frac{n \log d}{d^{k-j+1}} \overset{(d_4)}{\geq} 2m \frac{n \log d}{d^{k-j+1}} + 16m^3\frac{n\log (d)^2}{d^{k-j+2}}\\
    \frac{1}{4}d^{j}\sqrt{m n}\log n &\overset{(d_5)}{\geq} 24m^3 d^{j}  \log (n)^2 (2m n^\gamma\log d)  \overset{(d_6)}{\geq} 24m^3 d^{2j}\log (n)^2 \\
     \frac{1}{4}d^{j}\sqrt{m n}\log n &\overset{(d_7)}{\geq} 4\sqrt{m}m d^{j}n^{0.5-\gamma}\log n,
\end{align*}
where $(d_1)$ follows as $m^4 \log (d)^2 / d^{k-j+1} \leq m^4 \log (d)^2 / d^{2} \leq \log (n)^4 / (d\log d)^2 \rightarrow 0$ as $n \rightarrow \infty$. Next, $(d_2)$ follows as $2n^\gamma m \log d \geq d^{m}$, $(d_3)$ follows as $m \geq 2$, and $(d_4)$ follows as $m^2\log d/d \leq \log (n)^2/(d\log d) \rightarrow 0$. Now, $(d_5)$ follows as $\gamma < 0.5$ and $m \leq \log n$, and $(d_6)$ follows as $d^j \leq d^{m} = 2m n^\gamma\log d$. Lastly, $(d_7)$ follows as $\gamma > 0$ and $m \leq \log n$. Further, $(e)$ follows as there exists $n_4^{(3)} \in \bbZ_+$ such that for all $n \geq n_4^{(3)}$, we have
\begin{align*}
    \frac{1}{4}d^j \sqrt{m n}\log n &\geq \frac{1}{4}d \sqrt{m n}\log n \geq 2\sqrt{m n}\log n \\
    \frac{1}{4}d^j \sqrt{m n}\log n &\geq 4m d^{j-1}\sqrt{mn}\log n,
\end{align*}
by noting that $d = \Omega(\log(n)^3)$ while $m = o(\log n)$. Thus, for all $n \geq \max\left\{n_4^{(2)}, n_4^{(3)}\right\}$, we have $\Delta \tilde{W}_j(\BFs) \leq -\sqrt{m n}\log n$ when $\tilde{W}_j(\BFs) \geq 0$ and $\BFs \in \calC_{j+1}^{(4)} \cap \calD_{j-1}^{(4)}$. Now, by using Lemma \ref{lemma: iterative_ssc}, we obtain a high probability upper bound on $\bars_j$ as follows:
\begin{align*}
     &\P{\bars_{j} \leq  n - 2m \frac{n \log d}{d^{k-j+1}} - (4m-j) d^{j-1}\sqrt{m n}\log n - 16(k-j+1)m^2 \frac{n \log (d)^2}{d^{k-j+2}}-\sqrt{m n}\log n} \\
     ={}& \P{\tilde{W}_j(\bbars) \geq \sqrt{m n}\log n} \\
     \overset{(a)}{\leq}{}& \left(\frac{n}{n+\sqrt{m n}\log n}\right)^{(\sqrt{m n}\log n)/2} + \sqrt{n} \P{\bbars \notin \calC_{j+1}^{(4)} \cap \calD_{j-1}^{(4)}} \\
     \overset{(b)}{\leq}{}& \left(\frac{1}{n}\right)^{(m \log n)/4} + \sqrt{n}\left(\P{\bbars \notin  \calC_{j+1}^{(4)}} + \P{\bbars \notin \calD_{j-1}^{(4)}}\right) \\
     \overset{(c)}{\leq}{}& \left(\frac{1}{n}\right)^{(m \log n)/4} + \sqrt{n} \left(\frac{1}{n}\right)^{(m \log n)/4-4(k-0.25)m-(k-j)} + \sqrt{n}\left(\frac{1}{n}\right)^{(m \log n)/4-4(k-0.5)m} \\
     \overset{(d)}{\leq}{}&  \left(\frac{1}{n}\right)^{(m \log n)/4-4(k-0.25)m-(k-j+1)},
\end{align*}
where $(a)$ follows by Lemma \ref{lemma: iterative_ssc} and $(b)$ follows by Lemma \ref{lemma: simplified_ssc_identity}. Next, $(c)$ follows by upper bounding $\P{\bbars \notin  \calC_{j+1}^{(4)}}$ using induction hypothesis \eqref{eq: precise_inner_induction} and upper bounding $\P{\bbars \notin \calD_{j-1}^{(4)}}$ for all $n \geq n_2$ using Lemma \ref{lemma: remaining_weak_outer_induction}. Lastly, $(d)$ follows for all $n \geq n_4^{(4)}$ for some $n_4^{(4)} \in \bbZ_+$. By considering $n_4 \overset{\Delta}{=} n_4^{(1)} \geq \max_{k \in [4]}\left\{n^{(k)}_4, n_2, n_3\right\}$, the induction step is complete. To complete the proof, note that
\begin{align*}
  &\P{\bars_{j} \leq  n - 2m \frac{n \log d}{d^{k-j+1}} - 4m d^{j-1}\sqrt{m n}\log n - 16m^3 \frac{n \log (d)^2}{d^{k-j+2}}}  \\
   \overset{(*)}{\leq}{}& \P{\bars_{j} \leq  n - 2m \frac{n \log d}{d^{k-j+1}} - (4m-j) d^{j-1}\sqrt{m n}\log n - 16(k-j+1)m^2 \frac{n \log (d)^2}{d^{k-j+2}}-\sqrt{m n}\log n} \\
    \leq{}& \left(\frac{1}{n}\right)^{(m \log n)/4-4(k-0.25)m-(k-j+1)} \leq \left(\frac{1}{n}\right)^{(m \log n)/4-4km} \quad \forall j \in [k],
\end{align*}
where $(*)$ follows as $j \geq 1$ and $k \leq m$. This completes the proof of the lemma. \hfill $\square$
\endproof
\section{Proof for the Upper Bound} \label{sec: lemmas_for_upper_bound}
\subsection{Proof of Lemmas for Upper Bound}
We start by re-stating the expression of $B_i$ below that was defined in \eqref{eq: B_j_upper_bound} for convenience.
\begin{align*}
    B_i = 18m d^{i-1}\sqrt{m n}\log n + 48m^3 \frac{n \log (d)^2}{d^{m -i+2}}+ \frac{n^{1-\gamma}}{d^{m -i}}\mathbbm{1}\left\{m >1\right\} \quad \forall i \in [m].
\end{align*}
Now, we start by stating and proving a `master' lemma that will help us in proving Lemma \ref{lemma: weak_bound_s_mplus1}, Lemma \ref{lemma: strong_upper_bound_smplus1}, and Lemma \ref{lemma: base_case_upper_bound}.
\begin{lemma} \label{lemma: base_case_and_beyond_for_upper_bound} Let $B_{m + 2} \leq nb$ and $x \in \bbR_+$ and $\tilde{n}_0^{(1)} \in \bbZ_+$ be such that for all $n \geq \tilde{n}_0^{(1)}$, we have
\begin{align*}
    \P{\sum_{l=m +2}^b \bars_l \geq B_{m + 2}} \leq \left(\frac{1}{n}\right)^{\frac{m \log n}{x}}.\numberthis \label{eq: upper_bound_tail}
\end{align*}
Define 
\begin{align*}
    B'_{m} &= (1 + (b-1)\mathbbm{1}\left\{B_{m+2} \geq 2\right\})\left(17m d^{m-1}\sqrt{m n}\log n + 48m^3 \frac{n \log (d)^2}{d^{2}}+ n^{1-\gamma}\mathbbm{1}\left\{m >1\right\} \right).
\end{align*}
Then, there exists $\tilde{n}_0^{(2)} \in \bbZ_+$ such that for all $n \geq \tilde{n}_0^{(2)}$, we have
\begin{align*}
    \P{\bars_{m+1} \geq B'_{m}} &\leq \left(\frac{1}{n}\right)^{\frac{m \log n}{\max\{x, 5\}}-m -2} \numberthis \label{eq: master_lemma_smplus1} \\
    \P{\bars_{m} \geq n - 2m \frac{n\log d}{d}+ B'_{m}} &\leq \left(\frac{1}{n}\right)^{\frac{m \log n}{\max\{x, 5\}}-m-2}.\numberthis \label{eq: master_lemma_sm}
\end{align*}
\end{lemma}
\proof{Proof of Lemma \ref{lemma: base_case_and_beyond_for_upper_bound}}
Consider the following functions:
\begin{align*}
    L_{m +1}(\BFs) ={}& \sum_{l=m +1}^b s_l - \left(1+(b-1)\mathbbm{1}\left\{B_{m +2} \geq 2\right\}\right) \times \\
    &\left(8m d^{m-1}\sqrt{m n}\log n + 24m^3 \frac{n \log (d)^2}{d^{2}}+n^{1-\gamma}\mathbbm{1}\{m > 1\}\right) \\
    L_l(\BFs) ={}& n - 2m \frac{n\log d}{d^{m-l+1}} + 3l d^{l-1}\sqrt{m n}\log n + 7lm^2\frac{n\log (d)^2}{d^{m-l+2}} - s_{l} \quad \forall l \in [m].
\end{align*}
Now, we define the Lyapunov function $U_j(\BFs)$ in terms of $\{L_l(\BFs): l \in [m+1]\}$ as follows:
\begin{subequations}
\begin{align}
    U_j^{(1)}(\BFs) &= L_{m +1}(\BFs)-\sum_{l=j+1}^{m} L_l(\BFs) \\
    U_j^{(2)}(\BFs) &= L_{j}(\BFs)\\
    U_j(\BFs)&=\min\left\{U_j^{(1)}(\BFs), U_j^{(2)}(\BFs)\right\}.
\end{align}
\label{eq: u_lyapunov_functions}
\end{subequations}
We use induction on $j \in \left\{1, \hdots, m\right\}$ to show the following claim:
\begin{claim} There exists $\tilde{n}_c \geq \tilde{n}_0^{(1)}$ such that for all $n \geq \tilde{n}_c$, we have \label{claim: upper_bound}
\begin{align*}
    \P{U_1(\bbars) \geq \sqrt{m n}\log n} &\leq \left(\frac{1}{n}\right)^{\frac{m\log n}{\max\{x, 5\}}-m}.
\end{align*}
\end{claim}
We use the high probability upper bound $\sum_{l=m+2}^b s_l \leq B_{m+2}$ to prove the above claim. We defer the details of the proof to Appendix \ref{app: claim_ub} and continue with the proof of Lemma \ref{lemma: base_case_and_beyond_for_upper_bound}. Now, we analyze the drift of $U_0^{(1)}(\BFs)$ when $U_0^{(1)}(\BFs) \geq 0$ and $\BFs \in \tilde{\calC}^{(1)}_1$ where $\tilde{\calC}^{(1)}_1 = \left\{U_1(\BFs) \leq \sqrt{m n}\log n\right\}$. First, we obtain a useful lower bound on $s_1$ as follows:
\begin{align*}
    U_0^{(1)}(\BFs) \geq 0 &\Rightarrow L_{m +1}(\BFs)-\sum_{l=2}^{m}  L_l(\BFs)\geq L_1(\BFs) \\
    &\overset{(*)}{\Rightarrow} L_1(\BFs) \leq \sqrt{m n}\log n \\
    &\Rightarrow s_1 \geq n - 2m \frac{n\log d}{d^{m}} + 2\sqrt{m n}\log n + 7m^2\frac{n\log (d)^2}{d^{m+1}}, \numberthis \label{eq: lower_bound_s1}
\end{align*}
where $(*)$ follows as $\BFs \in \tilde{\calC}^{(1)}_1$. Now, the drift is given as follows:
\begin{align*}
    \Delta U_0^{(1)}(\BFs)& = -s_1 + \lambda\left(1 - \left(\frac{s_b}{n}\right)^{\lfloor d \rfloor}\right) \\
    &\overset{(a)}{\leq} -n + 2m \frac{n\log d}{d^{m}} - 2\sqrt{m n}\log n - 7m^2\frac{n\log (d)^2}{d^{m+1}} + n - n^{1-\gamma} \\
    &\overset{(b)}{=}- 2\sqrt{m n}\log n - 7m^2\frac{n\log (d)^2}{d^{m+1}} \\
    &\leq -\sqrt{m n}\log n.
\end{align*}
where $(a)$ follows by using the trivial bound $s_b \geq 0$ and lower bounding $s_1$ using \eqref{eq: lower_bound_s1}. Next, $(b)$ follows as $n^{1-\gamma} = 2m \frac{n\log d}{d^{m}}$. Thus, we have $\Delta U_0^{(1)}(\BFs) \leq -\sqrt{m n}\log n$ when $U_0^{(1)}(\BFs) \geq 0$ and $\BFs \in \tilde{\calC}^{(1)}_1$. Now, using Lemma \ref{lemma: iterative_ssc}, we have
\begin{align*}
    \P{U_0^{(1)}(\bbars) \geq \sqrt{m n}\log n} &\leq  \left(\frac{n}{n+\sqrt{m n}\log n}\right)^{(\sqrt{m n}\log n)/2} + \sqrt{n}\P{\bbars \notin \tilde{\calC}^{(1)}_1} \\
    &\overset{(a)}{\leq} \left(\frac{1}{n}\right)^{(m \log n)/4} + \sqrt{n}\P{\bbars \notin \tilde{\calC}^{(1)}_1} \\
    &\overset{(b)}{\leq} \left(\frac{1}{n}\right)^{(m \log n)/4} + \left(\frac{1}{n}\right)^{\frac{m \log n}{\max\{x,5\}}-m-0.5} \\
    &\overset{(c)}{\leq} \left(\frac{1}{n}\right)^{\frac{m \log n}{\max\{x,5\}}-m -1}, \numberthis \label{eq: upper_bound_on_U}
\end{align*}
where $(a)$ follows by Lemma \ref{lemma: simplified_ssc_identity} and $(b)$ follows by Claim \ref{claim: upper_bound} for all $n \geq \tilde{n}_c$. Lastly, $(c)$ follows for all $n \geq \tilde{n}_0^{(2)}$ for some $\tilde{n}_0^{(2)} \in \bbZ_+$. Now, we will translate the above probability bound to the one required for the lemma. Define $\tilde{\calD}_l^{(1)}$ as follows:
\begin{align*}
    \tilde{\calD}_l^{(1)} = \left\{s_l \geq n - 2m \frac{n \log d}{d^{m-l+1}} - 4m d^{l-1}\sqrt{m n}\log n -16m^3 \frac{n \log (d)^2}{d^{m-l+2}}\right\} \quad \forall l \in [m].
\end{align*}
We first prove the high probability upper bound on $\bars_{m + 1}$ below.
\begin{align*}
\P{\bars_{m+1} \geq B'_{m}} &\overset{(a)}{=} \P{\bbars \in \left\{s_{m+1} \geq B'_{m}\right\} \cap \bigcap_{l=1}^{m} \tilde{\calD}_l^{(1)}} + \P{\bbars \in \left\{s_{m+1} \geq B'_{m}\right\} \cap \bigcup_{l=1}^{m} \tilde{\calD}_l^{(1),c}} \\
&\overset{(b)}{\leq} \P{\bbars \in \{s_{m+1} \geq B'_{m}\} \cap \bigcap_{l=1}^{m}\tilde{\calD}_l^{(1)}} + \sum_{l=1}^{m} \P{\bbars \notin \tilde{\calD}_l^{(1)}} \\
&\overset{(c)}{\leq} \P{\bbars \in \{s_{m+1} \geq B'_{m}\} \cap \bigcap_{l=1}^{m} \tilde{\calD}_l^{(1)}} + \left(\frac{1}{n}\right)^{(m\log n)/5-1} \\
&\overset{(d)}{\leq} \P{U_0^{(1)}(\bbars) \geq \sqrt{m n}\log n}+  \left(\frac{1}{n}\right)^{(m\log n)/5-1} \\
&\overset{(e)}{\leq} \left(\frac{1}{n}\right)^{\frac{m\log n}{\max\{x,5\}}-m -1} + \left(\frac{1}{n}\right)^{(m\log n)/5-1} \\
&\leq \left(\frac{1}{n}\right)^{\frac{m\log n}{\max\{x,5\}}-m -2},
\end{align*}
where $(a)$ follows by the law of total probability, and $(b)$ follows by the union bound. Next, $(c)$ follows by Theorem \ref{theo: lower_bound}. Now, $(d)$ follows by noting the following:
\begin{align*}
    \left\{s_{m +1} \geq B'_{m}\right\} \cap \bigcap_{l=1}^{m}\left\{\BFs \in \tilde{\calD}_l^{(1)}\right\} &\overset{(*)}{\subseteq} \left\{L_{m+1}(\BFs) \geq 9m d^{m-1}\sqrt{m n}\log n + 24m^3 \frac{n \log (d)^2}{d^{2}}\right\} \cap \bigcap_{l=1}^{m}\left\{\BFs \in \tilde{\calD}_l^{(1)}\right\} \\
    &\overset{(**)}{\subseteq} \left\{U_0^{(1)}(\BFs) \geq \sqrt{m n}\log n\right\}. \numberthis \label{eq: upper_bound_smplus1}
\end{align*}
where $(*)$ follows as $\mathbbm{1}\{m > 1\}n^{1-\gamma} \geq 0$ and $(**)$ holds for all $n \geq \tilde{n}_0^{(3)}$ for some $\tilde{n}_0^{(3)} \in \bbZ_+$. In particular, note that $U_0^{(1)}(\BFs) = L_{m+1}(\BFs)-\sum_{l=1}^{m} L_l(\BFs)$ and we can upper bound $\sum_{l=1}^{m} L_l(\BFs)$ as follows:
\begin{align*}
    \BFs \in \bigcap_{l=1}^{m} \tilde{\calD}_l^{(1)} \Rightarrow \sum_{l=1}^{m} L_l(\BFs) &\leq 7m \sqrt{m n}\log n \sum_{l=1}^{m} d^{l-1} + 23m^3 n \log (d)^2 \sum_{l=1}^{m}\frac{1}{d^{m - l + 2}} \\
    &\leq 8m d^{m-1}\sqrt{m n} \log n+24m^3 \frac{n \log (d)^2}{d^2}.
\end{align*}
Lastly, $(e)$ follows by \eqref{eq: upper_bound_on_U}. This completes the proof of \eqref{eq: master_lemma_smplus1}. Now, we will prove \eqref{eq: master_lemma_sm}. Similar to the upper bound on $\bars_{m+1}$, we get
\begin{align*}
    &\P{\bars_{m} \geq n - 2m \frac{n\log d}{d} + B'_{m}} \\
    \leq{}& \P{\bbars \in \left\{s_{m} \geq n - 2m \frac{n\log d}{d} + B'_{m}\right\} \cap \bigcap_{l=1}^{m - 1}\left\{\BFs \in \tilde{\calD}_l^{(1)}\right\}} + \left(\frac{1}{n}\right)^{(m\log n)/5-1} \\
    \overset{(*)}{\leq}{}& \P{U_0^{(1)}(\bbars) \geq \sqrt{m n}\log n}+ \left(\frac{1}{n}\right)^{(m\log n)/5-1} \\
    \leq{}&  \left(\frac{1}{n}\right)^{\frac{m\log n}{\max\{x, 5\}}-m - 1}+ \left(\frac{1}{n}\right)^{(m\log n)/5-1} \\
    \leq{}& \left(\frac{1}{n}\right)^{\frac{m\log n}{\max\{x,5\}}-m -2}.
\end{align*}
where $(*)$ follows by noting that $U_0^{(1)}(\BFs)=L_{m+1}(\BFs)-\sum_{l=1}^{m} L_l(\BFs)$ and bounding the terms $\left\{L_l(\BFs) : l \in [m+1]\backslash \{m\}\right\}$. In particular, note that 
\begin{equation*}
    L_{m +1}(\BFs) \geq - (1+(b-1)\mathbbm{1}\left\{B_{m+2}\geq 2\right\})\left(8m d^{m-1}\sqrt{m n}\log n + 24m^3 \frac{n \log (d)^2}{d^{2}}+n^{1-\gamma}\mathbbm{1}\{m > 1\}\right)
\end{equation*}
Also, as $\BFs \in \bigcap_{l=1}^{m-1} \tilde{\calD}_l^{(1)}$, there exists $\tilde{n}_0^{(5)} \in \bbZ_+$ such that for all $n \geq \tilde{n}_0^{(5)}$, we have
\begin{align*}
    \sum_{l=1}^{m-1} L_l(\BFs) \leq{}& 7m \sqrt{m n}\log n \sum_{l=1}^{m-1} d^{l-1} + 23m^3 n \log (d)^2 \sum_{l=1}^{m-1}\frac{1}{d^{m - l + 2}} \\
    \leq{}& 8 m d^{m-2}\sqrt{m n}\log n  + 24m^3  \frac{n \log (d)^2}{d^{3}} \\
    \leq{}& m d^{m-1}\sqrt{m n}\log n  + m^3  \frac{n \log (d)^2}{d^{2}},
\end{align*}
where the last inequality follows for $n \geq \tilde{n}_0^{(6)}$ for some $\tilde{n}_0^{(6)} > 0$ since $d \rightarrow \infty$ as $n \rightarrow \infty$. Combining the bounds, we have 
\begin{align*}
    &L_{m+1}(\BFs)-\sum_{l=1}^{m} L_l(\BFs) \\
    \geq{}& - (1+(b-1)\mathbbm{1}\left\{B_{m+2}\geq 2\right\})\left(9m d^{m-1}\sqrt{m n}\log n + 25m^3 \frac{n \log (d)^2}{d^{2}}+n^{1-\gamma}\mathbbm{1}\{m > 1\}\right) - L_m(\BFs) \\
    \geq{}& - (1+(b-1)\mathbbm{1}\left\{B_{m+2}\geq 2\right\})\left(12m d^{m-1}\sqrt{m n}\log n + 34m^3 \frac{n \log (d)^2}{d^{2}}+n^{1-\gamma}\mathbbm{1}\{m > 1\}\right) + B_m^\prime \\
    \geq{}& \sqrt{m n} \log n,
\end{align*}
where the second last inequality follows by using the lower bound $s_m \geq n - 2m n \log d/d + B_m^\prime$. Thus, by defining $\tilde{n}_0^{(2)}=\max_{k \in \{1, \hdots, 6\}}\{\tilde{n}_0^{(k)}, \tilde{n}_c\}$, the proof is complete. \hfill $\square$
\endproof
\proof{Proof of Lemma \ref{lemma: weak_upper_bound_smplus1}}
Note that $\sum_{l=m+2}^b \bars_l \leq nb$ with probability 1, so, by applying Lemma \ref{lemma: base_case_and_beyond_for_upper_bound} with $B_{m+2}=nb$, $x=1$, $\tilde{n}_0^{(1)}=1$ and noting that $B'_{m} \leq bB_{m}$, for all $n \geq \tilde{n}_0^{(2)}$, we have
\begin{align*}
    \P{\bars_{m +1} \geq b B_{m}} \leq \left(\frac{1}{n}\right)^{(m \log n)/5-m -2} \leq \left(\frac{1}{n}\right)^{(m \log n)/6},
\end{align*}
where the last inequality follows for all $n \geq \tilde{n}_1^{(1)}$ for some $\tilde{n}_1^{(1)} \in \bbZ_+$ and $B_{m}$ is defined in \eqref{eq: B_j_upper_bound}. Thus, by defining $\tilde{n}_1\overset{\Delta}{=}\max\{\tilde{n}_0^{(2)}, \tilde{n}_1^{(1)}\}$, the proof is complete. \hfill $\square$
\endproof
\proof{Proof of Lemma \ref{lemma: very_light_tail}}
Consider the Lyapunov function:
\begin{align*}
    U_{m+2}(\BFs) = \sum_{l=m+2}^b s_l.
\end{align*}
In steady-state, we have
\begin{align*}
    0=\E{\Delta U_{m+2}(\bbars)} = -\E{\bars_{m+2}} + \E{\lambda\left(\left(\frac{\bars_{m +1}}{n}\right)^{\lfloor d \rfloor} - \left(\frac{\bars_b}{n}\right)^{\lfloor d \rfloor}\right)}.
\end{align*}
Thus, we have
\begin{align*}
    &\E{\bars_{m+2}} \\
    ={}& \E{\lambda\left(\left(\frac{\bars_{m +1}}{n}\right)^{\lfloor d \rfloor} - \left(\frac{\bars_b}{n}\right)^{\lfloor d \rfloor}\right)} \\
    \overset{(a)}{\leq}{}& n\E{\left(\frac{\bars_{m +1}}{n}\right)^{\lfloor d \rfloor} } \\
    \overset{(b)}{=}{}& n\E{\left(\frac{\bars_{m +1}}{n}\right)^{\lfloor d \rfloor} \bigg|\ \bars_{m +1} \geq b B_{m}}\P{\bars_{m +1} \geq b B_{m}} + n\E{\left(\frac{\bars_{m +1}}{n}\right)^{\lfloor d \rfloor} \bigg|\ \bars_{m +1} \leq b B_{m}}\P{\bars_{m +1} \leq b B_{m}} \\
    \overset{(c)}{\leq} {}& n\left(\frac{b B_{m}}{n}\right)^{\lfloor d \rfloor} + n\P{\bars_{m +1} \geq b B_{m}} \\
    \overset{(d)}{\leq}{}& n\left(\frac{b B_{m}}{n}\right)^{d-1} + \left(\frac{1}{n}\right)^{(m\log n)/6-1} \\
    \overset{(e)}{\leq}{}& \left(\frac{1}{n}\right)^{m\log n-1}+ \left(\frac{1}{n}\right)^{(m\log n)/6-1} \overset{(f)}{\leq} \left(\frac{1}{n}\right)^{(m\log n)/7},
\end{align*}
where $(a)$ follows as $\lambda \leq n$ and $s_b \geq 0$. Next, $(b)$ follows by the law of total expectation. Further, $(c)$ follows as $\P{\bars_{m +1} \geq b B_{m}} \leq 1$ and $s_{m+1} \leq n$. Now, $(d)$ follows for all $n \geq \tilde{n}_1$ by Lemma \ref{lemma: weak_bound_s_mplus1}. Lastly, $(e)$ follows for all $n \geq \tilde{n}_2^{(1)}$ for some $\tilde{n}_2^{(1)} \in \bbZ_+$ as $b B_{m}=o(n)$ and so $b B_m/n \leq 1/e$ implying $(b B_m/n)^{d-1} \leq (1/n)^{\Omega(\log(n)^2)} \leq (1/n)^{m \log n}$ as $d \geq \Omega(\log (n)^3)$ and $m = o(\log n)$. Lastly, $(f)$ follows for all $n \geq \tilde{n}_2^{(2)}$ for some $\tilde{n}_2^{(2)}$ as $m = o(\log n)$. Now, by Markov's inequality, we have
\begin{align*}
    \P{\bars_{m+2} \geq 1} \leq \E{\bars_{m+2}} \leq\left(\frac{1}{n}\right)^{(m\log n)/7}.
\end{align*}
Lastly, note that as $s_{k_1} \geq s_{k_2}$ for all $k_1 \leq k_2$, we have
\begin{align*}
    \left\{\bars_{m+2} \geq 1\right\} = \left\{\sum_{l=m+2}^b \bars_{l} \geq 1\right\}.
\end{align*}
Thus, by defining $\tilde{n}_2\overset{\Delta}{=}\max\left\{\tilde{n}_1, \tilde{n}_2^{(1)},\tilde{n}_2^{(2)}\right\}$, the proof is complete. \hfill $\square$
\endproof
\proof{Proof of Lemma \ref{lemma: strong_upper_bound_smplus1} and Lemma \ref{lemma: base_case_upper_bound}}
By using Lemma \ref{lemma: very_light_tail}, for all $n \geq \tilde{n}_2$, we have
\begin{align*}
    \P{\sum_{l=m+2}^b \bars_l \geq 1} \leq \left(\frac{1}{n}\right)^{(m\log n)/7}.
\end{align*}
Now, by using Lemma \ref{lemma: base_case_and_beyond_for_upper_bound} with $B_{m + 2}=1$, $x=7$, $\tilde{n}_0^{(1)}=\tilde{n}_2$ and noting that $B'_{m}  \leq B_{m}$ as $B_{m + 2} = 1$, for all $n \geq \tilde{n}_0^{(2)}$, we have
\begin{align*}
    \P{\bars_{m+1} \geq B_{m}} \leq \left(\frac{1}{n}\right)^{(m\log n)/7-m-2} \overset{(*)}{\leq} \left(\frac{1}{n}\right)^{(m\log n)/8} \\
    \P{\bars_{m} \geq n-2m \frac{n\log d}{d}+B_{m}} \leq \left(\frac{1}{n}\right)^{(m\log n)/7-m-2} \overset{(*)}{\leq} \left(\frac{1}{n}\right)^{(m\log n)/8},
\end{align*}
where $(*)$ follows for all $n \geq \tilde{n}_3^{(1)}$ for some $\tilde{n}_3^{(1)} \in \bbZ_+$. Thus, by defining $\tilde{n}_3 = \max\left\{\tilde{n}_0^{(2)},\tilde{n}_3^{(1)}\right\}$, the proof is complete. \hfill $\square$
\endproof
\proof{Proof of Lemma \ref{lemma: induction_step_upper_bound}}
Define the Lyapunov function:
\begin{align*}
    \tilde{W}_j(\BFs) = s_j - n + 2m \frac{n\log d}{d^{m - j+1}}-B_j - 2(m -j)m \frac{n \log d}{d^{m -j +2}},
\end{align*}
where $B_j=o(mn\log d/d^{m-j+1})$ as defined in \eqref{eq: B_j_upper_bound}. For any $j \in [m-1]$, we analyze the drift of $ \tilde{W}_j(\BFs)$ when $\tilde{W}_j(\BFs) \geq 0$ and $\BFs \in \tilde{\calC}^{(2)}_{j+1}$ where 
\begin{align*}
    \tilde{\calC}^{(2)}_{j+1} = \left\{\tilde{W}_{j+1}(\BFs) \leq \sqrt{m n}\log n\right\}.
\end{align*}
Thus, we have
\begin{subequations} \label{eq: bounds_induction_step_upper_bound}
\begin{align}
    s_j &\geq n - 2m \frac{n\log d}{d^{m - j+1}}+B_j + 2(m -j)m \frac{n \log d}{d^{m -j +2}} \\
    s_{j+1} &\leq n - 2m \frac{n\log d}{d^{m - j}}+B_{j+1} + 2(m -j-1)m \frac{n \log d}{d^{m -j +1}} + \sqrt{m n}\log n.
\end{align}
\end{subequations}
Now, the drift is given as follows:
\begin{align*}
    \Delta \tilde{W}_j(\BFs) ={}& -s_j + s_{j+1} + \lambda\left(\left(\frac{s_{j-1}}{n}\right)^{\lfloor d \rfloor} - \left(\frac{s_j}{n}\right)^{\lfloor d \rfloor}\right) \\
    \overset{(a)}{\leq}{}&  2m \frac{n\log d}{d^{m - j+1}}-B_j - 2(m -j)m \frac{n \log d}{d^{m -j +2}}\\
    &- 2m \frac{n\log d}{d^{m - j}}+B_{j+1} + 2(m -j-1)m \frac{n \log d}{d^{m -j +1}}+\sqrt{m n}\log n \\
    &+n\left(1-\left(1- 2m \frac{\log d}{d^{m - j+1}}+\frac{B_j}{n} + 2(m -j)m \frac{ \log d}{d^{m -j +2}}\right)^{\lfloor d \rfloor}\right) \\
    \overset{(b)}{\leq}{}&  2m \frac{n\log d}{d^{m - j+1}}-B_j - 2(m -j)m \frac{n \log d}{d^{m -j +2}}\\
    &- 2m \frac{n\log d}{d^{m - j}}+B_{j+1} + 2(m -j-1)m \frac{n \log d}{d^{m -j +1}}+\sqrt{m n}\log n \\
    &+n\left(2m \frac{\log d}{d^{m - j}}-\frac{B_{j+1}}{n}-2(m -j)m \frac{ \log d}{d^{m -j +1}}\right) \\
    \leq{}& \sqrt{m n}\log n-B_j \\
    \leq{}& \sqrt{m n}\log n-18m d^{j-1}\sqrt{m n}\log n \\
    \leq{}& -\sqrt{m n}\log n,
\end{align*}
where $(a)$ follows as $s_{j-1} \leq n$, and $s_{j-1}$ and $s_j$ is bounded as in \eqref{eq: bounds_induction_step_upper_bound} and $(b)$ follows for all $n \geq \tilde{n}_5^{(1)}$ for some $\tilde{n}_5^{(1)} \in \bbZ_+$ by Lemma \ref{lemma: power_of_d_not_going_to_zero} and noting that $dB_j = B_{j+1}$. In particular, for $n$ large enough independent of $j \in [m-1]$, we have
\begin{align*}
    2m \frac{\log d}{d^{m - j+1}}-\frac{B_j}{n} - 2(m -j)m \frac{ \log d}{d^{m -j +2}} \geq 0.
\end{align*}
Thus, we have $\Delta \tilde{W}_j(\BFs) \leq -\sqrt{m n}\log n$ when $\tilde{W}_j(\BFs) \geq 0$ and $\BFs \in \tilde{\calC}^{(2)}_{j+1}$. Thus, by Lemma \ref{lemma: iterative_ssc}, we have
\begin{align*}
    \P{\tilde{W}_j(\bbars) \geq \sqrt{m n}\log n} &\leq \left(\frac{n}{n+\sqrt{m n}\log n}\right)^{(\sqrt{m n}\log n)/2} + \sqrt{n}\P{\bbars \notin \tilde{\calC}^{(2)}_{j+1}} \\
    &\overset{(a)}{\leq} \left(\frac{1}{n}\right)^{(m \log n)/4} + \sqrt{n}\P{\bbars \notin \tilde{\calC}^{(2)}_{j+1}} \\
    &\overset{(b)}{\leq} \left(\frac{1}{n}\right)^{(m \log n)/4} + \sqrt{n}\left(\frac{1}{n}\right)^{(m \log n)/8-(m-j-1)} \\
    &\overset{(c)}{\leq} \left(\frac{1}{n}\right)^{(m \log n)/8-(m-j)},
\end{align*}
where $(a)$ follows by Lemma \ref{lemma: simplified_ssc_identity} and $(b)$ follows by the induction hypothesis \eqref{eq: induction_upper_bound}. Lastly, $(c)$ follows for all $n \geq \tilde{n}_5^{(2)}$ for some $\tilde{n}_5 \in \bbZ_+$. Now, by setting $\tilde{n}_{IH} \geq \max_{k \in [2]}\{\tilde{n}_5^{(k)}\}$, the induction step is complete. This completes the proof of the lemma. \hfill $\square$
\endproof
\section{Proof of Claims for Lower Bound} \label{app: claim_lb}
\proof{Proof of Claim \ref{claim: induction_for_weak_induction_step}}
For $l \in \{1, \hdots, k-1\}$, we consider the following induction hypothesis: There exists $n_{c1}  \in \bbZ_+$ such that for all $n \geq n_{c1}$, we have
\begin{align*}
    \P{L_{lk}(\bbars) \geq \sqrt{m n}\log n} \leq \left(\frac{1}{n}\right)^{(m \log n)/4-4(k-1)m-(k-l)}. \numberthis \label{eq: induction_hypothesis_claim1}
\end{align*}
\textbf{Base Case:} We analyze the drift of $L_{k-1,k}(\BFs)$ when $L_{k-1,k}(\BFs) \geq 0$. Thus, we have 
\begin{subequations} \label{eq: bounds_base_case_claim}
\begin{align}
    s_{k-1} &\geq n - 3(k-1)m \frac{n \log d}{d^2}  \\
    s_k &\leq n - \frac{3km n \log d}{d}.
\end{align}
\end{subequations}
First, consider the case when $L_{k-1,k}^{(1)}(\BFs) \geq L_{k-1,k}^{(2)}(\BFs)$. The drift is given as follows:
\begin{align*}
    \Delta L_{k-1,k}(\BFs) &\leq \lambda\left(\left(\frac{s_{k-2}}{n}\right)^{\lfloor d \rfloor} - \left(\frac{s_{k-1}}{n}\right)^{\lfloor d \rfloor}\right) - s_{k-1} + s_{k} \\
    &\overset{(a)}{\leq} \lambda\left(1 - \left(1 - 3(k-1)m \frac{\log d}{d^2} \right)^{\lfloor d \rfloor}\right) + 3(k-1)m \frac{n \log d}{d^2}  - \frac{3km n \log d}{d} \\
    &\overset{(b)}{\leq} \lambda\left(3(k-1)m \frac{\log d}{d}\right) + 3(k-1)m \frac{n \log d}{d^2} - \frac{3km n \log d}{d} \\
   &\leq -\frac{3m n \log d}{d} + 3(k-1)m \frac{n \log d}{d^2}  \\
   &\overset{(c)}{\leq} -\frac{2m n \log d}{d} \overset{(d)}{\leq} -\sqrt{m n}\log n,
\end{align*}
where $(a)$ follows as $s_{k-2} \leq n$, and we use the bounds on $s_{k-1}$ and $s_k$ given by \eqref{eq: bounds_base_case_claim}. Next, $(b)$ follows by Lemma \ref{lemma: power_of_d_not_going_to_zero}. Now, $(c)$ follows as there exists $n_{c1}^{(1)}$ such that for all $n \geq n_{c1}^{(1)}$, we have
\begin{align*}
    \frac{m n \log d}{d} \overset{(c_1)}{\geq} 3m^2 \frac{n \log d}{d^2}  \geq 3(k-1)m \frac{n \log d}{d^2},
\end{align*}
where $(c_1)$ follows as $m / d \rightarrow 0$ as $n \rightarrow \infty$. Lastly, $(d)$ follows as 
\begin{align*}
    \frac{2m n \log d}{d} \geq \frac{2m n \log d}{d^{m}} = n^{1-\gamma} \geq \sqrt{m n}\log n. \label{eq: some_eqn_in_claim1}
\end{align*}
where the last inequality follows as $\gamma \in (0, 0.5)$. Now, consider the case when $L_{k-1,k}^{(1)}(\BFs) \leq L_{k-1,k}^{(2)}(\BFs)$. The drift is given as follows:
\begin{align*}
    \Delta L_{k-1,k}(\BFs) &\leq s_k - s_{k+1} - \lambda\left(\left(\frac{s_{k-1}}{n}\right)^{\lfloor d \rfloor} - \left(\frac{s_{k}}{n}\right)^{\lfloor d \rfloor}\right) \\
    &\overset{(a)}{\leq} n - \frac{3km n \log d}{d} - \lambda\left(\left(1 - 3(k-1)m \frac{ \log d}{d^2}\right)^{\lfloor d \rfloor} - \left(1 - \frac{3km \log d}{d}\right)^{\lfloor d \rfloor}\right) \\
    &\overset{(b)}{\leq} n - \frac{3km n \log d}{d} - \lambda\left(1 - 3(k-1)m \frac{ \log d}{d} - \frac{2}{d^{3km}}\right) \\
    &\leq -\frac{3m n\log d}{d} + \frac{2n}{d^{3km}} + n^{1-\gamma} \\
    &\overset{(c)}{\leq} -\frac{m n\log d}{d}+ \frac{2n}{d^{3km}} \overset{(d)}{\leq} -\sqrt{m n}\log n,
\end{align*}
where $(a)$ follows as $s_{k-2} \leq n$, and we use the bounds on $s_{k-1}$ and $s_k$ given by \eqref{eq: bounds_base_case_claim}. Next, $(b)$ follows for all $n \geq n_{c1}^{(2)}$ for some $n_{c1}^{(2)} \in \bbZ_+$ by Lemma \ref{lemma: power_of_d_not_going_to_zero} and Lemma \ref{lemma: power_of_d_going_to_zero}. Now, $(c)$ follows by \eqref{eq: some_eqn_in_claim1}. Lastly, $(d)$ follows as there exists $n_{c1}^{(3)}$ such that for all $n \geq n_{c1}^{(3)}$ we have
\begin{align*}
    \frac{m n\log d}{2d} &\geq \frac{2n}{d^{3k}} \geq \frac{2n}{d^{3km}} \\
    \frac{m n\log d}{2d} &\geq  \frac{m n\log d}{2d^{m}} \geq \frac{1}{4}n^{1-\gamma} \geq \sqrt{m n}\log n,
\end{align*}
where the last inequality follows as $\gamma < 0.5$. By combining the two cases, we get $\Delta L_{k-1,k}(\BFs) \leq -\sqrt{m n}\log n$ when $L_{k-1,k}(\BFs) \geq 0$. Thus, by Lemma \ref{lemma: iterative_ssc}, we have
\begin{align*}
    \P{L_{k-1, k}(\bbars) \geq \sqrt{m n}\log n} \leq \left(\frac{n}{n + \sqrt{m n}\log n}\right)^{(\sqrt{m n} \log n)/2} \leq \left(\frac{1}{n}\right)^{m\log n /4} \leq \left(\frac{1}{n}\right)^{m\log n /4 - 4(k-1)m-1},
\end{align*}
where the last inequality follows by Lemma \ref{lemma: simplified_ssc_identity}. By considering $n_{c1} \geq \max_{k \in [3]}\left\{n_{c1}^{(k)}\right\}$, the base case is complete.

\textbf{Induction Step:} We analyze the drift of $L_{i-1,k}(\BFs)$ when $L_{i-1,k}(\BFs) \geq 0$ and $\BFs \in \calC_{i, k}^{(1)} \cap \bigcap_{l=i}^{k-1} \calD_l^{(1)}$ defined as in \eqref{eq: sets_weak_induction_base_case}. Similar to the proof of Lemma \ref{lemma: outer_induction_weak_induction_step} (Eq. \eqref{eq: upper_bound_s1} and \eqref{eq: upper_bound_sk}), there exists $n_{c1}^{(4)} \in \bbZ_+$, such that for all $n \geq n_{c1}^{(4)}$, we have
\begin{subequations} \label{eq: induction_step_upper_bound_claim}
\begin{align}
    s_{i} &\leq n - \frac{3i m n\log d}{d^{k-i+1}} + \sqrt{m n}\log n  \\
    s_k &\leq n - 3m \frac{n \log d}{d}.
\end{align}
\end{subequations}

We proceed by analyzing the drift for the case when $L_{i-1,k}^{(1)}(\BFs) \geq L_{i-1,k}^{(2)}(\BFs)$.
\begin{align*}
    \Delta L_{i-1,k}(\BFs) &\leq \lambda\left(\left(\frac{s_{i-2}}{n}\right)^{\lfloor d \rfloor} - \left(\frac{s_{i-1}}{n}\right)^{\lfloor d \rfloor}\right) - s_{i-1} + s_i \\
    &\overset{(a)}{\leq} \lambda\left(1 - \left(1 - 3(i-1)m \frac{\log d}{d^{k-i+2}}\right)^{\lfloor d \rfloor}\right) + 3(i-1)m \frac{n\log d}{d^{k-i+2}} - \frac{3i m n\log d}{d^{k-i+1}} + \sqrt{m n}\log n   \\
    &\overset{(b)}{\leq} \lambda\left(3(i-1)m \frac{\log d}{d^{k-i+1}} \right) + 3(i-1)m \frac{n\log d}{d^{k-i+2}} - \frac{3i m n\log d}{d^{k-i+1}} + \sqrt{m n}\log n   \\
    &\leq - 3m \frac{n\log d}{d^{k-i+1}} + 3(i-1)m \frac{n\log d}{d^{k-i+2}}+ \sqrt{m n}\log n \\
    &\overset{(c)}{\leq} - m \frac{n\log d}{d^{k-i+1}} \overset{(c)}{\leq} -\sqrt{m n}\log n,
\end{align*}
where $(a)$ follows by lower bounding $s_{i-1}$ using $L_{i-1,k}(\BFs) \geq 0$, upper bounding $s_i$ using \eqref{eq: induction_step_upper_bound_claim}, and trivially upper bounding $s_{i-2}$ by $n$. Next, $(b)$ follows by Lemma \ref{lemma: power_of_d_not_going_to_zero}. Lastly, $(c)$ follows as there exists $n_{c1}^{(5)} \in \bbZ_+$ such that for all $n \geq n_{c1}^{(5)}$, we have
\begin{align*}
    m \frac{n\log d}{d^{k-i+1}} &\overset{(c_1)}{\geq} 3m^2 \frac{n\log d}{d^{k-i+2}} \geq 3(i-1)m \frac{n\log d}{d^{k-i+2}} \\
    m \frac{n\log d}{d^{k-i+1}} &\geq   m \frac{n\log d}{d^{m}} \geq \frac{1}{2}n^{1-\gamma} \overset{(c_2)}{\geq} \sqrt{m n}\log n,
\end{align*}
where $(c_1)$ follows as $m / d \leq \log n/d \rightarrow 0$ as $n \rightarrow \infty$ and $(c_2)$ follows as $\gamma < 0.5$. Now, consider the case when $L_{i-1,k}^{(1)}(\BFs) \leq L_{i-1,k}^{(2)}(\BFs)$. The drift is given as follows:
\begin{align*}
    \Delta L_{i-1,k}(\BFs) 
    \overset{(a)}{\leq}{}& s_i - s_{k+1} - \lambda\left(\left(\frac{s_{i-1}}{n}\right)^{\lfloor d \rfloor} - \left(\frac{s_k}{n}\right)^{\lfloor d \rfloor}\right) \\
    \overset{(b)}{\leq}{}& n - \frac{3i m n\log d}{d^{k-i+1}} + \sqrt{m n}\log n -\lambda\left(\left(1 - 3(i-1)m \frac{\log d}{d^{k-i+2}}\right)^{\lfloor d \rfloor} - \left(1 - 3m\frac{ \log d}{d}\right)^{\lfloor d \rfloor}\right) \\
    \leq{}& n - \frac{3i m n\log d}{d^{k-i+1}} + \sqrt{m n}\log n -\lambda\left(1 - 3(i-1)m \frac{\log d}{d^{k-i+1}} - \frac{2}{d^{3m}}\right) \\
    \leq{}& - 3m \frac{n\log d}{d^{k-i+1}}+ \sqrt{m n}\log n  + \frac{2n}{d^{3m}} + n^{1-\gamma} \\
    \overset{(c)}{\leq}{}& - m \frac{n\log d}{d^{k-i+1}}+ \sqrt{m n}\log n  + \frac{2n}{d^{3m}} \overset{(d)}{\leq} -\sqrt{m n}\log n,
\end{align*}
where $(a)$ follows by lower bounding $s_{i-1}$ using $L_{i-1,k}(\BFs) \geq 0$, upper bounding $s_i$ and $s_k$ using \eqref{eq: induction_step_upper_bound_claim}, and trivially lower bounding $s_{k+1}$ by $0$. Next, $(b)$ follows for all $n \geq n_{c1}^{(6)}$ for some $n_{c1}^{(6)} \in \bbZ_+$ by Lemma \ref{lemma: power_of_d_going_to_zero} and Lemma \ref{lemma: power_of_d_not_going_to_zero}. Further, $(c)$ follows as
\begin{align*}
    2m \frac{n\log d}{d^{k-i+1}} \geq 2m \frac{n\log d}{d^{m}} = n^{1-\gamma}.
\end{align*}
Lastly, $(d)$ follows as there exists $n^{(7)}_{c1}$ such that for all $n \geq n^{(7)}_{c1}$, we have
\begin{align*}
    m \frac{n\log d}{2d^{k-i+1}} &\geq m \frac{n\log d}{2d^{m}} \geq \frac{2n}{d^{3m}} \\
     m \frac{n\log d}{4d^{k-i+1}} &\geq m \frac{n\log d}{4d^{m}} \geq \frac{1}{8}n^{1-\gamma} \geq \sqrt{m n}\log n.
\end{align*}
By combining the two cases, we get $\Delta L_{i-1,k}(\BFs) \leq -\sqrt{m n}\log n$ when $L_{i-1,k}(\BFs) \geq 0$ and $\BFs \in \calC_{i, k}^{(1)} \cap \bigcap_{l=i}^{k-1} \calD_l^{(1)}$. Thus, by Lemma \ref{lemma: iterative_ssc}, we have
\begin{align*}
    &\P{L_{i-1,k}(\bbars)\geq \sqrt{m n}\log n}\\
    \leq{}& \left(\frac{n}{n+\sqrt{m n}\log n}\right)^{(\sqrt{m n}\log n)/2} + \sqrt{n}\P{\bbars \notin \calC_{i, k}^{(1)} \cap \bigcap_{l=i}^{k-1} \calD_l^{(1)}} \\
    \overset{(a)}{\leq}{}& \left(\frac{1}{n}\right)^{(m \log n)/4} + \sqrt{n}\left(\P{\bbars \notin \calC_{i, k}^{(1)}} +  \sum_{l=i}^{k-1} \P{\bbars \notin\calD_l^{(1)}}\right) \\
    \overset{(b)}{\leq}{}& \left(\frac{1}{n}\right)^{(m \log n)/4} + \sqrt{n}\left(\frac{1}{n}\right)^{(m \log n)/4-4(k-1)m-(k-i)} + \sqrt{n}m \left(\frac{1}{n}\right)^{(m \log n)/4-4(k-1)m} \\
    \overset{(c)}{\leq}{}& \left(\frac{1}{n}\right)^{m (\log n)/4-4(k-1)m-(k-i+1)},
\end{align*}
where $(a)$ follows by Lemma \ref{lemma: simplified_ssc_identity}. Next, $(b)$ follows by upper bounding $\P{\bbars \notin \calC_{i, k}^{(1)}}$ using the induction hypothesis \eqref{eq: induction_hypothesis_claim1}. Also, similar to \eqref{eq: converting_to_leading_term}, $\P{\bbars \notin\calD_l^{(1)}}$ is upper bounded for all $n \geq n_{c1}^{(8)}$ for some $n_{c1}^{(8)} \in \bbZ_+$ by \eqref{eq: outer_induction}. Lastly, $(c)$ follows for all $n \geq n_{c1}^{(9)}$ for some $n_{c1}^{(9)} \in \bbZ_+$ By fixing $n_{c1} \geq \max_{k \in [9]}\left\{n_{c1}^{(k)}\right\}$, the induction step is complete. This completes the proof of the claim. \hfill $\square$
\endproof
\proof{Proof of Claim \ref{claim: base_case_of_induction_hypothesis}}
For some $l \in [k-1]$, we consider the following induction hypothesis: There exists $n_{c2} \in \bbZ_+$, such that for all $n \geq n_{c2}$, we have
\begin{align*}
    \P{Z_{lk}(\bbars) \geq \sqrt{m n}\log n} \leq \left(\frac{1}{n}\right)^{m (\log n)/4-4(k-0.5)m-(k-l)}. \numberthis \label{eq: induction_hypothesis_claim_2}
\end{align*}
\textbf{Base Case:} We analyze the drift of $Z_{k-1,k}(\BFs)$ when $Z_{k-1,k}(\BFs) \geq 0$. Thus, by \eqref{eq: z_lyapunov_functions}, we have 
\begin{subequations} \label{eq: base_case_bounds_claim2}
\begin{align}
    s_{k-1} &\geq n - 2m \frac{n\log d}{d^{2}} -  (2k-1)d^{k-2}\sqrt{m n}\log n - 3km\frac{n\log d}{d^{3}}  \\
    s_k &\leq n - 2m \frac{n \log d}{d} - 2k d^{k-1}\sqrt{m n}\log n-10m^2 \frac{n \log d}{d^2}.
\end{align}
\end{subequations}
First, consider the case when $Z_{k-1,k}^{(1)}(\BFs) \geq Z_{k-1,k}^{(2)}(\BFs)$. The drift is given as follows:
\begin{align*}
    &\Delta Z_{k-1,k}(\BFs)\\
    \leq{}& \lambda\left(\left(\frac{s_{k-2}}{n}\right)^{\lfloor d \rfloor} - \left(\frac{s_{k-1}}{n}\right)^{\lfloor d \rfloor}\right) - s_{k-1} + s_{k} \\
    \overset{(a)}{\leq}{}& \lambda\left(1 - \left(1 - 2m \frac{\log d}{d^{2}} -  (2k-1)d^{k-2}\frac{\sqrt{m}\log n}{\sqrt{n}}- 3km\frac{\log d}{d^{3}} \right)^{\lfloor d \rfloor}\right) + 2m \frac{n\log d}{d^{2}} \\
    &+  (2k-1)d^{k-2}\sqrt{m n}\log n+ 3km\frac{n\log d}{d^{3}}- 2m \frac{n \log d}{d} - 2k d^{k-1}\sqrt{m n}\log n-10m^2 \frac{n \log d}{d^2} \\
    \overset{(b)}{\leq}{}& \lambda\left(2m \frac{\log d}{d} +  (2k-1) d^{k-1}\frac{\sqrt{m}\log n}{\sqrt{n}}+ 3km\frac{\log d}{d^{2}}\right) +  (2k-1) d^{k-2}\sqrt{m n}\log n\\
    &+ 3km\frac{n\log d}{d^{3}}- 2m \frac{n \log d}{d} - 2k d^{k-1}\sqrt{m n}\log n-8m^2 \frac{n \log d}{d^2} \\
    \overset{(c)}{\leq}{}& -d^{k-1}\sqrt{m n}\log n +(2k-1) d^{k-2}\sqrt{m n}\log n\\
    \overset{(d)}{\leq}{}& -\frac{1}{2}d^{k-1}\sqrt{m n}\log n \overset{(e)}{\leq} -\sqrt{m n}\log n,
\end{align*}
where $(a)$ follows by upper bounding $s_{k-2}$ by $n$ and using the bounds on $s_{k-1}$ and $s_k$ given by \eqref{eq: base_case_bounds_claim2}. Next, $(b)$ follows by Lemma \ref{lemma: power_of_d_not_going_to_zero}. Now, $(c)$ follows as 
\begin{align*}
    3km\frac{n\log d}{d^{2}} + 3km\frac{n\log d}{d^{3}} \leq 3m^2\frac{n\log d}{d^{2}} + 3m^2\frac{n\log d}{d^{3}} \leq 6m^2\frac{n\log d}{d^{2}} \leq 8m^2\frac{n\log d}{d^{2}}.
\end{align*}
Lastly, $(d)$ follows as there exists $n_{c2}^{(1)}$ such that for all $n \geq n_{c2}^{(1)}$, we have
\begin{align*}
    \frac{1}{2}d^{k-1}\sqrt{m n}\log n \geq 2m d^{k-2}\sqrt{m n}\log n \geq (2k-1) d^{k-2}\sqrt{m n}\log n,
\end{align*}
where the first inequality follows as $m / d \leq \log n/d \rightarrow 0$ as $n \rightarrow \infty$. Lastly, $(e)$ follows for all $n \geq n_{c2}^{(2)}$ for some $n_{c2}^{(2)} \in \bbZ_+$ as $k \geq 2$. Note that $k=1$ corresponds to the base case of \eqref{eq: outer_induction} which is proved in Lemma \ref{lemma: base_case_lower_bound}. Now, consider the case when $Z_{k-1,k}^{(1)}(\BFs) \leq Z_{k-1,k}^{(2)}(\BFs)$. The drift is given as follows:
\begin{align*}
    \Delta Z_{k-1,k}(\BFs) \leq{}& s_k - s_{k+1} - \lambda\left(\left(\frac{s_{k-1}}{n}\right)^{\lfloor d \rfloor} - \left(\frac{s_{k}}{n}\right)^{\lfloor d \rfloor}\right) \\
    \overset{(a)}{\leq}{}& n - 2m \frac{n \log d}{d} - 2k d^{k-1}\sqrt{m n}\log n-10m^2 \frac{n \log d}{d^2} \\
    &- \lambda\left(\left(1 - 2m \frac{\log d}{d^{2}} -  (2k-1)d^{k-2}\frac{\sqrt{m}\log n}{\sqrt{n}}-3km\frac{\log d}{d^{3}}\right)^{\lfloor d \rfloor} - \left(1 - 2m \frac{\log d}{d}\right)^{\lfloor d \rfloor}\right) \\
    \overset{(b)}{\leq}{}& n - 2m \frac{n \log d}{d} - 2k d^{k-1}\sqrt{m n}\log n-10m^2 \frac{n \log d}{d^2} \\
    &- \lambda\left(1- 2m \frac{\log d}{d} -  (2k-1)d^{k-1}\frac{\sqrt{m}\log n}{\sqrt{n}}-3km\frac{\log d}{d^{2}} -\frac{2}{d^{2m}}\right) \\
    \leq{}& -d^{k-1}\sqrt{m n}\log n-10m^2 \frac{n \log d}{d^2} + 3km\frac{n\log d}{d^{2}} + \frac{2n}{d^{2m}} + n^{1-\gamma} \\
    \overset{(c)}{\leq}{}& -d^{k-1}\sqrt{m n}\log n \leq -\sqrt{m n}\log n.
\end{align*}
where $(a)$ follows by noting that $s_{k+1} \geq 0$ and using the bounds on $s_{k-1}$ and $s_k$ given by \eqref{eq: base_case_bounds_claim2}. Next, $(b)$ follows for all $n \geq n_{c2}^{(3)}$ for some $n_{c2}^{(3)} \in \bbZ_+$ by Lemma \ref{lemma: power_of_d_going_to_zero} and Lemma \ref{lemma: power_of_d_not_going_to_zero}. Lastly, $(c)$ follows as 
\begin{align*}
    3km\frac{n\log d}{d^{2}} + \frac{2n}{d^{2m}} + n^{1-\gamma} &\leq 5m^2\frac{n\log d}{d^{2}} + n^{1-\gamma} 
    \leq 5m^2\frac{n\log d}{d^{2}} + 2m \frac{n \log d}{d^{m}} \leq 10m^2\frac{n\log d}{d^{2}},
\end{align*}
where the last inequality follows as $m \geq 2$. Thus, by the above two cases, we have $\Delta Z_{k-1,k}(\BFs) \leq -\sqrt{m n}\log n$ when $Z_{k-1,k}(\BFs) \geq 0$. Thus, by Lemma \ref{lemma: iterative_ssc}, we have
\begin{align*}
    \P{Z_{k-1, k}(\bbars) \geq \sqrt{m n}\log n} \leq \left(\frac{n}{n+\sqrt{m n}\log n}\right)^{(\sqrt{m n}\log n)/2} \leq \left(\frac{1}{n}\right)^{(m \log n)/4},
\end{align*}
where the last inequality follows by Lemma \ref{lemma: simplified_ssc_identity}. Thus, by considering $n_{c2} \geq \max_{k \in [3]}\left\{n_{c2}^{(k)}\right\}$, the base case is complete.

\textbf{Induction Step:} We analyze the drift of $Z_{i-1,k}(\BFs)$ when $Z_{i-1,k}(\BFs) \geq 0$ and $\BFs \in \calC_{i, k}^{(3)} \cap \bigcap_{l=i}^{k-1} \calD_l^{(3)}$, where $\calC_{i, k}^{(3)}$ and $\calD_l^{(3)}$ are defined in \eqref{eq: sets_induction_step_base_case}. Similar to the proof of Lemma \ref{lemma: base_case_of_induction_hypothesis} (Eq. \eqref{eq: precise_upper_bound_s1} and \eqref{eq: precise_upper_bound_sk}), we can get the following bounds on $s_i$ and $s_k$:
\begin{align*}
    s_{i} &\leq  n - 2m \frac{n\log d}{d^{k-i+1}} -  (2i+1)d^{i-1}\sqrt{m n}\log n - 3(i+1)m\frac{n\log d}{d^{k-i+2}} + \sqrt{m n}\log n, \numberthis \label{eq: induction_step_claim2_si} \\
    s_k &\leq n - 2m \frac{n \log d}{d}. \numberthis \label{eq: induction_step_claim_sk}
\end{align*}
Now, we analyze the drift for the case when $Z_{i-1,k}^{(1)}(\BFs) \geq Z_{i-1,k}^{(2)}(\BFs)$.
\begin{align*}
    &\Delta Z_{i-1,k}(\BFs) \\
    \leq{}& \lambda\left(\left(\frac{s_{i-2}}{n}\right)^{\lfloor d \rfloor} - \left(\frac{s_{i-1}}{n}\right)^{\lfloor d \rfloor}\right) - s_{i-1} + s_i \\
    \overset{(a)}{\leq}{}& \lambda\left(1 - \left(1 - 2m \frac{\log d}{d^{k-i+2}} -  (2i-1)d^{i-2}\frac{\sqrt{m}\log n}{\sqrt{n}}- 3im\frac{\log d}{d^{k-i+3}}\right)^{\lfloor d \rfloor}\right) + 2m \frac{n\log d}{d^{k-i+2}} \\
    & +  (2i-1)d^{i-2}\sqrt{m n}\log n+ 3im\frac{n\log d}{d^{k-i+3}} - 2m \frac{n\log d}{d^{k-i+1}} -  (2i+1)d^{i-1}\sqrt{m n}\log n \\
    & - 3(i+1)m\frac{n\log d}{d^{k-i+2}} + \sqrt{m n}\log n  \\
     \overset{(b)}{\leq}{}& \lambda\left(2m \frac{\log d}{d^{k-i+1}} +  (2i-1) d^{i-1}\frac{\sqrt{m}\log n}{\sqrt{n}}+ 3im\frac{\log d}{d^{k-i+2}}\right) + 2m \frac{n\log d}{d^{k-i+2}} +  (2i-1)d^{i-2}\sqrt{m n}\log n\\
     &+ 3im\frac{n\log d}{d^{k-i+3}}- 2m \frac{n\log d}{d^{k-i+1}} -  (2i+1)d^{i-1}\sqrt{m n}\log n - 3(i+1)m\frac{n\log d}{d^{k-i+2}}+ \sqrt{m n}\log n \\
    \leq& -2d^{i-1}\sqrt{m n}\log n- m\frac{n\log d}{d^{k-i+2}} +  (2i-1)d^{i-2}\sqrt{m n}\log n+ 3im\frac{n\log d}{d^{k-i+3}}  + \sqrt{m n}\log n \\
     \overset{(c)}{\leq}{}&-d^{i-1}\sqrt{m n}\log n \leq -\sqrt{m n}\log n,
\end{align*}
where $(a)$ follows by noting that $s_{i-2} \leq n$, using the bound on $s_i$ given by \eqref{eq: induction_step_claim2_si}, and bounding $s_{i-1}$ by using the fact that $Z_{i-1,k}(\BFs) \geq 0$. Next, $(b)$ follows by Lemma \ref{lemma: power_of_d_not_going_to_zero}. Lastly, $(c)$ follows as there exists $n_{c2}^{(4)} \in \bbZ_+$ such that for all $n \geq n_{c2}^{(4)}$, we have
\begin{align*}
    m\frac{n\log d}{d^{k-i+2}} \overset{(c_1)}{\geq}{}& 3m^2\frac{n\log d}{d^{k-i+3}} \geq 3im\frac{n\log d}{d^{k-i+3}} \quad \forall i \leq m \\
    d^{i-1}\sqrt{m n}\log n \overset{(c_2)}{\geq}{}& 2m d^{i-2}\sqrt{m n}\log n \geq 2i d^{i-2}\sqrt{m n}\log n \\
    \geq{}& (2i-1)d^{i-2}\sqrt{m n}\log n + \sqrt{m n}\log n \quad \forall i \leq m,
\end{align*}
where $(c_1)$ and $(c_2)$ follows as $m / d \leq \log n/d \rightarrow 0$ as $n \rightarrow \infty$. Now, consider the case when $Z_{i-1,k}^{(1)}(\BFs) \leq Z_{i-1,k}^{(2)}(\BFs)$. The drift is given as follows:
\begin{align*}
    \Delta Z_{i-1,k}(\BFs) 
    \leq{}& s_i - s_{k+1} - \lambda\left(\left(\frac{s_{i-1}}{n}\right)^{\lfloor d \rfloor} - \left(\frac{s_k}{n}\right)^{\lfloor d \rfloor}\right) \\
    \overset{(a)}{\leq}{}& n - 2m \frac{n\log d}{d^{k-i+1}} -  (2i+1)d^{i-1}\sqrt{m n}\log n - 3(i+1)m\frac{n\log d}{d^{k-i+2}}+ \sqrt{m n}\log n\\
    &-\lambda\left(\left(1 - 2m \frac{\log d}{d^{k-i+2}} -  (2i-1) d^{i-2}\frac{\sqrt{m}\log n}{\sqrt{n}}- 3im\frac{\log d}{d^{k-i+3}}\right)^{\lfloor d \rfloor} - \left(1 - 2m \frac{\log d}{d} \right)^{\lfloor d \rfloor}\right) \\
    \overset{(b)}{\leq}{}& n - 2m \frac{n\log d}{d^{k-i+1}} -  (2i+1)d^{i-1}\sqrt{m n}\log n - 3(i+1)m\frac{n\log d}{d^{k-i+2}}+ \sqrt{m n}\log n\\
    &-\lambda\left(1 - 2m \frac{\log d}{d^{k-i+1}} -  (2i-1) d^{i-1}\frac{\sqrt{m}\log n}{\sqrt{n}}- 3im\frac{\log d}{d^{k-i+2}} - \frac{2}{d^{2m}}\right) \\
    \leq{}& -2d^{i-1}\sqrt{m n}\log n -3m\frac{n\log d}{d^{k-i+2}} + \frac{2n}{d^{2m}} + n^{1-\gamma}+ \sqrt{m n}\log n \\
    \overset{(c)}{\leq}{}& -2d^{i-1}\sqrt{m n}\log n+ \sqrt{m n}\log n \leq -\sqrt{m n}\log n,
\end{align*}
where $(a)$ follows by using the bound on $s_i$ and $s_k$ given by \eqref{eq: induction_step_claim2_si} and \eqref{eq: induction_step_claim_sk}, and bounding $s_{i-1}$ by using the fact that $Z_{i-1,k}(\BFs) \geq 0$. Next, $(b)$ follows for all $n \geq n_{c2}^{(5)}$ for some $n_{c2}^{(5)} \in \bbZ_+$ by Lemma \ref{lemma: power_of_d_going_to_zero} and Lemma \ref{lemma: power_of_d_not_going_to_zero}. Lastly, $(c)$ follows as there exists $n_{c2}^{(6)}$ such that for all $n \geq n_{c2}^{(6)}$, we have
\begin{align*}
    3m\frac{n\log d}{d^{k-i+2}} \overset{(c_1)}{\geq} 3m\frac{n\log d}{d^{m}} \geq \frac{2n}{d^{2m}} + 2m\frac{n\log d}{d^{m}} \geq \frac{2n}{d^{2m}} + n^{1-\gamma},
\end{align*}
where $(c_1)$ follows as $i \geq 2$ and $k \leq m$. By the above two cases, we get $\Delta Z_{i-1,k}(\BFs) \leq -\sqrt{m n}\log n$ when $Z_{i-1,k}(\BFs)\geq 0$ and $\BFs \in \calC_{i, k}^{(3)} \cap \bigcap_{l=i}^{k-1} \calD_l^{(3)}$. Thus, by Lemma \ref{lemma: iterative_ssc}, we have
\begin{align*}
    &\P{Z_{i-1, k}(\bbars) \geq \sqrt{m n}\log n}\\
    \leq{}& \left(\frac{n}{n+\sqrt{m n}\log n}\right)^{(m \log n)/2} + \sqrt{n}\P{\bbars \notin \calC_{i, k}^{(3)} \cap \bigcap_{l=i}^{k-1} \calD_l^{(3)}} \\
    \overset{(a)}{\leq}{}& \left(\frac{1}{n}\right)^{(m \log n)/4} + \sqrt{n}\left(\P{\bbars \notin \calC_{i, k}^{(3)}} + \sum_{l=i}^{k-1} \P{\bbars \notin \calD_l^{(3)}}\right) \\
    \overset{(b)}{\leq}{}& \left(\frac{1}{n}\right)^{(m \log n)/4} + \sqrt{n}\left(\frac{1}{n}\right)^{(m \log n)/4-4(k-0.5)m-(k-i)} + \sqrt{n}m \left(\frac{1}{n}\right)^{(m \log n)/4-4(k-1)m}\\
    \overset{(c)}{\leq}{}& \left(\frac{1}{n}\right)^{(m\log n)/4-4(k-0.5)m-(k-i+1)},
\end{align*}
where $(a)$ follows by Lemma \ref{lemma: simplified_ssc_identity}. Next, $(b)$ follows by upper bounding $\P{\bbars \notin \calC_{i, k}^{(3)}}$ using the inducting hypothesis given by \eqref{eq: induction_hypothesis_claim_2}. Also, $\P{\bbars \notin \calD_l^{(3)}}$ is upper bounded for all $n\geq n_2$ by Lemma \ref{lemma: remaining_weak_outer_induction}. Lastly, $(c)$ follows for all $n \geq n_{c2}^{(7)}$ for some $n_{c2}^{(7)} \in \bbZ_+$. By fixing $n_{c2} \geq \max_{k \in [7]}\left\{n_{c2}^{(k)},n_2\right\}$, the induction step is complete.  \hfill $\square$
\endproof
\section{Proof of Claims for Theorem \ref{theo: upper_bound}} \label{app: claim_ub}
\proof{Proof of Claim \ref{claim: upper_bound}} The proof is induction based. The induction hypothesis is as follows. There exists $\tilde{n}_c \in \bbZ_+$ such that for all $n \geq \tilde{n}_c$, we have
\begin{align}
    \P{U_j(\bbars) \geq \sqrt{m n}\log n} \leq \left(\frac{1}{n}\right)^{\frac{m \log n}{\max\{x,5\}}-(m -j+1)}. \label{eq: induction_hypothesis_upper_bound_claim}
\end{align}

\textbf{Base Case} $(j=m)$: We analyze the drift of $U_{m}(\BFs)$ as defined in \eqref{eq: u_lyapunov_functions} when $U_{m}(\BFs) \geq 0$ and $\BFs \in \tilde{\calD}_{m - 1} \cap \tilde{\calD}_{m +2} $ where
\begin{align}
    \tilde{\calD}_{m - 1} &= \left\{s_{m - 1} \geq n - \left(2m \frac{n\log d}{d^2} + 4m d^{m-2}\sqrt{m n}\log n + 16m^3 \frac{n\log (d)^2}{d^{3}}\right)\mathbbm{1}\left\{m > 1\right\}\right\}. \label{eq: base_case_bound_smminus1_claim} \\
    \tilde{D}_{m+2} &= \left\{\sum_{l=m+2}^b s_l \leq B_{m+2}\right\}. \nonumber
\end{align} 
As $U_{m}(\BFs) \geq 0$, we get the following bounds on $s_{m}$ and $s_{m + 1}$:
\begin{subequations} \label{eq: base_case_bounds_claim}
\begin{align}
    s_{m+1} &\geq 8m d^{m-1}\sqrt{m n}\log n + 24m^3 \frac{n \log (d)^2}{d^{2}}+n^{1-\gamma}\mathbbm{1}\left\{m > 1\right\}-2 \label{eq: base_case_bounds_claim_s_mplus1} \\
    s_{m} &\leq n - 2m \frac{n\log d}{d} + 3m d^{m-1}\sqrt{m n}\log n+7m^3\frac{n\log (d)^2}{d^{2}},
\end{align}
\end{subequations}
where \eqref{eq: base_case_bounds_claim_s_mplus1} follows by considering two cases. If $B_{m+2} \geq 2$, then we use the bound $\sum_{l=m+1}^b s_l \leq bs_{m+1}$ to obtain  \eqref{eq: base_case_bounds_claim_s_mplus1}. Else if $B_{m+2} < 2$, then we use the bound $\sum_{l=m+1}^b s_l \leq s_{m+1}+B_{m+2} \leq s_{m+1}+2$ to obtain  \eqref{eq: base_case_bounds_claim_s_mplus1}. First, consider the case when $U^{(1)}_{m}(\BFs) \leq U^{(2)}_{m}(\BFs)$. In this case, the drift is as follows:
\begin{align*}
    &\Delta U_{m}(\BFs) \\
    \leq{}& -s_{m + 1} + \lambda\left(\left(\frac{s_{m}}{n}\right)^{\lfloor d \rfloor} - \left(\frac{s_b}{n}\right)^{\lfloor d \rfloor}\right) \\
    \overset{(a)}{\leq}{}& - 8m d^{m-1}\sqrt{m n}\log n - 24m^3 \frac{n \log (d)^2}{d^{2}}+2 + n\left(1 - 2m \frac{\log d}{d} + 3m \sqrt{m} d^{m-1}\frac{\log n}{\sqrt{n}}+7m^3\frac{\log (d)^2}{d^2}\right)^{\lfloor d \rfloor} \\
    \overset{(b)}{\leq}{}& -8m d^{m-1}\sqrt{m n}\log n - 24m^3 \frac{n \log (d)^2}{d^{2}}+2 + \frac{2n}{d^{2m}} \overset{(c)}{\leq} -\sqrt{m n}\log n,
\end{align*}
where $(a)$ follows as $s_b \geq 0$, and $s_{m}$ and $s_{m + 1}$ are bounded as in \eqref{eq: base_case_bounds_claim}. Next, $(b)$ follows by Lemma \ref{lemma: power_of_d_going_to_zero}. Lastly, $(c)$ follows for all $n \geq \tilde{n}_c^{(1)}$ for some $\tilde{n}_c^{(1)} \in \bbZ_+$ as $24m^3 \frac{n \log (d)^2}{d^{2}} \geq \frac{2n}{d^{2m}} + 2$ and $8m d^{m-1}\sqrt{m n}\log n \geq \sqrt{m n}\log n$. Now, consider the case when $U^{(1)}_{m}(\BFs) \geq U^{(2)}_{m}(\BFs)$. In this case, the drift is as follows:
\begin{align*}
    &\Delta U_{m}(\BFs) \\
    \leq{}& s_{m} - s_{m + 1} - \lambda\left(\left(\frac{s_{m-1}}{n}\right)^{\lfloor d \rfloor} - \left(\frac{s_{m}}{n}\right)^{\lfloor d \rfloor}\right) \\
    \overset{(a)}{\leq}{}& n - 2m \frac{n\log d}{d} + 3m d^{m-1}\sqrt{m n}\log n+7m^3\frac{n\log (d)^2}{d^2}-8m d^{m-1}\sqrt{m n}\log n - 24m^3 \frac{n \log (d)^2}{d^{2}} \\
    &-n^{1-\gamma}\mathbbm{1}\left\{m > 1\right\}+2-\lambda\left(1- \left(2m \frac{\log d}{d^2} + 4m d^{m-2}\frac{\sqrt{m}\log n}{\sqrt{n}} + 16m^3 \frac{\log (d)^2}{d^{3}}\right)\mathbbm{1}\left\{m > 1\right\}\right)^{\lfloor d \rfloor} \\
    &+ \lambda\left(1 - 2m \frac{\log d}{d} + 3m d^{m-1}\frac{\sqrt{m}\log n}{\sqrt{n}}+7m^3\frac{\log (d)^2}{d^2}\right)^{\lfloor d \rfloor} \\
    \overset{(b)}{\leq}{}& n - 2m \frac{n\log d}{d} - 5m d^{m-1}\sqrt{m n}\log n - 17m^3 \frac{n \log (d)^2}{d^{2}}-n^{1-\gamma}\mathbbm{1}\left\{m > 1\right\}+2 \\
    &-\lambda\left(1- 2m \frac{\log d}{d}\mathbbm{1}\left\{m > 1\right\} - 4m d^{m-1}\frac{\sqrt{m}\log n}{\sqrt{n}} -16m^3 \frac{\log (d)^2}{d^{2}} - \frac{2}{d^{2 m}}\right) \\
    \overset{(c)}{\leq}{}& -m d^{m-1}\sqrt{m n}\log n-m^3 \frac{n \log (d)^2}{d^{2}} + \frac{2n}{d^{2 m}}+2 \\
    \overset{(d)}{\leq}{}& -\sqrt{m n}\log n,
\end{align*}
where $(a)$ follows by substituting bounds on $s_{m -1}, s_{m}$, and $s_{m +1}$ given by \eqref{eq: base_case_bound_smminus1_claim} and \eqref{eq: base_case_bounds_claim}. Next, $(b)$ follows by Lemma \ref{lemma: power_of_d_not_going_to_zero} and Lemma \ref{lemma: power_of_d_going_to_zero}. Now, $(c)$ follows as $2 m n \log d / d \geq n^{1-\gamma}$. Lastly, $(d)$ follows for all $n \geq \tilde{n}_c^{(2)}$ for some $\tilde{n}_c^{(2)} \in \bbZ_+$ as $m^3 \frac{n \log (d)^2}{d^{2}} \geq \frac{2n}{d^{2 m}}+2$ and $m d^{m-1}\sqrt{m n}\log n \geq \sqrt{m n}\log n$. Thus, by the above two cases, we have $\Delta U_{m}(\BFs) \leq -\sqrt{m n}\log n$ when $U_{m}(\BFs) \geq 0$ and $\BFs \in \tilde{\calD}_{m - 1} \cap \tilde{\calD}_{m+2}$. Combining the two cases and using Lemma \ref{lemma: iterative_ssc}, we get
\begin{align*}
    \P{U_{m}(\bbars) \geq \sqrt{m n}\log n} &\leq \left(\frac{n}{n+\sqrt{m n}\log n}\right)^{(\sqrt{m n}\log n)/2}+\P{\bbars \notin \tilde{\calD}_{m - 1}\cap \tilde{\calD}_{m+2}} \\
    &\overset{(a)}{\leq} \left(\frac{1}{n}\right)^{(m\log n)/4} + \sqrt{n}\P{\bbars \notin \tilde{\calD}_{m - 1}}+\sqrt{n}\P{\bbars \notin \tilde{\calD}_{m+2}} \\
    &\overset{(b)}{\leq} \left(\frac{1}{n}\right)^{(m\log n)/4} + \left(\frac{1}{n}\right)^{(m\log n)/5-0.5} +\sqrt{n}\P{\bbars \notin \tilde{\calD}_{m+2}}\\
    &\overset{(c)}{\leq} \left(\frac{1}{n}\right)^{(m\log n)/4} + \left(\frac{1}{n}\right)^{(m\log n)/5-0.5} +\left(\frac{1}{n}\right)^{(m\log n)/x-0.5}\\
    &\overset{(d)}{\leq} \left(\frac{1}{n}\right)^{\frac{m\log n}{\max\{x,5\}}-1},
\end{align*}
where $(a)$ follows by Lemma \ref{lemma: simplified_ssc_identity}. Next, if $m=1$, then $(b)$ follows trivially as $\P{\bbars \notin \tilde{\calD}_{m - 1}} = 0$. Else if, $m>1$, then $(b)$ follows for all $n \geq n_{LB}$ by Theorem \ref{theo: lower_bound}. Now, $(c)$ follows by the high probability upper bound on $\sum_{l=m+2}^b \bars_{l}$ assumed in the statement of the lemma. Lastly, $(d)$ follows for all $n \geq \tilde{n}_c^{(3)}$ for some $\tilde{n}_c^{(3)} \in \bbZ_+$. By considering $\tilde{n}_c \geq \max_{k \in [3]}\{\tilde{n}_c^{(k)}, n_{LB}\}$, the base case is complete. Note that the base case completes the proof of the claim if $m = 1$. So now we assume $m \geq 2$.

\textbf{Induction Step:} For $j \in [m]$, we analyze the drift of $U_{j-1}(\BFs)$ as defined in \eqref{eq: u_lyapunov_functions} when $U_{j-1}(\BFs) \geq 0$ and $\BFs \in \tilde{\calC}_{j}^{(1)} \cap \tilde{\calD}_{j-2}^{(1)}$ where
\begin{align*}
    \tilde{\calC}_{j}^{(1)} &= \left\{U_{j}(\BFs) \leq \sqrt{m n}\log n\right\} \\
    \tilde{\calD}_{j-2}^{(1)} &= \left\{s_{j-2} \geq n - 2m \frac{n\log d}{d^{m-j+3}} - 4m d^{j-3}\sqrt{m n}\log n -16m^3 \frac{n\log (d)^2}{d^{m-j+4}}\right\}. \numberthis \label{eq: induction_step_bound_sjminus2_claim}
\end{align*}
Now, we obtain a useful lower bound on $s_j$ as follows:
\begin{align*}
    U_{j-1}(\BFs) \geq 0 &\Rightarrow L_{m +1}(\BFs)-\sum_{l=j+1}^{m} L_l(\BFs)  \geq  L_j(\BFs) \\
    &\overset{(*)}{\Rightarrow} L_j(\BFs) \leq \sqrt{m n}\log n \\
    &\Rightarrow s_{j} \geq n - 2m \frac{n\log d}{d^{m-j+1}} + 3j d^{j-1}\sqrt{m n}\log n+ 7jm^2\frac{n\log (d)^2}{d^{m-j+2}} - \sqrt{m n}\log n, \numberthis \label{eq: induction_step_bound_sj_claim}
\end{align*}
where $(*)$ follows as $\BFs \in \tilde{\calC}^{(1)}_j$. Next, we obtain a useful upper bound on $s_{j-1}$ as follows:
\begin{align*}
    U_{j-1}(\BFs) \geq 0 \Rightarrow s_{j-1} \leq n - 2m \frac{n\log d}{d^{m-j+2}} + 3(j-1) d^{j-2}\sqrt{m n}\log n+ 7(j-1)m^2\frac{n\log (d)^2}{d^{m-j+3}}. \numberthis \label{eq: induction_step_bound_sjminus1_claim}
\end{align*}
First, consider the case when $U_{j-1}^{(1)}(\BFs) \leq U_{j-1}^{(2)}(\BFs)$. The drift is given as follows:
\begin{align*}
    \Delta U_{j-1}(\BFs) \leq{}& -s_j + \lambda\left(\left(\frac{s_{j-1}}{n}\right)^{\lfloor d \rfloor} - \left(\frac{s_b}{n}\right)^{\lfloor d \rfloor}\right) \\
    \overset{(a)}{\leq}{}& - n + 2m \frac{n\log d}{d^{m-j+1}} - 3j d^{j-1}\sqrt{m n}\log n- 7jm^2\frac{n\log (d)^2}{d^{m-j+2}} + \sqrt{m n}\log n \\
    &+n\left(1 - 2m \frac{\log d}{d^{m-j+2}} + 3(j-1) d^{j-2}\frac{\sqrt{m}\log n}{\sqrt{n}}+ 7(j-1)m^2\frac{\log (d)^2}{d^{m-j+3}}\right)^{\lfloor d \rfloor} \\
    \overset{(b)}{\leq}{}&- n + 2m \frac{n\log d}{d^{m-j+1}} - 3j d^{j-1}\sqrt{m n}\log n- 7jm^2\frac{n\log (d)^2}{d^{m-j+2}} + \sqrt{m n}\log n \\
    &+n\left(1 - 2m \frac{\lfloor d \rfloor\log d}{d^{m-j+2}} + 3(j-1) d^{j-2}\frac{\lfloor d \rfloor\sqrt{m}\log n}{\sqrt{n}}+ 7(j-1)m^2\frac{\lfloor d \rfloor\log (d)^2}{d^{m-j+3}}\right) \\
    &+ \frac{3n}{2} \left(4m^2 \frac{\log (d)^2}{d^{2m-2j+2}} + 9(j-1)^2 d^{2j-2}\frac{m\log (n)^2}{n}+ 49(j-1)^2m^4\frac{\log (d)^4}{d^{2m-2j+4}}\right) \\
    \leq{}& 2m \frac{n\log d}{d^{m-j+2}}-3 d^{j-1}\sqrt{m n}\log n - 7m^2\frac{n\log (d)^2}{d^{m -j+2}} + \sqrt{m n}\log n+6m^2 \frac{n \log (d)^2}{d^{2m-2j+2}}\\
    &+ 13.5 m^3 d^{2j-2} \log (n)^2 + 147m^6\frac{n \log (d)^4}{2d^{2m-2j+4}} \\
    \overset{(c)}{\leq}{}& - \sqrt{m n}\log n,
\end{align*}
where $(a)$ follows as $s_b \geq 0$, $\lambda \leq n$, and $s_{j-1}$ and $s_j$ are bounded as in \eqref{eq: induction_step_bound_sjminus1_claim} and \eqref{eq: induction_step_bound_sj_claim} respectively. Next, $(b)$ follows by Lemma \ref{lemma: power_of_d_not_going_to_zero} for all $n \geq \tilde{n}_c^{(4)}$ for some $\tilde{n}_c^{(4)} \in \bbZ_+$ independent of $j \in [n]$ as 
\begin{align*}
    2m \frac{\log d}{d^{m-j+2}} - 3(j-1) d^{j-2}\frac{\sqrt{m}\log n}{\sqrt{n}}- 7(j-1)m^2\frac{\log (d)^2}{d^{m-j+3}} \geq 0
\end{align*} 
for $n$ large enough. Note that we also use the inequality $(a+b+c)^2 \leq 3(a^2+b^2+c^2)$. Next, $(c)$ holds for $n \geq \tilde{n}_c^{(5)}$ for some $\tilde{n}_c^{(5)}$ independent of $j \in [m]$ as
\begin{align*}
    \frac{6.5m^2 n \log (d)^2}{d^{m-j+2}} &\overset{(c_1)}{\geq} \frac{2m n \log d}{d^{m-j+2}} +
    \frac{6m^2 n \log (d)^2}{d^{m-j+2}} \overset{(c_2)}{\geq} \frac{2m n \log d}{d^{m-j+2}} +   \frac{6m^2 n \log (d)^2}{d^{2m-2j+2}} \\
    \frac{m^2 n \log (d)^2}{2d^{m-j+2}} &\overset{(c_3)}{\geq} \frac{m^2 n \log (d)^2}{2d^{m-j+2}} \times \frac{147m^4 \log (d)^2}{d^{m-j+2}} = 147m^6\frac{n \log (d)^4}{2d^{2m-2j+4}} \\
    3 d^{j-1}\sqrt{m n}\log n &\overset{(c_4)}{\geq} 2\sqrt{m n}\log n +  d^{j-1}\sqrt{m n}\log n \times   27 m^{3.5} n^{\gamma-0.5} \log (n)^2 \\
    &\overset{(c_5)}{\geq}  2\sqrt{m n}\log n +  d^{j-1}\sqrt{m n}\log n \times   \frac{27 m^{2.5} d^{m} \log n}{2\sqrt{n}} \\
    &\overset{(c_6)}{\geq}  2\sqrt{m n}\log n + 13.5 m^3 d^{2j-2} \log (n)^2,
\end{align*}
where $(c_1)$ holds as $d \rightarrow \infty$ as $n \rightarrow \infty$ and $(c_2)$ follows as  $j \leq m$. Next, $(c_3)$ holds as $j \leq m$, and so, $m^4 \log (d)^2 / d^{m-j+2} \leq \log (n)^4 / d^2 \rightarrow 0$ as $n \rightarrow \infty$. Further, $(c_4)$ holds as $j \geq 1$ and $\gamma < 0.5$ and $(c_5)$ holds as $n^\gamma \geq d^m / (2m \log n)$. Lastly, $(c_6)$ holds as $j \leq m$. Now, consider the case when $U_{j-1}^{(1)}(\BFs) \geq U_{j-1}^{(2)}(\BFs)$. The drift is given as follows:
\begin{align*}
    &\Delta U_{j-1}(\BFs) \\
    \leq{}& s_{j-1} - s_j - \lambda\left(\left(\frac{s_{j-2}}{n}\right)^{\lfloor d \rfloor} - \left(\frac{s_{j-1}}{n}\right)^{\lfloor d \rfloor}\right) \\
    \overset{(a)}{\leq}{}&  - 2m \frac{n\log d}{d^{m-j+2}} + 3(j-1) d^{j-2}\sqrt{m n}\log n+ 7(j-1)m^2\frac{n\log (d)^2}{d^{m-j+3}} \\
    &+ 2m \frac{n\log d}{d^{m-j+1}} - 3j d^{j-1}\sqrt{m n}\log n- 7jm^2\frac{n\log (d)^2}{d^{m-j+2}} + \sqrt{m n}\log n \\
    &-\lambda\left(1 - 2m \frac{\log d}{d^{m-j+3}} - 4m d^{j-3}\frac{\sqrt{m}\log n}{\sqrt{n}} -16m^3 \frac{\log (d)^2}{d^{m-j+4}}\right)^{\lfloor d \rfloor}\\
    &+\lambda\left(1 - 2m \frac{\log d}{d^{m-j+2}} + 3(j-1) d^{j-2}\frac{\sqrt{m}\log n}{\sqrt{n}}+ 7(j-1)m^2\frac{\log (d)^2}{d^{m-j+3}}\right)^{\lfloor d \rfloor} \\
   \overset{(b)}{\leq}{}&  - 2m \frac{n\log d}{d^{m-j+2}} + 3(j-1) d^{j-2}\sqrt{m n}\log n+ 7(j-1)m^2\frac{n\log (d)^2}{d^{m-j+3}} \\
    &+ 2m \frac{n\log d}{d^{m-j+1}} - 3j d^{j-1}\sqrt{m n}\log n- 7jm^2\frac{n\log (d)^2}{d^{m-j+2}} + \sqrt{m n}\log n \\
    &-\lambda\left(1 - 2m \frac{\log d}{d^{m-j+2}} - 4m d^{j-2}\frac{\sqrt{m}\log n}{\sqrt{n}} -16m^3 \frac{\log (d)^2}{d^{m-j+3}}\right)\\
    &+\lambda\left(1 - 2m \frac{\lfloor d \rfloor\log d}{d^{m-j+2}} + 3(j-1) d^{j-2}\frac{\lfloor d \rfloor\sqrt{m}\log n}{\sqrt{n}}+ 7(j-1)m^2\frac{\lfloor d \rfloor\log (d)^2}{d^{m-j+3}}\right) \\
    &+\frac{3\lambda}{2}\left(4m^2 \frac{\log (d)^2}{d^{2m-2j+2}} + 9m(j-1)^2 d^{2j-2}\frac{\log (n)^2}{n}+ 49(j-1)^2m^4\frac{\log (d)^4}{d^{2m-2j+4}}\right) \\
   \leq{}& 2m \frac{n \log d}{d^{m-j+2}}-3d^{j-1}\sqrt{m n}\log n-7m^2\frac{n\log (d)^2}{d^{m-j+2}}+\left(4 m+3(j-1)\right) d^{j-2}\sqrt{m n}\log n\\
    &+m^2\left(16m+7j-7\right) \frac{n\log (d)^2}{d^{m-j+3}}+\sqrt{m n}\log n+6m^2 \frac{n\log (d)^2}{d^{2m-2j+2}} + \frac{27m}{2}(j-1)^2 d^{2j-2}\log (n)^2\\
    &+ \frac{147}{2}(j-1)^2m^4\frac{n\log (d)^4}{d^{2m-2j+4}} + 2m \frac{n^{1-\gamma} \log d}{d^{m-j+1}} \\
    \overset{(c)}{\leq}{}& -3d^{j-1}\sqrt{m n}\log n+\left(4 m+3(j-1)\right) d^{j-2}\sqrt{m n}\log n + \sqrt{m n}\log n \\
    \overset{(d)}{\leq}{}& -d^{j-1}\sqrt{m n}\log n \leq -\sqrt{m n}\log n,
\end{align*}
where $(a)$ follows by using the bounds on $s_{j-2}$, $s_{j-1}$ and $s_j$ given by \eqref{eq: induction_step_bound_sjminus2_claim}, \eqref{eq: induction_step_bound_sjminus1_claim} and \eqref{eq: induction_step_bound_sj_claim} respectively. Next, $(b)$ follows by Lemma \ref{lemma: power_of_d_not_going_to_zero} for $n \geq \tilde{n}_c^{(6)}$ for some $\tilde{n}_c^{(6)} \in \bbZ_+$ independent of $j \in [m]$ as
\begin{align*}
    2m \frac{\lfloor d \rfloor\log d}{d^{m-j+2}} - 3(j-1) d^{j-2}\frac{\lfloor d \rfloor\sqrt{m}\log n}{\sqrt{n}}- 7(j-1)m^2\frac{\lfloor d \rfloor\log (d)^2}{d^{m-j+3}} \geq 0
\end{align*}
for $n$ large enough. Now, $(c)$ follows as there exists $\tilde{n}_c^{(7)} \in \bbZ_+$ independent of $j$ such that for all $n \geq \tilde{n}_c^{(7)}$, we have
\begin{align*}
    \frac{m^2}{4}\frac{n\log (d)^2}{d^{m-j+2}} &\overset{(c_1)}{\geq} 23m^3 \frac{n\log (d)^2}{d^{m-j+3}} \geq m^2\left(16m+7(j-1)\right) \frac{n\log (d)^2}{d^{m-j+3}} \\
    6.25\frac{m^2n\log (d)^2}{d^{m-j+2}} &\overset{(c_2)}{\geq} 6m^2 \frac{n\log (d)^2}{d^{2m-2j+2}} +  \frac{m^2n\log (d)^2}{4d^{m-j+2}}\overset{(c_3)}{\geq} 6m^2 \frac{n\log (d)^2}{d^{2m-2j+2}} + \frac{147}{2}(j-1)^2m^4\frac{n\log (d)^4}{d^{2m-2j+4}} \\
    \frac{m^2}{4}\frac{n\log (d)^2}{d^{m-j+2}} &\overset{(c_4)}{=} \frac{1}{16}d^{m + j-2}n^{1-2\gamma} \overset{(c_5)}{\geq} \frac{1}{16}d^{2j-2}n^{1-2\gamma} \overset{(c_6)}{\geq} \frac{27m}{2}(j-1)^2 d^{2j-2}\log (n)^2 \\
    \frac{m^2}{4}\frac{n\log (d)^2}{d^{m-j+2}} &\overset{(c_7)}{\geq} 4m \frac{n \log d}{d^{m-j+2}} \overset{(c_8)}{\geq} 2m \frac{n \log d}{d^{m-j+2}} + 2m \frac{dn^{1-\gamma} \log d}{d^{m-j+2}},
\end{align*}
where $(c_1)$ follows as $m \log d / d \leq \log n/d \rightarrow 0$ as $n \rightarrow \infty$. Next, $(c_2)$ follows as $j \leq m$ and $(c_3)$ follows as $(j-1)^2m^2 \log(d)^2/d^2 \leq m^4\log(d)^2 / d^2 \leq \log (n)^4/d^2 \rightarrow 0$ with $n \rightarrow \infty$ as $m \leq \log n/\log d$. Now, $(c_4)$ follows by noting that $2m n \log d / d^{m}=n^{1-\gamma}$ as $m$ is assumed to be an integer, $(c_5)$ follows as $m \geq j$, and $(c_6)$ follows as $1-2\gamma > 0$. Lastly, $(c_7)$ holds as $d \rightarrow \infty$ as $n \rightarrow \infty$ and $(c_8)$ holds as $d n^{-\gamma} \leq (2m \log n) n^{-\gamma + \gamma/m} \rightarrow 0$ as $m \geq 2$.

Now, $(d)$ follows as there exists $\tilde{n}_c^{(8)} \in \bbZ_+$ independent of $j$ such that for all $n \geq \tilde{n}_c^{(8)}$, we have
\begin{align*}
    d^{j-1}\sqrt{m n}\log n &\geq \sqrt{m n}\log n \\
    d^{j-1}\sqrt{m n}\log n &\overset{(*)}{\geq} 7m d^{j-2}\sqrt{m n}\log n \geq \left(4 m+3(j-1)\right) d^{j-2}\sqrt{m n}\log n,
\end{align*}
where $(*)$ follows as $m / d \rightarrow 0$ as $n \rightarrow \infty$. By combining the above two cases, we get $\Delta U_{j-1}(\BFs) \leq -\sqrt{m n}\log n$ when $U_{j-1}(\BFs) \geq 0$ and $\BFs \in \tilde{\calC}_{j}^{(1)} \cap \tilde{\calD}_{j-2}^{(1)}$. Now, for all $n \geq \max_{k \in [8]}\{\tilde{n}_c^{(k)}\}$, using Lemma \ref{lemma: iterative_ssc}, we get
\begin{align*}
    \P{U_{j-1}(\bbars) \geq \sqrt{m n}\log n} &\leq \left(\frac{n}{n+\sqrt{m n}\log n}\right)^{(\sqrt{m n}\log n)/2} + \sqrt{n}\P{\bbars \notin \tilde{\calC}_{j}^{(1)} \cap \tilde{\calD}_{j-2}^{(1)}} \\
    &\overset{(a)}{\leq} \left(\frac{1}{n}\right)^{(m \log n)/4} + \sqrt{n}\left(\P{\bbars \notin \tilde{\calC}_{j}^{(1)}} + \P{\bbars \notin \tilde{\calD}_{j-2}^{(1)}}\right) \\
    &\overset{(b)}{\leq} \left(\frac{1}{n}\right)^{(m \log n)/4} + \left(\frac{1}{n}\right)^{\frac{m \log n}{\max\{x,5\}}-(m +1 -j)-0.5} + \sqrt{n}\P{\bbars \notin \tilde{\calD}_{j-2}^{(1)}} \\
    &\overset{(c)}{\leq} \left(\frac{1}{n}\right)^{(m \log n)/4} + \left(\frac{1}{n}\right)^{\frac{m \log n}{\max\{x,5\}}-(m +1 -j)-0.5} + \left(\frac{1}{n}\right)^{(m \log n)/5-0.5} \\
    &\overset{(d)}{\leq} \left(\frac{1}{n}\right)^{\frac{m \log n}{\max\{x,5\}}-(m+2-j)},
\end{align*}
where $(a)$ follows by Lemma \ref{lemma: simplified_ssc_identity}, $(b)$ follows by the induction hypothesis \eqref{eq: induction_hypothesis_upper_bound_claim}, and $(c)$ follows by Theorem \ref{theo: lower_bound}. Lastly, $(d)$ follows for all $n \geq \tilde{n}_c^{(9)}$ for some $\tilde{n}_c^{(9)} \in \bbZ_+$. By considering $\tilde{n}_c \geq \max_{k \in [9]}\{\tilde{n}_c^{(k)}\}$, the induction step is complete. \hfill $\square$
\endproof
\section{Proof of Preliminary Lemmas} \label{app: taylor_series}
\proof{Proof of Lemma \ref{lemma: power_of_d_going_to_zero}}
Let $d_0$ be such that $(r\log d -df(d))/d < 1$ as $df(d) \rightarrow 0$ and $r\log d/d \rightarrow 0$. Now, for all $d \geq d_0$, we have
\begin{align*}
    \log \left(d^r \left(1 - r\frac{\log d}{d}+f(d)\right)^{\lfloor d \rfloor}\right) &= \lfloor d \rfloor\log\left(1 - r\frac{\log d}{d}+f(d)\right) + r\log d  \\
    &\leq r\log d \left(1-\frac{\lfloor d \rfloor}{d}\right) + \lfloor d \rfloor f(d) \\
    &\leq \frac{r\log d}{d}  + d f(d) \rightarrow 0 \textit{ as } d \rightarrow \infty, \numberthis \label{eq: converging_to_zero_app_lemma}
\end{align*}
where the last inequality follows as $\log(1+x) \leq x$ for $x > -1$. Lastly, note that \eqref{eq: converging_to_zero_app_lemma} implies that $\limsup_{d \rightarrow \infty} d^r\left(1 - r\frac{\log d}{d}+f(d)\right)^d \leq 1$ which completes the proof. \hfill $\square$
\endproof
\proof{Proof of Lemma \ref{lemma: power_of_d_not_going_to_zero}}
Note that $1-xd \leq 1-x\lfloor d \rfloor \leq (1-x)^{\lfloor d \rfloor}$ holds for all $0 \leq x \leq 1$ by the Bernoulli's inequality. This completes the first part of the lemma. Now to prove the second part of the lemma, by Binomial series expansion, we have
\begin{align*}
    \left(1-f(d)\right)^{\lfloor d \rfloor} ={}& \sum_{k=0}^\infty \binom{\lfloor d \rfloor}{k} (-f(d))^k \\
    ={}&1 - \lfloor d \rfloor f(d) + \frac{1}{2}\lfloor d \rfloor(\lfloor d \rfloor-1)f(d)^2 - \frac{1}{6}\lfloor d \rfloor(\lfloor d \rfloor-1)(\lfloor d \rfloor-2)f(d)^3 +  \sum_{k=4}^\infty \binom{\lfloor d \rfloor}{k} (-f(d))^k \\
    \overset{(a)}{}&{\leq} 1 - \lfloor d \rfloor f(d) + \frac{1}{2}(\lfloor d \rfloor)^2f(d)^2 - \frac{1}{6}\lfloor d \rfloor(\lfloor d \rfloor-1)(\lfloor d \rfloor-2)f(d)^3 \\
    &+ \frac{1}{6}\lfloor d \rfloor(\lfloor d \rfloor-1)(\lfloor d \rfloor-2)f(d)^3\sum_{k=1}^\infty (\lfloor d \rfloor f(d))^k \\
    ={}&1 - \lfloor d \rfloor f(d) + \frac{1}{2}(\lfloor d \rfloor)^2f(d)^2 - \frac{1}{6}\lfloor d \rfloor(\lfloor d \rfloor-1)(\lfloor d \rfloor-2)f(d)^3\left(1 - \frac{\lfloor d \rfloor f(d)}{1-\lfloor d \rfloor f(d)}\right) \\
    \overset{(b)}{}&{\leq} 1 - \lfloor d \rfloor f(d) + \frac{1}{2}(\lfloor d \rfloor)^2f(d)^2 - \frac{1}{12}\lfloor d \rfloor(\lfloor d \rfloor-1)(\lfloor d \rfloor-2)f(d)^3 \\
    \overset{(c)}{}&{\leq} 1 - \lfloor d \rfloor f(d) + \frac{1}{2}d^2f(d)^2,
\end{align*}
where $(a)$ follows as $f(d) \geq 0$ for all $d \geq d_1$, and we have $\binom{\lfloor d \rfloor}{k} \leq (\lfloor d \rfloor(\lfloor d \rfloor-1)(\lfloor d \rfloor-2)\lfloor d \rfloor^{k-3}/6)$ for all $k \geq 4$. Next, $(b)$ follows for all $d \geq d_2$ for some $d_2 \geq d_1$ as $\lfloor d \rfloor f(d) \rightarrow 0$. Lastly, $(c)$ follows for all $d \geq d_1$ as $f(d) \geq 0$ and $\lfloor d \rfloor \leq d$. This completes the proof. \hfill $\square$
\endproof
\proof{Proof of Lemma \ref{lemma: simplified_ssc_identity}} We have
\begin{align*}
    \left(\frac{n}{n+\sqrt{m n}\log n}\right)^{(\sqrt{m n} \log n)/2} &= \left(1+\frac{\sqrt{m}\log n}{\sqrt{ n}}\right)^{-(\sqrt{m n}\log n)/2} 
    = e^{-m\log (n)^2/2 \left(\frac{\log(1+\sqrt{m}\log n/\sqrt{n})}{\sqrt{m}\log n/\sqrt{n}}\right)} 
    \\
    &\leq e^{-m\log (n)^2/4}
    =\left(\frac{1}{n}\right)^{(m\log n)/4},
\end{align*}
where the inequality holds due to the following. Observe that $m \leq \log n/\log d$ by \eqref{eq: m}. Thus, $\sqrt{m}\log n/\sqrt{n} \rightarrow 0$, which implies $\log(1+\sqrt{m}\log n/\sqrt{n})/(\sqrt{m}\log n/\sqrt{n}) \rightarrow 1$. Thus, there exists $n_a \in \bbZ_+$ such that for all $n \geq n_a$, we have $\log(1+\sqrt{m}\log n/\sqrt{n})/(\sqrt{m}\log n/\sqrt{n}) \geq 0.5$. This completes the proof. \hfill $\square$
\endproof
\section{Discussion on Lower Order Terms} \label{app: lower_order_terms}
Note that the leading order term in the concentration bounds of Theorem~\ref{theo: informal} is equal to the fixed point of the deterministic dynamical system $\left(\frac{\lambda}{n}\right)^{\frac{d^i-1}{d-1}}$. Here we discuss the order of the lower order terms that one should expect and compare it to the lower order terms obtained in Theorem \ref{theo: lower_bound} and \ref{theo: upper_bound}. For any $i \in [m]$, observe that $s_i$ can be approximated as a queue with arrivals governed by a Poisson process with rate $\lambda\left(\left(\frac{s_{i-1}}{n}\right)^d-\left(\frac{s_{i}}{n}\right)^d\right)$ and the service given by an exponential distribution with rate $s_i-s_{i+1}$. By setting $s_i \approx n\left(\frac{\lambda}{n}\right)^{\frac{d^i-1}{d-1}} \approx n\left(1-n^{-\gamma}d^{i-1}\right)$, the arrival process can be approximated by a Poisson process with rate $\lambda n^{-\gamma}\left(d^{i}-d^{i-1}\right) \approx n^{1-\gamma} d^i \approx n^{1-\gamma+\gamma i/m}$ as $d \approx n^{\gamma/m}$. Similarly, the service process can approximated as an exponential distribution with rate $n^{1-\gamma}\left(d^{i}-d^{i+1}\right) \approx n^{1-\gamma}d^i \approx n^{1-\gamma+\gamma i/m}$. Thus, we should expect the standard deviation of $s_i$ to be equal to $\sqrt{n^{1-\gamma+\gamma i/m}} = n^{0.5 - \frac{\gamma}{2}\left(1-\frac{i}{m}\right)}$. Thus, one should expect
\begin{align*}
    s_i \approx n\left(\frac{\lambda}{n}\right)^{\frac{d^i-1}{d-1}} \pm \Theta\left(n^{0.5 - \frac{\gamma}{2}\left(1-\frac{i}{m}\right)}\right) \quad \forall i \in [m].
\end{align*}
Note that the lower order term is equal to $\sqrt{n}$ for $i=m$ and $o(\sqrt{n})$ for $i < m$ which is consistent with the diffusion scaling observed in \cite{amarjit_power_of_d_sub_halfin_whitt}: $\Theta(\sqrt{n})$ for $i=m$ and $o(\sqrt{n})$ for $i < m$. While our bounds of Theorems \ref{theo: lower_bound} and \ref{theo: upper_bound} show concentration around the fixed point, we obtain a lower order term equal to $d^{i-1} \sqrt{n} \approx n^{0.5 + (i-1)\gamma/m}$ which is larger than that of $n^{0.5 - \frac{\gamma}{2}\left(1-\frac{i}{m}\right)}$. Note that \cite{amarjit_power_of_d_sub_halfin_whitt} obtains the correct scaling of the lower order terms for $i=m$ as $n \rightarrow \infty$, so characterizing the pre-limit lower-order scalings for $i \in [m]$ is an interesting future direction.

\section{Zero Waiting Time for \texorpdfstring{$m=1$}{m=1}} \label{app: zero_waiting}
Let $\mathcal{W}$ be the event where an arrival is routed to a queue with non-zero queue length when the system is in the steady-state. By the PASTA property, the queue length at the time of a customer arrival is $\bbars$. Thus, we have $\P{\mathcal{W} | \bbars} = \left(\frac{\bars_1}{n}\right)^d$ when then implies
\begin{align*}
    \P{\mathcal{W}} ={}& \E{\P{\mathcal{W}| \bbars}} = \E{\left(\frac{\bars_1}{n}\right)^d} \\
    \leq{}& \left(1 - \frac{2 \log d}{d} + \frac{19\log n}{\sqrt{n}} + 49\frac{ \log(d)^2}{d^2}\right)^d+  \P{\bars_1 > n - \frac{2 n\log d}{d} + 19 \sqrt{n}\log n + 49\frac{n\log(d)^2}{d^2}} \\
    \overset{(a)}{\leq}{}& \left(1 - \frac{2 \log d}{d} + \frac{19\log n}{\sqrt{n}} + 49\frac{ \log(d)^2}{d^2}\right)^d  + \left(\frac{1}{n}\right)^{(\log n)/9} \\
    \overset{(b)}{\leq}{}& \frac{2}{d^{2}} + \left(\frac{1}{n}\right)^{(\log n)/9} \rightarrow 0 \textit{ as } n \rightarrow \infty,
\end{align*}
where $(a)$ follows by Theorem~\ref{theo: upper_bound} and $(b)$ follows by Lemma~\ref{lemma: power_of_d_going_to_zero} for $n$ large enough. This completes the proof.

\end{APPENDICES}

\end{document}